\documentclass[reqno,11pt,letterpaper]{amsart}

\usepackage[marginratio=1:1, total={6in, 8.5in}]{geometry}

\usepackage[utf8]{inputenc}

\usepackage{amsmath}
\usepackage{amssymb}
\usepackage{amsthm}
\usepackage{mathtools}
\usepackage{graphicx}
\usepackage{comment}
\usepackage[title]{appendix}
\usepackage{hyperref}
\usepackage{csquotes}
\usepackage{enumitem}

\newtheorem{thm}{Theorem}[section]
\newtheorem{prop}[thm]{Proposition}
\newtheorem{lemma}[thm]{Lemma}

\theoremstyle{definition}
\newtheorem{defn}[thm]{Definition}
\theoremstyle{remark}
\newtheorem{rmk}[thm]{Remark}

\newcommand{\R}{\mathbb{R}}
\newcommand{\C}{\mathbb{C}}
\newcommand{\N}{\mathbb{N}}

\newcommand{\Sph}{\mathbb{S}}
\newcommand{\D}{\mathcal{D}'}

\newcommand{\brac}[1]{\langle #1 \rangle}
\newcommand{\pair}[2]{\bigl\langle #1,\, #2 \bigr\rangle}
\newcommand{\WF}{\operatorname{WF}'}
\newcommand{\rkerr}{\mathfrak{r}}
\newcommand{\bdf}{\rho}
\newcommand{\scf}{\mathrm{scf}}
\newcommand{\tf}{\mathrm{tf}}
\newcommand{\zf}{\mathrm{zf}}
\newcommand{\scb}{\text{sc-b}}
\DeclareMathOperator{\supp}{supp}
\DeclareMathOperator{\Ell}{Ell}
\DeclareMathOperator{\Char}{Char}

\usepackage[backend=biber, style=alphabetic, sorting=nyt, giveninits=true, isbn=false, url=false, doi=false, date=year, maxalphanames=5, maxbibnames=99]{biblatex}
\title{Quasinormal modes for the Kerr black hole}

\author{Thomas Stucker}
\address{Department of Mathematics, ETH Zurich, R\"amistrasse 101, 8092, Z\"urich, Switzerland}
\email{thomas.stucker@math.ethz.ch}

\begin{document}

\begin{abstract}
    We provide a rigorous definition of quasinormal modes for the Kerr black hole. They are obtained as the discrete set of poles of the meromorphically continued cutoff resolvent. The construction combines the method of complex scaling near asymptotically flat infinity with microlocal methods near the black hole horizon. We study the distribution of quasinormal modes in both the high and low energy regimes. We establish the existence of a high energy spectral gap and exclude the accumulation of quasinormal modes at zero energy.
\end{abstract}

\maketitle

\tableofcontents

\section{Introduction}
\label{section_introduction}

\subsection{Quasinormal modes on black hole spacetimes}

The study of quasinormal modes (QNMs) on black hole spacetimes has a long history in the physics literature, see the review articles \cite{kokkotas_physicsQNM,berti_cardoso} and references therein. QNMs are exponentially damped oscillating solutions to linear wave equations on a stationary black hole background. They can be described by a discrete set of complex frequencies (the negative imaginary part corresponding to the rate of decay), which are characteristic of the specific black hole spacetime. This should be compared to the purely oscillatory normal modes describing waves propagating on a compact domain. In contrast, our system is dissipative -- energy can escape to infinity or into the black hole -- and QNMs fail to form a complete set. Nonetheless, it is expected that a large portion of the gravitational radiation emitted by a perturbed black hole, after an initial response depending strongly on the nature of the perturbation, takes the form of a superposition of QNMs.
In the words of Chandrasekhar:
\begin{displayquote}[\cite{chandrasekhar_book}]
    \dots we may expect on general grounds that any initial perturbation will, during its last stages, decay in a manner characteristic of the black hole and independently of the original cause. In other words, we may expect that during the very last stages, the black hole will emit gravitational waves with frequencies and rates of damping, characteristic of itself, in the manner of a bell sounding its last dying pure notes. 
\end{displayquote}

The first mathematically rigorous investigation of QNMs, using tools from scattering theory, was conducted by Bachelot–Motet-Bachelot \cite{bachelot} and Sá Barreto-Zworski \cite{saBarreto_zworski} for spherically symmetric black holes. QNMs were defined as scattering resonances, that is, poles of the meromorphic continuation of a resolvent operator. This approach also proved effective for the study of black hole QNMs in the context of de Sitter space (positive cosmological constant), see \cite{dyatlov_QNM_KdS,dyatlov_QNM_quantization,vasy_KdS} for the Kerr-de Sitter case. In the de Sitter setting, ringdown can be described by means of a resonance expansion. That is, the late-time behavior of solutions to the wave equation is characterized by a finite superposition of QNMs, up to an error of exponentially faster decay. 

In contrast, solutions to wave equations on asymptotically flat spacetimes tend to exhibit polynomial tails \cite{tataru_price,hintz_pricelaw,gajic_priceLawKerr} and the meaning of QNMs in this context is not entirely clear. In particular, there is, as of yet, no rigorous definition of QNMs for the Kerr black hole. In this paper, we intend to remedy this situation by providing a mathematically robust construction of quasinormal modes on subextremal Kerr spacetimes and studying their distribution in both the high and low energy limits.

The concrete subject of our investigation is the scalar wave equation on the background of a Kerr black hole. Kerr spacetime \cite{kerr} is a family of stationary, asymptotically flat 4-dimensional Lorentzian manifolds $(M,g_{m,a})$ parameterized by $m>0$ and $a \in [-m,m]$, which are solutions to the vacuum Einstein equation with vanishing cosmological constant, i.e. $\mathrm{Ric}(g_{m,a})=0$. They describe a rotating black hole with mass $m$ and specific angular momentum $a$. In the region exterior to the black hole $\R_t\times (r_+,\infty)_r\times\Sph^2_{\theta,\varphi}$, the metric is traditionally written in Boyer-Lindquist coordinates $(t,r,\theta,\varphi)$, see Section \ref{section_kerr_metric} for the explicit form. Here, we only consider the subextremal range of angular momentum ${a \in (-m,m)}$. Our methods, specifically the radial point estimates over the horizon, break down in the extremal case.

Scalar waves on Kerr spacetime are described by the linear wave equation
$$\Box_g v = 0,$$
where $\Box_g$ is the Laplace-Beltrami operator for the Kerr metric. This forms a well-posed Cauchy problem for initial data specified on the spatial slice $\{t=0\}$, e.g.
$$v|_{t=0} = f_0, \quad \partial_tv|_{t=0} = f_1, \quad\mathrm{with}\quad f_0, f_1 \in C_c^\infty\bigl((r_+,\infty)\times \Sph^2\bigr),$$
where we take smooth data with compact support in the exterior region for simplicity. The late-time behavior of solutions to this Cauchy problem is now well understood. They are governed by inverse polynomial decay rates, a result known as Price's law. Specifically, for fixed $x$ in the exterior region, scalar waves on Kerr satisfy the pointwise asymptotics
$$v(t,x) \sim ct^{-3} \quad\text{as}\quad t\to \infty$$
where the constant $c$ can be explicitly calculated from the initial data. This was first conjectured by Price \cite{price1,price2} and rigorously proved in \cite{tataru_price,hintz_pricelaw,gajic_priceLawKerr}. Such polynomial tails seem to be a generic feature of wave equations on asymptotically flat spacetimes, see for instance \cite{gajic_priceLaw,ma_zhang_priceLaw,millet_priceLaw,luk_oh_priceLaw}.

While polynomial tails have been observed in numerics, see e.g. \cite{polynomial_tails_numerics}, at earlier time-scales numerical solutions of wave equations on Kerr spacetime are seen to be dominated by quasinormal modes, i.e. exponentially decaying, oscillating behavior \cite{berti_cardoso,dorband_numerics,buonanno_numerics}. QNMs were first numerically observed by Vishveshwara \cite{vishveshwara} (for Gaussian initial data) and have since become a topic of intense study in the numerical relativity community, see for example the references above, with ever more accurate calculations of their frequencies in the complex plane. Moreover, the presence of QNMs in the decay of gravitational perturbations can be said to be experimentally confirmed by gravitational wave astronomy \cite{ligo2016}. As shown in \cite{ligo2021,testing_noHair}, the gravitational wave signal from the ringdown phase of a black hole merger can be accurately modeled by a superposition of QNMs.

This suggests that the behavior of waves on asymptotically flat black hole spacetimes should be divided into three regimes: a prompt response depending strongly on the initial data, an intermediate phase where quasinormal modes dominate, and a very late time regime where the inverse polynomial decay kicks in. However, it remains an open problem to reconcile the rigorously established polynomial tails with the observed exponentially damped resonant behavior. In particular, one may pose the following questions:
\begin{itemize}[leftmargin=6mm]
    \item At what time-scale are quasinormal modes overtaken by inverse polynomial decay?
    \item How does the presence of these contrasting regimes depend on initial conditions?
\end{itemize}
A preliminary result in this direction was obtained by Dyatlov in \cite{dyatlov_ringdown}. He showed that for initial data localized at very high frequencies $\sim \lambda$, that is, in the ${\lambda \to \infty}$ limit, exponential decay persists up to time-scales of order $\log(\lambda)$. However, as the aforementioned numerical simulations are generally run with initial data that cannot be characterized as high frequency (e.g. Gaussian data), this answer remains unsatisfactory.
In order to address these questions, a better understanding of QNMs in the asymptotically flat setting is necessary.

In the physics literature, quasinormal modes are defined as solutions of the form
$$\Box_g\bigl(e^{-i\sigma t}u(r,\theta,\varphi)\bigr) = 0, \quad\text{for}\quad \sigma \in \C,$$
where the function $u$ on the spatial slice satisfies outgoing boundary conditions at infinity and ingoing boundary conditions at the horizon. This is usually stated in terms of the tortoise coordinate $r_*$ with $r_* \sim \frac{\alpha}{2}\log(r-r_+)$ at the horizon for a constant $\alpha$ depending on the black hole parameters. Then $u$ is supposed to satisfy the asymptotics
$$u \sim e^{i\sigma r_*} \quad\text{as}\quad r_*\to\infty, \qquad u \sim e^{-i\sigma r_*} \quad\text{as}\quad r_*\to-\infty\,\,\, (r\searrow r_+).$$
This perspective on QNMs goes back to Chandrasekhar and Detweiler \cite{chandrasekhar_detweiler}. Note that the boundary conditions make physical sense, since one would expect a compact perturbation to create waves travelling out towards infinity and in towards the horizon. However, the precise meaning of the outgoing condition at spatial infinity is unclear. For $\Im(\sigma) < 0$, which is exactly the frequency range where QNMs occur, one is trying to enforce exponentially growing outgoing behavior while excluding the exponentially decaying ingoing modes. In fact, since $e^{-i\sigma r_*} = e^{i\sigma r_*}e^{-2i\sigma r_*}$
with $e^{-2i\sigma r_*}$ a smooth function of $r_*^{-1}$ near $0$ for $\Im(\sigma) < 0$, there is a certain ambiguity in the boundary condition.

At the black hole horizon, this issue can be resolved by characterizing the ingoing behavior in terms of regularity. To this end, one should work in coordinates $(t_*,r,\theta,\varphi_*)$ that extend beyond the (future) event horizon, see Section \ref{section_kerr_metric}. Note that $t_* \sim t + r_*$ at the event horizon and the hypersurfaces of constant $t_*$ intersect the future horizon transversally. In these coordinates ingoing modes are smooth across the event horizon, whereas outgoing modes behave as $u \sim (r-r_+)^{i\alpha\sigma}$ near $r=r_+$. Since this fails to lie in Sobolev spaces $H^s_\mathrm{loc}$ for $s > \frac{1}{2} - \alpha\Im(\sigma)$, one can enforce the boundary condition at the horizon by working on Sobolev spaces of high enough regularity. The fact that these asymptotics actually describe solutions to the Kerr wave equation can be heuristically justified by separating into spherical harmonics and studying the resulting ODE, which has a regular singular point at $r=r_+$ but an irregular singular point at $r=\infty$.

These considerations are made rigorous by the Fredholm framework of Vasy \cite{vasy_KdS}. QNMs are constructed as the poles of the meromorphically continued resolvent operator. Thus, one considers the spectral family associated to the stationary metric $g$, i.e. the Fourier transformed wave operator
$$P(\sigma) = e^{i\sigma t_*}\Box_g e^{-i\sigma t_*}.$$
The crucial step is to find function spaces $\mathcal{X},\mathcal{Y}$ such that $P(\sigma): \mathcal{X} \to \mathcal{Y}$ defines a Fredholm operator for all $\sigma$ in some domain $D \subset \C$. 
This is achieved via Fredholm estimates for $P(\sigma)$ and its formal adjoint, i.e. by estimating
$$\|u\|_\mathcal{X} \leq C\bigl(\|P(\sigma)u\|_\mathcal{Y} + \|u\|_\mathcal{Z}\bigr),$$
where the inclusion $\mathcal{X} \hookrightarrow \mathcal{Z}$ is compact, and similarly for $P(\sigma)^*$. If one can additionally show the invertibility of $P(\sigma)$ when $\Im(\sigma)$ is large and positive, a fact that follows for instance from simple energy estimates, then the analytic Fredholm theorem implies that the resolvent $P(\sigma)^{-1}: \mathcal{Y} \to \mathcal{X}$ extends to all of $\sigma \in D$ as a meromorphic family of operators. The quasinormal modes in $D$ are then identified with the finite rank poles of $P(\sigma)^{-1}$. The guiding principle in finding the right function spaces is that $\mathcal{X}$ should admit the desired ingoing/outgoing behavior, while excluding the objectionable asymptotics.

The discussion above then suggests that, near the horizon, Sobolev spaces of sufficient regularity can be used to characterize QNMs as resonances. In the de Sitter setting, where asymptotically flat infinity is replaced by a cosmological horizon, such a characterization of quasinormal modes has been successfully implemented. In particular, QNMs have been rigorously constructed for the Kerr-de Sitter black hole. Moreover, resonance expansions have been established for the wave equation on Kerr-de Sitter spacetime. These take the form
$$v(t_*,x) = \sum_j e^{-i\sigma_j t_*}v_j(x) + \mathcal{O}(e^{-\gamma t_*}), \quad\text{as}\quad t_*\to\infty$$
for some $\gamma > 0$, where the sum is over the set of QNMs with $\Im(\sigma) > -\gamma$. In \cite{saBarreto_zworski} Sá Barreto and Zworski constructed QNMs for the Schwarzschild-de Sitter black hole and showed that these approach a lattice in the high energy, $|\Re(\sigma)| \to \infty$, limit. A resonance expansion for Schwarzschild-de Sitter was then established by Bony and Häfner \cite{bony_haefner_res_exp}. In a series of papers \cite{dyatlov_QNM_KdS,dyatlov_QNM_KdS_2,dyatlov_QNM_quantization}, Dyatlov constructed QNMs for Kerr-de Sitter black holes, gave a description of high energy resonances and proved a resonance expansion in the slowly rotating case. This was extended by Vasy \cite{vasy_KdS} to a large range of angular momenta, $|a| < \frac{\sqrt{3}}{2}m$. Vasy also provided a general framework for proving Fredholm estimates near the event horizon using microlocal methods. Finally, in \cite{petersen_vasy_full_subextremal} resonance expansions were established for the full subextremal range of Kerr-de Sitter black holes.

Because of the presence of asymptotically flat infinity, these methods are not immediately applicable to the Kerr black hole. In the asymptotically flat setting, the aforementioned guiding principle for proving Fredholm estimates suggests that for $\Im(\sigma)<0$ one should work with function spaces that somehow exclude the exponentially decaying ingoing behavior while including the exponentially growing outgoing behavior at infinity. We will achieve this through a trick known as complex scaling, see the discussion below. We note that Gajic and Warnick have developed an alternative characterization of quasinormal modes for asymptotically flat spacetimes as the set of eigenvalues of an appropriate evolution semigroup. This is accomplished by working with spaces based on Gevrey regularity. These methods were applied in \cite{gajic_warnick_QNM_Reissner_Nordstrom} to construct QNMs on extremal Reissner-Nordström spacetime and have now been extended to the case of Kerr spacetimes \cite{gajic_warnick_QNM_Kerr}, see Remark \ref{rmk_gajic_warnick} for a brief comparison of results.

\subsection{Meromorphic continuation of the cutoff resolvent}

We will study the Kerr spectral family $P(\sigma)$ on the spatial slice $X = (r_0,\infty)\times\Sph^2$, where the boundary at $r=r_0$ is located inside the horizon. In Section \ref{section_fredholm}, we prove the meromorphic continuation of the cutoff resolvent for the Kerr spectral family to a subset of the logarithmic cover. Note that the cutoff is necessary to avoid the expected exponentially growing behavior produced by the resolvent. QNMs can then be characterized as the discrete set of finite-rank poles of this meromorphic family of operators.
\begin{thm}
\label{thm_analytic_continuation}
    The cutoff resolvent for the Kerr spectral family has a meromorphic continuation to the first sheets of the logarithmic cover $\Lambda \to \C\setminus\{0\}$ of the complex plane. More precisely, for any $s\in\R$ and any $\chi \in \Bar{C}^\infty(X)$ with compact support in $\Bar{X}$ and satisfying $\chi=1$ on some ball $B_R$, the operator
    $$\chi P(\sigma)^{-1} \chi: \Bar{H}^{s-1}(X) \to \Bar{H}^{s}(X)$$
    extends from $\Im(\sigma) \gg 0$ to
    $$\Bigl\{\sigma\in\Lambda \,\,|\,\, \arg(\sigma) \in (-\pi,2\pi), \, \Im(\sigma) > \tfrac{1}{\alpha}\bigl(\tfrac{1}{2} - s\Bigr) \bigr\}$$
    as a meromorphic family of bounded operators with poles of finite rank. Here, by the imaginary part of $\sigma$ we mean $\Im(\sigma) = \sin(\arg(\sigma))|\sigma|$ and $\alpha$ is a constant depending on the black hole parameters, specifically $\alpha = 2(m + \frac{m^2}{\sqrt{m^2-a^2}}).$

    The poles of the cutoff resolvent, for various $s$, form a discrete subset of 
    $$\bigl\{\sigma\in\Lambda \,\,|\,\, \arg(\sigma) \in (-\pi,2\pi)\bigr\},$$
    which is independent of the cutoff function $\chi$, for $R$ large enough, and we define the quasinormal modes of the Kerr black hole as this discrete set of poles.
\end{thm}

\begin{rmk}
    The Kerr resolvent exhibits a logarithmic singularity at $\sigma = 0$, see \cite{hintz_pricelaw}. Thus, the Riemann surface of the logarithm, i.e. the universal cover $\Lambda \to \C\setminus\{0\}$, is the natural domain for the meromorphic continuation in Theorem \ref{thm_analytic_continuation}. This can be compared to the case of potential scattering in even dimensions, see \cite{dyatlov_zworski_scattering_book}. We expect that a meromorphic continuation to the full logarithmic cover is possible, but this would require more sophisticated methods. In any case, the meaning of resonances on more distant sheets of the logarithmic cover is not entirely clear. Note that Theorem \ref{thm_analytic_continuation} in particular establishes the extension of $\chi P(\sigma)^{-1} \chi$ to $\arg(\sigma) \in (-\frac{\pi}{2},\frac{3\pi}{2})$ which can be identified with $\sigma \in \C \setminus i\R_{\leq 0}$. However, the range of $\sigma$ in Theorem \ref{thm_analytic_continuation} goes slightly beyond this. Moving clockwise around the singularity at $\sigma=0$, we reach the negative imaginary axis at $\arg(\sigma)=-\frac{\pi}{2}$ and can extend the cutoff resolvent further all the way to $\arg(\sigma)>-\pi$. In the counterclockwise direction, we reach the negative imaginary axis at $\sigma=\frac{3\pi}{2}$ and can extend further to $\arg(\sigma)<2\pi$. Thus, under the projection $\Lambda \to \C\setminus\{0\}$ the lower half-plane is covered twice by our domain, see Figure \ref{fig_log_cover}. For the purpose of a potential resonance expansion (with an additional term encoding the polynomial tail), the range $\arg(\sigma) \in [-\frac{\pi}{2},\frac{3\pi}{2}]$ seems most relevant.
\end{rmk}

\begin{rmk}
\label{rmk_gajic_warnick}
    The approach of Gajic and Warnick to the construction of QNMs \cite{gajic_warnick_QNM_Reissner_Nordstrom,gajic_warnick_QNM_Kerr} differs significantly from the one used here. An advantage of our approach is the ability to access QNMs in the relatively large domain of Theorem \ref{thm_analytic_continuation}, whereas Gajic-Warnick can only access QNMs in the sector $-\frac{1}{6}\pi < \arg(\sigma) < \frac{7}{6}\pi$. A disadvantage is the necessity to use a cutoff function, which leads to a loss of information on the behavior of resonant states, i.e. mode solutions, near infinity. Thus, in a potential modified (i.e. with polynomial tail) resonance expansion for the Kerr wave equation, we would be restricted to initial data supported away from infinity. In contrast, the methods of Gajic-Warnick allow for a description of quasinormal mode solutions all the way to infinity. Note however that we only cutoff near infinity and not near the horizon, so the behavior of resonant states at the horizon is captured by our cutoff resolvent.
\end{rmk}

In practice, we will not obtain QNMs directly from the cutoff resolvent, but rather via a family of deformed operators $P_\beta(\sigma)$ for $\beta \in (-\pi,\pi)$. For $\sigma$ in a $\beta$-dependent tilted half plane $\Lambda_\beta$, the poles of the cutoff resolvent coincide with the values of $\sigma$ where $P_\beta(\sigma)$ has non-trivial kernel. Increasing the scaling angle $\beta$ then allows one to uncover resonances ever farther into the negative half-plane. This leads to an alternative definition of QNMs, see Definition \ref{def_qnm}. The deformed operator $P_\beta(\sigma)$ is obtained from $P(\sigma)$ by complex scaling, see Section \ref{section_complex_scaling}.

The method of complex scaling goes back to Aguilar-Combes \cite{aguilar_combes}, Balslev-Combes \cite{balslev_combes} and was further developed by Sj\"ostrand-Zworski \cite{sjostrand_zworski_complex_scaling}. The idea of complex scaling is to deform the original spatial slice $X \subset \R^3$ to a real submanifold $X_\beta \subset \C^3$ and replace the operator $P(\sigma)$ by a complex scaled operator $P_\beta(\sigma)$ on $X_\beta$ with better properties in the $|x|\to\infty$ limit. The deformation takes place far from the horizon in the region where the Kerr spectral family is elliptic, and hence does not alter the behavior of $P(\sigma)$ near the horizon. In fact, for some large $R_1 < R_2$, the deformed space $X_\beta$ agrees with $X$ in $\{|x| < R_1\}$, but near infinity it is scaled by the angle $\beta$ into the complex domain, that is, $X_\beta$ agrees with $e^{i\beta}\R^3$ in $\{|x| > R_2\}$. Thus, on $X_\beta$ outgoing behavior is described by the asymptotics $\sim e^{i\sigma e^{i\beta} r}$, which is exponentially decaying for $-\beta < \arg(\sigma) < \pi-\beta$.

The method relies on the analyticity of the Kerr spectral family. In Section \ref{section_scaled_operator} we show that $P(\sigma)$ extends outside some ball $B_{R_0}$ to a differential operator with analytic coefficients on a complex domain $U \subset \C^3$. The complex scaled operator is then obtained by analytic continuation of $P(\sigma)$ to the deformed space $X_\beta$. Crucial for the success of this method is the deformation result of Sjöstrand-Zworski \cite{sjostrand_zworski_complex_scaling}, see Lemma \ref{deformation_lemma}, which allows solutions $u \in C^\infty(X)$ of $P(\sigma)u = 0$ to be analytically continued to solutions $u_\beta \in C^\infty(X_\beta)$ of $P_\beta(\sigma)u_\beta = 0$. Thus, for $\arg(\sigma) \in (-\beta,\pi-\beta)$ solutions with outgoing asymptotics can be characterized by square-integrability on the deformed space $X_\beta$.

We will work with Sobolev spaces $\Bar{H}^s(X_\beta)$ on $X_\beta$ (whose elements extend beyond the boundary of $X_\beta$ at $r=r_0$ chosen arbitrarily inside the horizon). The main ingredient in the proof of Theorem \ref{thm_analytic_continuation} is Proposition \ref{fredholm_prop}, where we show that the complex scaled operators define analytic families of Fredholm operators
$$P_\beta(\sigma): \mathcal{X}^s_\beta = \{u \in \Bar{H}^s(X_\beta) \,\,|\,\, P_\beta(0)u \in \Bar{H}^{s-1}(X_\beta)\} \to \Bar{H}^{s-1}(X_\beta)$$
for all $\sigma$ in the half-plane
$$\Lambda_\beta = \bigl\{\sigma \in \C\setminus\{0\} \,\,|\,\, \arg(\sigma) \in (-\beta,\pi-\beta)\bigr\}$$
satisfying $\Im(\sigma) > \frac{1}{\alpha}(\frac{1}{2}-s)$. This is achieved via the Fredholm estimates of Proposition \ref{prop_fredholm_estimates}. We first show in Section \ref{section_scattering_elliptic} that $P_\beta(\sigma)$ and its formal adjoint are elliptic in the region where the complex deformation was applied, and moreover scattering elliptic near infinity. Away from the complex scaling region, we use the microlocal methods of \cite{vasy_KdS}. Thus, in Section \ref{section_hamiltonian_flow} we study the Hamiltonian flow of the principal symbol on the characteristic set. The dynamics are very similar to the Kerr-de Sitter case studied in \cite{vasy_KdS}. In particular, over the black hole horizon there is a radial source and a radial sink for the flow. We can thus apply microlocal radial estimates to our operator $P_\beta(\sigma)$ and propagate these estimates throughout the characteristic set. Note that the radial estimate over the horizon is where the condition on the Sobolev regularity $s$ enters. Inside the horizon we close our Fredholm estimates by using the strict hyperbolicity of $P_\beta(\sigma)$ with respect to $r$.

Analytic Fredholm theory then gives the meromorphic continuation of the complex scaled resolvent $P_\beta(\sigma)^{-1}$ to the half-plane $\Lambda_\beta$. Poles occur when $P_\beta(\sigma)$ has non-trivial kernel on the Sobolev spaces $\Bar{H}^s(X_\beta)$. In order to relate this to the original Kerr spectral family, one needs to establish that the behavior is in some sense unaffected by the complex scaling procedure, which is addressed in Section \ref{section_def_QNM}. Thus, in Proposition \ref{prop_qnm_scaling_independent} we show that the dimension of the kernel of $P_\beta(\sigma)$ is actually independent of the scaling angle $\beta$. This allows for a characterization of QNMs as the poles of the complex scaled resolvent. Furthermore, in Proposition \ref{prop_cutoffs_agree} we show that the action of the complex scaled resolvents away from the region where the complex deformation was applied does not depend on $\beta$. This allows the $P_\beta(\sigma)^{-1}$ to be patched together to a meromorphic continuation of the cutoff resolvent as in Theorem \ref{thm_analytic_continuation}.

\subsection{Quasinormal modes in the high and low energy limits}

With a definition of Kerr quasinormal modes at hand, a natural question regards the distribution of QNMs in the complex domain. Note that the mode stability results of \cite{whiting,shlapentokh-rothman_mode_stability} imply the absence of resonances in the upper half-plane. The characterization of quasinormal modes in terms of the kernel of the complex scaled operator allows us to infer properties of the QNMs by studying the operator $P_\beta(\sigma)$. In particular, we will study the behavior of $P_\beta(\sigma)$ in the high energy ($|\Re(\sigma)| \to \infty$) and low energy ($|\sigma| \to 0$) regimes and establish certain uniform estimates in these limits. The high energy estimates imply the presence of a high energy spectral gap.

\begin{thm}
\label{thm_high_energy}
There exist $\gamma>0$ and $C>0$ such that there are no quasinormal modes in 
$$\bigl\{\sigma \in \C\setminus i\R_{\leq 0} \,\,\,|\,\,\, \Im(\sigma) > -\gamma,\,\, |\Re(\sigma)| > C\bigr\}.$$
\end{thm}

\begin{rmk}
    Notice that Theorem \ref{thm_high_energy} is stated for $\sigma \in \C\setminus i\R_{\leq 0}$, which is identified with the subset of the logarithmic cover determined by $\arg(\sigma) \in (-\frac{\pi}{2},\frac{3\pi}{2})$. The theorem establishes a resonance free region at high energy near the positive ($\arg(\sigma)=0$) or negative ($\arg(\sigma)=\pi$) real axis, see Figure \ref{fig_log_cover}. Under the projection $\Lambda \to \C\setminus\{0\}$, this region is covered a second time by the domain of Theorem \ref{thm_analytic_continuation}, i.e. for $\arg(\sigma)$ near $2\pi$ or $-\pi$. Theorem \ref{thm_high_energy} says nothing about resonances in this second ``high energy'' region further along the logarithmic cover.
\end{rmk}

We will prove Theorem \ref{thm_high_energy} in Section \ref{section_high_energy}. It is convenient to transform the high energy regime into a semiclassical problem with small parameter $h = |\sigma|^{-1}$. We take as our semiclassical operator $P_\hbar = h^2P_\beta(h^{-1}z)$ with $|z|=1$. The limit $|\Re(\sigma)| \to \infty$ with $\sigma$ confined to a strip of the form $|\Im(\sigma)|<\gamma$ then corresponds to the semiclassical limit $h\to 0$ with $\Im(z) < \gamma h$. The main ingredient in the proof of Theorem \ref{thm_high_energy} is then the semiclassical resolvent estimate of Proposition \ref{prop_semiclassical_estimate}.

The proof of this estimate is quite similar to the proof of a corresponding estimate in the Kerr-de Sitter case, see for instance \cite{vasy_KdS}. They follow by applying various semiclassical estimates of propagation-type on the semiclassical characteristic set of $P_\hbar$, and thus require a careful study of the Hamiltonian flow of the semiclassical principal symbol, which is related to the null-geodesic flow of the Kerr metric. The salient feature is the presence of trapped null-geodesics, which leads to a trapped set for the Hamiltonian flow. The trapped set on Kerr has been shown to be normally hyperbolic, first by Wunsch-Zworski \cite{wunsch_zworski_trapping} for small angular momentum and then by Dyatlov \cite{dyatlov_ringdown} in the full subextremal range. This allows the trapped set to be dealt with using the normally hyperbolic trapping estimates of \cite{hintz_vasy_quasilinear}, see also \cite{wunsch_zworski_trapping,dyatlov_trapping}. Note that the maximal possible value of $\gamma$ in Theorem \ref{thm_high_energy} is related to the minimal expansion rate in the normal directions at trapping.

Our use of complex scaling leads to another subtlety. The semiclassical principal symbol is now complex-valued in the region where the complex deformation takes place. In order to apply propagation of regularity, we must show that the imaginary part of the principal symbol has a definite sign on the characteristic set, as is done in Section \ref{section_definite_sign}. Estimates can then only be propagated to and from the complex scaling region in a definite direction along the Hamiltonian flow.

Finally, in Section \ref{section_low_energy} we study the low energy ($|\sigma|\to 0$) regime. Notice that the discreteness of the set of quasinormal modes on the logarithmic cover $\Lambda \to \C\setminus\{0\}$ does not a priori exclude the possibility that QNMs could accumulate at $\sigma = 0$. This possibility was already discussed by Sá Barreto and Zworski \cite{saBarreto_zworski} for the Schwarzschild black hole. Even in the Schwarzschild case the accumulation of QNMs at the origin has not previously been ruled out. We address this question by establishing a resonance free region around $\sigma=0$ in the logarithmic cover.

\begin{thm}
\label{thm_low_energy}
For each $\delta>0$, there exists $c>0$ such that no quasinormal modes are contained in the set
$$\bigl\{\sigma \in \Lambda \,\,|\,\, \arg(\sigma) \in [-\pi+\delta,2\pi-\delta],\,\, |\sigma| \leq c \bigr\}.$$
\end{thm}

\begin{rmk}
Note that Theorem \ref{thm_low_energy} in particular rules out accumulation of QNMs in the physically relevant region $\arg(\sigma) \in [-\frac{\pi}{2},\frac{3\pi}{2}]$, which under the projection $\Lambda \to \C\setminus\{0\}$ covers the entire punctured complex plane with the negative imaginary axis covered twice, see Figure \ref{fig_log_cover}. In fact, the non-accumulation result goes safely beyond the negative imaginary axis in both directions. However, as one moves further along the logarithmic cover, where the physical meaning of resonances is less clear, there could still be accumulation at the origin. In particular, Theorem \ref{thm_low_energy} does not exclude a sequence of resonances within the domain of Theorem \ref{thm_analytic_continuation} converging to $\sigma=0$ while $\arg(\sigma)$ decreases towards $-\pi$ or increases towards $2\pi$.
\end{rmk}

\begin{figure}[b]
    \centering
    \includegraphics{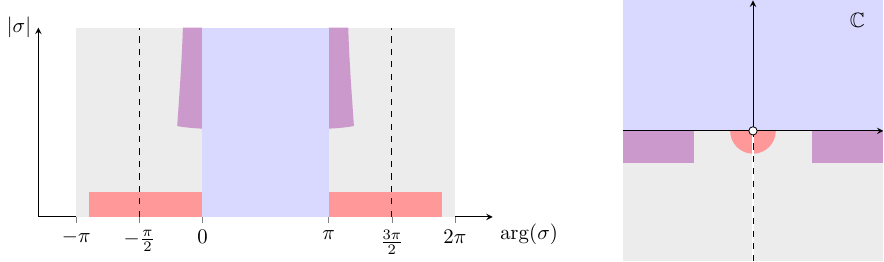}
    \caption{
    \textit{On the left:} the domain $(0,\infty)_{|\sigma|}\times(-\pi,2\pi)_{\arg(\sigma)}$ to which the cutoff resolvent continues meromorphically by Theorem \ref{thm_analytic_continuation}, as a subset of the logarithmic cover $\Lambda \simeq (0,\infty)_{|\sigma|}\times(-\infty,\infty)_{\arg(\sigma)}$. \textit{On the right:} the complex plane with a branch cut along the negative imaginary axis. Note that the cut complex plane is already covered by the subset $(0,\infty)_{|\sigma|}\times(-\frac{\pi}{2},\frac{3\pi}{2})_{\arg(\sigma)}$, indicated by the dashed lines on the left. Various resonance free regions are represented schematically in color. \textit{In blue:} the upper half-plane, where the resolvent itself is well-defined, and in fact analytic by the mode stability results of \cite{whiting,shlapentokh-rothman_mode_stability}. \textit{In violet:} the high energy region where resonances are excluded by Theorem \ref{thm_high_energy}. \textit{In red:} the low energy region where resonances are excluded by Theorem \ref{thm_low_energy}.
    }
    \label{fig_log_cover}
\end{figure}

Theorem \ref{thm_low_energy} will follow from the uniform low energy estimates of Proposition \ref{prop_low_energy_estimate}. The estimates are derived by viewing $P_\beta(\sigma)$ as an element of the scattering-b-transition calculus of \cite{guillarmou_hassell}, see also \cite{hintz_mode_stability}. This calculus captures the different behavior at infinity in the $|\sigma|>0$ versus $|\sigma|=0$ case. In fact, $P_\beta(\sigma)$ is well-behaved as a scattering operator for $|\sigma|>0$ with principal symbol at spatial infinity $e^{-2i\beta}|\xi|^2-|\sigma|^2e^{2i\arg(\sigma)}$ elliptic uniformly in $\arg(\sigma) \in [-\beta+\delta,\pi-\beta-\delta]$. The scattering calculus is however not the right venue for the zero energy operator and $P_\beta(0)$ should rather be viewed as a b-differential operator. Indeed, in Lemma \ref{lemma_zero_energy_estimate} we show that $P_\beta(0)$ has trivial kernel on certain weighted b-Sobolev spaces, where regularity is measured with respect to $r\partial_r$ and $\partial_\omega$ (with $\omega$ denoting coordinates on $\Sph^2$), and prove estimates for the zero energy operator on these spaces. This follows from a similar analysis of the Kerr zero energy operator without complex scaling performed in \cite{hintz_mode_stability}. The scattering-b-transition calculus in some sense patches together the b-calculus at zero frequency with the scattering calculus at non-zero frequencies. This is achieved by resolving the point $|\sigma|=r^{-1}=0$ through a blow-up. The behavior of $P_\beta(\sigma)$ in the limit $|\sigma|\to 0$, $r\to\infty$ is sensitive to the product $|\sigma|r$ and the blow up procedure can be thought of as using the rescaled coordinate $\tau = |\sigma|r$ all the way down to $\sigma=0$, $r=\infty$. The transition between zero and non-zero frequencies is governed by a model operator in $\tau$, which we study in Lemma \ref{lemma_tf_operator}. The uniform low energy estimates then hold on corresponding scattering-b-transition Sobolev spaces, where regularity is measured with respect to the frequency-dependent vector fields $\frac{r}{1+|\sigma|r}\partial_r, \frac{1}{1+|\sigma|r}\partial_\omega$.

\subsection{Outline of the paper}

\begin{itemize}
    \item In Section \ref{section_preliminaries} we discuss the method of complex scaling and review some microlocal estimates that are needed for the proof of Theorem \ref{thm_analytic_continuation}.
    \item In Section \ref{section_fredholm} we apply complex scaling to the Kerr spectral family and show that the resulting operator is Fredholm on appropriate spaces. The main ingredient are the Fredholm estimates of Proposition \ref{prop_fredholm_estimates}. In Section \ref{section_def_QNM} we then relate the complex scaled operator to the cutoff resolvent and prove Theorem \ref{thm_analytic_continuation}. 
    \item In Section \ref{section_high_energy} we derive uniform estimates in the $|\Re(\sigma)|\to\infty$ limit and prove Theorem \ref{thm_high_energy}. The main ingredient is Proposition \ref{prop_semiclassical_estimate}.
    \item In Section \ref{section_low_energy} we derive uniform estimates in the $|\sigma|\to 0$ limit and prove Theorem \ref{thm_low_energy}. The key result is Proposition \ref{prop_low_energy_estimate}.
\end{itemize}

\subsection*{Acknowledgements}
I am very grateful to my Ph.D. advisor Peter Hintz for suggesting the problem, for many invaluable discussions and for carefully reading parts of the manuscript. I would also like to thank Dejan Gajic for helpful discussions. I gratefully acknowledge the hospitality of the Erwin Schrödinger Institute in Vienna in July 2023 during the writing of this paper.

\section{Preliminaries}
\label{section_preliminaries}

In this section, we present some analytic preliminaries that are needed for the construction of quasinormal modes in Section \ref{section_fredholm}. We first treat the method of complex scaling and then discuss various microlocal estimates. We also introduce the Sobolev spaces of extendable distributions. A reader familiar with these tools may wish to skip this section. The high energy estimates of Section \ref{section_high_energy} will in addition require methods from semiclassical analysis, which are reviewed in Section \ref{section_semiclassical}. The low energy estimates of Section \ref{section_low_energy} will require the notion of scattering-b-transition pseudodifferential operators. These are introduced in Section \ref{section_scattering-b-transition}.

\subsection{Complex scaling}
\label{section_complex_scaling}
The method of complex scaling was introduced by Aguilar-Combes \cite{aguilar_combes} and Balslev-Combes \cite{balslev_combes}. It has found widespread use in numerical analysis, where it is sometimes called the method of perfectly matched layers. We follow the approach of Sj\"ostrand-Zworski \cite{sjostrand_zworski_complex_scaling}, see also \cite[Section 4.5]{dyatlov_zworski_scattering_book} and \cite[Section 7]{sjostrand_lectures_resonances}.

In section \ref{section_scaled_operator} we will show that the Kerr spectral family extends to a differential operator with analytic coefficients on an open set $U\subset\C^3$. The complex scaled operator is obtained by restricting this operator to a real submanifold of $U$. Here, we describe the method in more generality. Thus, let
\begin{equation}
\label{analytic_pde}
    P(z,\partial_z) = \sum_{|\alpha|\leq k} a_\alpha(z)\partial_z^\alpha
\end{equation}
be a differential operator with analytic coefficients in an open set $U \subset \C^n$, where $\partial_z$ denotes the complex differential. We wish to obtain a well-behaved restriction $P|_\Gamma$ of $P$ to a smooth real submanifold $\Gamma \subset U$, that is, the restriction should satisfy
$$P|_\Gamma u|_\Gamma = (Pu)|_\Gamma$$
for all $u$ analytic in a neighborhood of $\Gamma$. This requires certain properties of the submanifold $\Gamma$.

\begin{defn}
\label{def_totally_real}
    A smooth real submanifold $\Gamma\subset\C^n$ is called \textit{maximally totally real} if $\Gamma$ has (real) dimension $n$ and
    $$T_x\Gamma \cap iT_x\Gamma = \{0\}, \quad \forall x\in\Gamma,$$
    where we identify $T_x\Gamma$ with a real subspace of $\C^n$.
\end{defn}

The simplest example of a totally real submanifold is $\R^n\subset\C^n$. On the other hand, the definition excludes such examples as 
$$\Gamma = \{(x+iy,0) \in \C^2, (x,y)\in\R^2\}.$$
Note that in this example, it is impossible to restrict the operator $\partial_{z_2}$ to $\Gamma$ in such a way that $\partial_{z_2}|_\Gamma u|_\Gamma = (\partial_{z_2}u)|_\Gamma$ for all $u$ analytic in a neighborhood of $\Gamma$, take for instance $u = z_2$. On the other hand, we have the following lemma, see \cite[Lemma 4.30]{dyatlov_zworski_scattering_book}.

\begin{lemma}
    Let $P(z,\partial_z)$ be a differential operator with analytic coefficients on an open set $U\subset\C^n$ as in \eqref{analytic_pde}. Let $\Gamma\subset U$ be a maximally totally real submanifold. Then there is a unique differential operator on $\Gamma$
    $$P|_\Gamma: C^\infty(\Gamma) \to C^\infty(\Gamma),$$
    such that for all $u$ analytic in a neighborhood of $\Gamma$,
    $$P|_\Gamma(u|_\Gamma) = P(u)|_\Gamma.$$
\end{lemma}

The method of complex scaling consists in deforming $\R^n$ outside of a compact set to a family of maximally totally real submanifolds $\Gamma_\beta \subset \C^n$. The deformed submanifold $\Gamma_\beta$ will be given as the image of a smooth injective map $F_\beta: \Omega\subset\R^n \to \C^n$. In this case, the property of being maximally totally real can be characterized in terms of the differential of $F_\beta$.

Indeed, let $\Gamma = F(\Omega)$ for a smooth injective map $F:\Omega\to\C^n$. Identifying $T_x\Omega$ with $\R^n$ and $T_{F(x)}\Gamma$ with a real subspace of $\C^n$, we have $T_{F(x)}\Gamma = dF(x)(\R^n)$, where we view the differential at $x$ as a real linear map $dF(x): \R^n \to \C^n$. Extending $dF(x)$ by linearity to a complex linear map $dF(x): \C^n \to \C^n$, the condition in Definition \ref{def_totally_real} becomes
$$dF(x)(\R^n) \cap dF(x)(i\R^n) = \{0\}.$$
This is equivalent to the injectivity, and thus invertibility, of $dF(x)$ as a complex linear map. In other words, $F(\Gamma)\subset\C^n$ is maximally totally real if and only if $\det(dF(x))\neq 0$ for all $x\in\Omega$.

In the case $\Gamma=F(\Omega)$, we can also give a more explicit expression for the operator $P|_\Gamma$. Let $P(z,\partial_z)$ be given as in \eqref{analytic_pde}, with analytic coefficients in an open set $U\subset\C^n$, and let $\Gamma = F(\Omega) \subset U$. Then for $u\in C^\infty(\Gamma)$, we have
\begin{equation}
    P|_\Gamma u(F(x)) = \sum_{|\alpha|\leq k} a_\alpha(F(x))( ^\intercal dF(x)^{-1}\cdot\partial_x)^\alpha (u\circ F)(x).
\end{equation}
This provides the local coordinate expression for $P|_\Gamma$ in the coordinate chart given by $F^{-1}:\Gamma\to\Omega\subset\R^n$.

We now construct the family of maps $F_\beta$ used for complex scaling. 

\begin{defn}
\label{def_phase_function}
    Given $R_1 > 0$ and $\varepsilon > 0$, we choose $\psi \in C^\infty(\R)$, satisfying
    \begin{itemize}
        \item $\psi(t) \in [0,1], \quad \forall t\in\R$,
        \item $\psi'(t) \leq \varepsilon \quad \forall t\in\R$,
        \item $\psi = 0$ on $(-\infty,\log(R_1))$,
        \item $\psi = 1$ on $(\log(R_2),\infty)$,
    \end{itemize}
    for some $R_2 = R_2(\varepsilon, R_1)$. For $\beta \in (-\pi,\pi)$ we define the map $f_\beta: (0,\infty) \to \C$, by
    \begin{equation}
        f_\beta(r) = e^{i\phi_\beta(r)}r, \quad\text{where}\quad \phi_\beta(r) = \beta\psi(\log(r)).
    \end{equation}
    Finally, we define the complex scaling map $F_\beta$ by
    $$F_\beta: \R^n \to \C^n, \quad F_\beta(0) = 0, \quad F_\beta(x) = f_\beta(|x|)\frac{x}{|x|},\quad\text{for}\,\, x\neq0.$$
\end{defn}
Note that $f_\beta$ defined in this way is a smooth injective map into $\C$, depending smoothly on $\beta \in (-\pi,\pi)$. Furthermore, $f_\beta$ satisfies
\begin{itemize}
    \item $f_\beta(r) = r$ \,for\, $r < R_1$,
    \item $f_\beta(r) = e^{i\beta}r$ \,for\, $r > R_2$,
    \item $\arg(f_\beta(r)) \in [0,\beta]$, $\forall r$\, (respectively $[\beta,0]$ when $\beta<0$),
    \item $\partial_r f_\beta(r) \neq 0$, and $|\arg(\partial_r f_\beta(r)) - \arg(f_\beta(r))| < \pi\varepsilon$, $\forall r$.
\end{itemize}
The last property follows from
$$\partial_r f_\beta(r) = e^{i\phi_\beta(r)}\bigl(1 + ir\phi_\beta'(r)\bigr) = e^{i\phi_\beta(r)}\bigl(1 + i\beta \psi'(\log(r))\bigr).$$
This shows that $\arg(\partial_r f_\beta(r)) = \arg(f_\beta(r)) + \tan^{-1}(\beta \psi'(\log(r)))$ and since $\psi'\leq \varepsilon$, we have $|\tan^{-1}(\beta \psi'(\log(r)))| \leq |\beta|\varepsilon$.

The image of the map $F_\beta$ is a maximally totally real submanifold of $\C^n$, which we denote by $\Gamma_\beta$. Indeed, denoting $r=|x|$, the differential satisfies
$$dF_\beta(x)_{ij} = e^{i\phi_\beta(r)}\bigl(\delta_{ij} + i\frac{x_ix_j}{r}\phi_\beta'(r)\bigr).$$
Thus, for any nonzero $v\in\C^n$, we have
$$\Bar{v}\cdot dF_\beta(x) v = e^{i\phi_\beta(r)}\bigl(|v|^2 + i \frac{|x\cdot v|^2}{r}\phi_\beta'(r)\bigr) \neq 0.$$
This shows that $dF_\beta(x)$ is injective. Note that the space $\Gamma_\beta$ agrees with $\R^n$ in the ball $B_{R_1}$ and has all coordinates scaled by an angle $\beta$ into the complex plane outside of the ball $B_{R_2}$.

The fundamental lemma we need is a deformation result, which will allow us to analytically continue a solution $u_1 \in \Gamma_{\beta_1}$ of $P|_{\Gamma_{\beta_1}}u_1 = 0$ to a solution $u_2 \in \Gamma_{\beta_2}$ of $P|_{\Gamma_{\beta_2}}u_2 = 0$ for an elliptic operator with analytic coefficients $P$. See \cite[Lemma 3.1]{sjostrand_zworski_complex_scaling} for a proof.

\begin{lemma}
\label{deformation_lemma}
    Let $\Omega \subset \R^n$ be open and $F: [0,1]\times\Omega \to \C^n$ be a smooth proper map, such that $F(s,\,\cdot\,)$ is injective $\forall s\in[0,1]$ and $\det(d_x F(s,x)) \neq 0$, $\forall x\in\Omega,\, s\in[0,1]$.
    In addition, assume there exists a compact set $K\subset\Omega$, such that $F(s,x)=F(0,x)$ $\forall x\in\Omega\setminus K, s\in[0,1]$. Denote $\Gamma_s = F(\{s\}\times\Omega)$ and let $P(z,D_z)$ be a differential operator with analytic coefficients in some neighborhood $U$ of $\bigcup_{s\in[0,1]} \Gamma_s$, such that $P|_{\Gamma_s}$ is elliptic $\forall s\in[0,1]$. If $u_0 \in C^\infty(\Gamma_0)$ and $P|_{\Gamma_0}u_0$ extends to an analytic function on $U$, then $u_0$ extends to a, possibly multivalued, analytic function on a neighborhood of $\bigcup_{s\in[0,1]} \Gamma_s$. More precisely, for each $s\in[0,1]$ there is an analytic function $\Tilde{u}_s$ defined in a neighborhood $\Tilde{U}_s$ of $\Gamma_s$, such that $\Tilde{u}_0|_{\Gamma_0}=u_0$ and for some $\varepsilon>0$, independent of $s$, and all $|s_1-s_2|<\varepsilon$, $\Tilde{u}_{s_1}=\Tilde{u}_{s_2}$ on $\Tilde{U}_{s_1}\cap\Tilde{U}_{s_2}$.
\end{lemma}

\begin{rmk}
\label{remark_multivalued}
Note that the possibility of a multivalued analytic function is due to the fact that the contours $\Gamma_s$ could, for example, wrap around the origin, in such a way that $\Gamma_1$ has a non-trivial intersection with $\Gamma_0$. However, if we exclude such non-trivial intersection, that is, if the deformations $\Gamma_s$ satisfy the additional requirement
$$z \in \Gamma_{s_1}\cap\Gamma_{s_2} \quad\text{for}\quad s_1<s_2 \quad\implies\quad z \in \Gamma_{s_1}\cap\Gamma_{s} \quad\forall\,s_1<s<s_2,$$
then the function $u_0$ of Lemma \ref{deformation_lemma} extends to a well-defined analytic function on a neighborhood of $\bigcup_{s\in[0,1]}\Gamma_s$.
\end{rmk}

\subsection{Microlocal analysis}
\label{section_microlocal_estimates}

In this section, we provide a brief introduction to pseudodifferential operators and state some microlocal estimates, which will be used in the proof of the Fredholm property of the complex scaled Kerr spectral family in Section \ref{section_fredholm_estimates}. For a detailed overview of the theory of pseudodifferential operators, see for instance \cite[Chapter 18]{hörmander3} or \cite{hintz_microlocal_notes}.

The space of (uniform) symbols on $\R^n$ of order $m \in \R$, denoted by $S^m(\R^n,\R^n)$, consists of all smooth functions $a(x,\xi)\in C^\infty(\R^n\times\R^n)$, satisfying the following estimate for all $\alpha\in\N_0^n$, $\beta\in\N_0^n$ and some constants $C_{\alpha\beta}$:
\begin{equation}
\label{pseudo_estimate}
    |\partial_x^\alpha \partial_\xi^\beta a(x,\xi)| \leq C_{\alpha,\beta} \brac{\xi}^{m-|\beta|}, \quad \forall x\in\R^n,\, \xi\in\R^n,
\end{equation}
where $\brac{\xi} = (1+|\xi|^2)^\frac{1}{2}$ and $|\beta| = \sum_{k=1}^n\beta_k$. Note that the family of norms
\begin{equation}
\label{pseudo_seminorms}
    \|a\|_{m,k} = \sup_{(x,\xi)\in \R^{2n}} \max_{|\alpha|+|\beta| \leq k} \brac{\xi}^{|\beta|-m}|\partial_x^\alpha \partial_\xi^\beta a(x,\xi)|,
\end{equation}
defines the topology of a Fréchet space on $S^m(\R^n,\R^n)$.

Given a symbol $a\in S^m(\R^n,\R^n)$, we define its quantization
$$\mathrm{Op}(a): C_c^\infty(\R^n) \to C^\infty(\R^n),$$
by the formula
\begin{equation}
\label{quantization}
    \mathrm{Op}(a)u(x) = \frac{1}{(2\pi)^n}\int_{\R^n} e^{ix\cdot \xi} a(x,\xi) \hat{u}(\xi)\,d\xi,
\end{equation}
where $\hat{u}$ denotes the Fourier transform of $u$. The space of $m$-th order (uniform) pseudodifferential operators on $\R^n$, denoted $\Psi^m(\R^n)$, consists of all operators ${A: C_c^\infty(\R^n) \to C^\infty(\R^n)}$, which are obtained as the quantization of a symbol in $S^m(\R^n,\R^n)$, as in \eqref{quantization}. We define the space of residual operators as 
$$\Psi^{-\infty}(\R^n) = \bigcap_{m\in\R}\Psi^m(\R^n).$$

Elements of $\Psi^m(\R^n)$ extend to define bounded operators between Sobolev spaces. More precisely, for all $s\in\R$, $m\in\R$ and $A\in\Psi^m(\R^n)$
$$A: H^s(\R^n) \to H^{s-m}(\R^n)$$
is bounded. Furthermore, pseudodifferential operators form an algebra under composition:
$$\Psi^m(\R^n) \circ \Psi^{m'}(\R^n) \subset \Psi^{m+m'}(\R^n).$$

The principal symbol of a pseudodifferntial operator $A=\mathrm{Op}(a)$, where ${a\in S^m(\R^n,\R^n)}$, is the equivalence class $ \sigma_m(A) = [a] \in S^m(\R^n,\R^n) / S^{m-1}(\R^n,\R^n)$. The principal symbol defines an algebra homomorphism
$$\sigma_m: \Psi^m(\R^n) \to S^m(\R^n,\R^n) / S^{m-1}(\R^n,\R^n),$$ where the product on the latter space is just multiplication of (equivalence classes of) symbols. Futhermore, this map fits into a short exact sequence
\begin{equation*}
\label{short_exact_sequence}
    0 \rightarrow \Psi^{m-1}(\R^n) \rightarrow \Psi^m(\R^n) \xrightarrow[]{\sigma_m} S^m(\R^n,\R^n) / S^{m-1}(\R^n,\R^n) \rightarrow 0.
\end{equation*}
This implies in particular that the commutator of two operators $A \in \Psi^m(\R^n)$, $B \in \Psi^{m'}(\R^n)$, satisfies $[A,B] \in \Psi^{m+m'-1}(\R^n)$.

Another important notion is the wavefront set, denoted $\WF(A)$, of a pseudodifferential operator $A = \mathrm{Op}(a)\in \Psi^m(\R^n)$. This is the complement of the set where $A$ is microlocally a residual operator. More precisely, a point $(x_0,\xi_0) \in \R^n\times(\R^n\setminus\{0\})$ does not lie in $\WF(A)$ if and only if there exists a conic neighborhood $V$ of $(x_0,\xi_0)$ (i.e. $(x,\xi)\in V$ implies $(x,\lambda\xi) \in V$, $\forall \lambda>0$) such that for all $\alpha,\beta \in \N_0^n$ and all $N\in\N$, we have
$$|\partial_x^\alpha \partial_\xi^\beta a(x,\xi)| \leq C_{\alpha,\beta,N} \brac{\xi}^{-N}, \quad \forall (x,\xi)\in V.$$
Note that $\WF(A)$ is a closed conic subset of $\R^n\times(\R^n\setminus\{0\})$, away from which $A$ is microlocally smoothing.

The theory of pseudodifferential operators provides a natural setting in which to phrase elliptic regularity. Here, we formulate a version for uniformly elliptic operators on $\R^n$. Below, we will state a more microlocal version of elliptic regularity on manifolds.

We say that $P \in \Psi^m(\R^n)$ is uniformly elliptic on an open subset $U \subset \R^n$ if there exist constants $C, c > 0$, such that the principal symbol of $P$ satisfies
$$|\sigma_m(P)(x,\xi)| \geq C |\xi|^m, \quad \forall (x,\xi) \in U\times\R^n,\, |\xi|\geq c.$$
Note that this property is independent of the choice of representative in the equivalence class $\sigma_m(P)$.

\begin{prop}[Uniform Elliptic Estimate]
\label{uniform_elliptic_estimate}
    Let $P \in \Psi^m(\R^n)$ and let $U\in\R^n$ be an open set, such that $P$ is uniformly elliptic on $U$. Let further $\chi, \Tilde{\chi} \in C^\infty(\R^n)$ satisfy $\supp(\chi) \subset U$ and $\Tilde{\chi} = 1$ on $\supp(\chi)$. Then for any $s,N\in\R$, there exists $C>0$, such that $\forall u \in H^{-N}(\R^n)$ with $\Tilde{\chi}Pu \in H^{s-m}(\R^n)$, we have $\chi u \in H^s(\R^n)$, and the following estimate holds:
    $$\|\chi u\|_{H^s(\R^n)} \leq C\bigl(\|\Tilde{\chi} Pu\|_{H^{s-m}(\R^n)} + \|u\|_{H^{-N}(\R^n)}\bigr).$$
\end{prop}

The notion of pseudodifferential operators carries over to smooth manifolds. An operator $A:C^\infty_c(M) \to C^\infty(M)$ belongs to $\Psi^m(M)$, the space of $m$-th order pseudodifferential operators on a manifold $M$, if and only if its Schwartz kernel is smooth away from the diagonal in $M\times M$, and for every chart $\varphi: U \to V$, with $U\subset M$ and $V \subset \R^n$, and every cutoff function $\chi \in C^\infty_c(U)$, we have $(\varphi^{-1})^*\chi A\chi\varphi^* \in \Psi^m(\R^n)$.
With this definition, the space of residual operators $\Psi^{-\infty}(M) = \bigcap_{m\in\R}\Psi^m(M)$ consists precisely of those operators whose Schwartz kernel is smooth on $M\times M$.

\begin{rmk}
    Note that for pseudodifferential operators on $\R^n$, we required the estimate \eqref{pseudo_estimate} to be uniform with respect to $x$. On a non-compact manifold, without additional structure, no such coordinate-invariant notion of uniformity is available. Thus, in the above definition, we only require the local coordinate version of \eqref{pseudo_estimate} to hold on compact subsets of $M$. Viewing $\R^n$ as a manifold, the above definition specifies a larger class of operators than the initial uniform definition.
\end{rmk}

We will say that a an operator $A\in\Psi^m(M)$ is compactly supported if its Schwartz kernel $\mathcal{K}_A$ has compact support in $M\times M$. We will say that $A$ is properly supported if the sets
$$\supp(\mathcal{K}_A)\cap\pi_1^{-1}(K), \quad \supp(\mathcal{K}_A)\cap\pi_2^{-1}(K)$$
are compact for every compact $K\subset M$, where $\pi_1,\pi_2: M\times M \to M$ are the projection maps on the first and second factor. Note that, in particular, differential operators are properly supported. If $A\in\Psi^m(M)$ is properly supported, then
$$A: C_c^\infty(M) \to C_c^\infty(M), \quad A: C^\infty(M) \to C^\infty(M),$$
and properly supported pseudodifferential operators form an algebra under composition.

On a general non-compact manifold, no invariant definition of Sobolev spaces is available. However, we can define the local Sobolev spaces, $H^s_{\mathrm{loc}}(M)$, to consist of all $u\in\D(M)$, such that $(\varphi^{-1})^*\chi u \in H^s(\R^n)$ for any chart $\varphi: U\to V$ and any $\chi \in C^\infty_c(U)$. We further define $H^s_c(M)$ to consist of all compactly supported elements of $H^s_{\mathrm{loc}}(M)$. We then have the following mapping property for all properly supported $A \in \Psi^m(M)$:
$$A: H^s_c(M) \to H^s_c(M), \quad A: H^s_{\mathrm{loc}}(M) \to H^s_{\mathrm{loc}}(M).$$

Note that these are not normed spaces. However, for any compact set $K\subset M$, we can define a norm on the space $H^s_K(M) = \{u \in H^s_c(M) \,|\, \supp(u) \subset K\}$ by choosing an arbitrary finite cover of $K$ by coordinate charts $\varphi_j: U_j \to V_j$ and a partition of unity $\chi_j \in C^\infty_c(U_j)$ subordinate to this cover and setting
$$\|u\|_{H^s} = \sum_j \|(\varphi^{-1})^*\chi_j u\|_{H^s(\R^n)}.$$
The norms introduced in this way for different choices of charts and cutoff functions are all equivalent. In the estimates below, we will use the notation $\|\cdot\|_{H^s}$ to denote such a choice of norm, when all functions involved are supported in the same compact set. The constants in the estimates will of course depend on the choice of norm, but the form of the estimates will not. In this sense, for any properly supported $A \in \Psi^m(M)$ we have
$$\|Au\|_{H^{s-m}} \leq C \|u\|_{H^s}, \quad \forall u \in H^s_K(M).$$

On a manifold, the principal symbol is invariantly defined as an equivalence class of functions on the cotangent bundle. Indeed, denote by $S^m(T^*M)$ the space of all $a \in C^\infty(T^*M)$ such that the estimates \eqref{pseudo_estimate} hold on any compact set $K\subset U$ and any coordinate chart $\varphi: U \to V$ with $\xi$ denoting the coordinates induced by $\varphi$ on the fibers of the cotangent bundle. Then the principal symbol map defines an algebra morphism
$$\sigma_m: \Psi^m(M) \to S^m(T^*M) / S^{m-1}(T^*M)$$
and the corresponding version of the short exact sequence \eqref{short_exact_sequence} holds. Locally, a representative of $\sigma_m(A)$ is obtained by taking the principal symbol of ${(\varphi^{-1})^*\chi A\chi\varphi^* \in \Psi^m(\R^n)}$.

The wavefront set can also be patched together from local coordinates to define a closed conic set $\WF(A) \subset T^*M \setminus \{0\}$, where $\{0\}$ denotes the graph of the zero section in $T^*M$. Note that $\WF(A)=\emptyset$ implies $A\in\Psi^{-\infty}(M)$ and if, in addition, $A$ is properly supported, we have for all $s,N\in\R$, $K\subset M$ compact, and $u\in H_K^{-N}(M)$:
$$\|Au\|_{H^s} \leq C_{s,N,K}\|u\|_{H^{-N}}.$$
Furthermore, for properly supported pseudodifferential operators $A,B$, the wavefront set satisfies $\WF(AB)\subset \WF(A)\cap\WF(B)$.

An operator $A\in\Psi^m(M)$ is elliptic at a point $(x_0,\xi_0) \in T^*M\setminus\{0\}$, if there exists a conic neighborhood $V \subset T^*M$ of $(x_0,\xi_0)$ and constants $C,c >0$, such that in local coordinates
$$|\sigma_m(A)(x,\xi)| \geq C|\xi|^m, \quad \forall (x,\xi)\in V,\, |\xi|\geq c.$$
We denote the set of points at which $A$ is elliptic by $\Ell(A)$. Note that this defines an open conic subset of $T^*M\setminus\{0\}$. The complement of the elliptic set is the characteristic set, denoted $\Char(A)= (T^*M\setminus\{0\}) \setminus \Ell(A)$.

\begin{rmk}
\label{sphere_bundle}
    A different perspective is to regard the conic sets $\Ell(A)$ and $\WF(A)$ as subsets of the sphere bundle 
    $$S^*M = (T^*M\setminus\{0\}) / \R_+,$$
    where the action of $\R_+$ by dilations on the fibers of the cotangent bundle has been quotiented out. This can, for instance, be useful for compactness arguments. Furthermore, if the principal symbol of $A$ has a homogeneous representative $a$, then one could also view the principal symbol as an element $\Tilde{a}$ of $C^\infty(S^*M)$ by setting $\Tilde{a}(x,[\xi]) = a(x,\frac{\xi}{|\xi|})$.
\end{rmk}

A useful result is the existence of microlocal partitions of unity, see \cite[Lemma 6.10]{hintz_microlocal_notes}, stated in terms of the sphere bundle as in the preceding remark.
\begin{prop}[Microlocal Parition of Unity]
\label{microlocal_partition_of_unity}
    Let $V\subset S^*M$ be compact. Let $U_1,\dots,U_N$ be an open cover of $V$. Then there exist compactly supported operators $A_1,\dots,A_N \in \Psi^0(M)$ such that $\WF(A_j) \subset U_j$ for all $j$ and $\WF(Id - \sum_j A_j) \cap V = \emptyset$.
\end{prop}

We can now state a microlocal version of elliptic regularity valid on manifolds, see \cite[Proposition 6.31]{hintz_microlocal_notes}.
\begin{prop}[Microlocal Elliptic Estimate]
\label{elliptic_estimate}
    Let $P\in\Psi^m(M)$ be properly supported and $B, G \in \Psi^0(M)$ compactly supported, such that $\WF(B)\subset\Ell(P)\cap\Ell(G)$. Then for any $s,N\in\R$, there exists $C>0$ and $\chi\in C^\infty_c(M)$, such that $\forall u \in H^{-N}_{\mathrm{loc}}(M)$ with $GPu \in H^{s-m}_c(M)$, we have $Bu \in H^s_c(M)$, and the following estimate holds:
    $$\|Bu\|_{H^s} \leq C\bigl(\|GPu\|_{H^{s-m}} + \|\chi u\|_{H^{-N}}\bigr).$$
\end{prop}

On the characteristic set of an operator $P \in \Psi^m(M)$, where elliptic regularity is not available, it is sometimes possible to control regularity by studying the Hamiltonian flow of the principal symbol of $P$. To this end, let $\sigma_m(P)$ have a real-valued, homogeneous representative $p \in C^\infty(T^*M\setminus\{0\})$. 

The cotangent bundle is naturally equipped with a symplectic form, given in local coordinates $(x,\xi)$ by $\omega = \sum_j d\xi^j\wedge dx^j$. The Hamiltonian vector field 
$$H_p \in C^\infty(T^*M,T(T^*M))$$
associated to $p$ is defined to satisfy $\omega(H_p,X) = dp(X)$ for all vector fields $X$ on $T^*M$. In local coordinates it takes the form
$$H_p = \sum_j \Bigl(\frac{\partial p}{\partial\xi^j}\frac{\partial}{\partial x^j} - \frac{\partial p}{\partial x^j}\frac{\partial}{\partial\xi^j}\Bigr).$$
The Hamiltonian vector field appears in the principal symbol of commutators. Indeed, for $A \in \Psi^{m'}(M)$ we have
$$\sigma_{m+m'-1}\Bigl(\frac{1}{i}[P,A]\Bigr) = H_pa.$$

We denote by $\exp(tH_p)$ the Hamiltonian flow, that is, the flow generated by the vector field $H_p$. Since $H_pp = 0$, the level sets of $p$, in particular the characteristic set, are left invariant by the flow of $H_p$. The next result makes precise the notion that the regularity of solutions can be propagated inside the characteristic set along the Hamiltonian flow of $p$. See \cite[Theorem 8.7]{hintz_microlocal_notes} for a proof.

\begin{prop}[Propagation of Singularities]
\label{propagation_estimate}
    Let $P\in\Psi^m(M)$ be properly supported and have a real-valued homogeneous principal symbol. Let $B,E,G \in \Psi^0(M)$ be compactly supported such that $\WF(B),\WF(E) \subset \Ell(G)$. Assume moreover that for all ${(x,\xi) \in \WF(B)}$, there exists $T\geq 0$ such that
    \begin{equation}
    \label{propagation_condition}
        e^{-TH_p}(x,\xi)\in\Ell(E), \quad e^{tH_p}(x,\xi)\in\Ell(G), \,\,\forall t\in[-T,0].
    \end{equation}
    Then for any $s,N\in\R$, there exists $C>0$ and $\chi\in C^\infty_c(M)$, such that $\forall u \in H^{-N}_{\mathrm{loc}}(M)$ with $GPu \in H^{s-m+1}_c(M)$ and $Eu \in H^s_c(M)$, we have $Bu \in H^s_c(M)$, and the following estimate holds:
    $$\|Bu\|_{H^s} \leq C\bigl(\|GPu\|_{H^{s-m+1}} + \|Eu\|_{H^s} + \|\chi u\|_{H^{-N}}\bigr).$$
\end{prop}
\begin{rmk}
    In the statement of Proposition \ref{propagation_estimate}, regularity is propagated forward along the flow of $H_p$. It is also possible to propagate regularity backward along the Hamiltonian flow. Thus, the Proposition remains true if we replace \eqref{propagation_condition} by
    $$e^{TH_p}(x,\xi)\in\Ell(E), \quad e^{tH_p}(x,\xi)\in\Ell(G), \,\,\forall t\in[0,T].$$
\end{rmk}

\begin{rmk}
\label{homogeneous_hamiltonian}
    Note that the Hamiltonian vector field of the degree $m$ homogeneous principal symbol $p$ is homogeneous of degree $m-1$, in the sense that $M_\lambda^*H_p = \lambda^{m-1}H_p$ for $\lambda > 0$, where $M_\lambda^*$ is the pullback by the map $M_\lambda(x,\xi) = (x,\lambda\xi)$. Thus, the projection to the sphere bundle $S^*M$ of integral curves of $H_p$ through $(x,\xi)$ and $(x,\lambda\xi)$ agree up to reparametrization. In fact, the rescaled Hamiltonian vector field $|\xi|^{-(m-1)}H_p$ can be projected down to the sphere bundle to define a flow on $S^*M$. So the Hamiltonian flow of a homogeneous symbol can be viewed as living on $S^*M$. However, this perspective gives up information on the rate of expansion or contraction of the flow in the radial direction of the fibers.
\end{rmk}

Notice that the statement of Proposition \ref{propagation_estimate} becomes vacuous at radial points, that is, points $(x,\xi) \in T^*M$ where $H_p$ is parallel to the generator of dilations in the fiber. When viewed projected down to the sphere bundle, as in Remark \ref{homogeneous_hamiltonian}, these are fixed points of the Hamiltonian flow. More generally, if $\Lambda \subset T^*M$ is an invariant submanifold for the flow of $H_p$, the estimates of Proposition \ref{propagation_estimate} cannot be used to propogate regularity into or out of $\Lambda$. However, for certain kinds of invariant submanifolds, namely radial sources and radial sinks, propagation estimates of a somewhat different nature are available.

\begin{defn}[Radial Source, Radial Sink]
\label{def_source_sink}
    Let $\Lambda \subset \Char(P) \subset T^*M\setminus\{0\}$ be a smooth conic submanifold invariant under the flow of $H_p$, whose projection to $S^*M$ is compact. Let $\rho_r \in C^\infty(T^*M\setminus\{0\})$ be positive, homogeneous of degree $-1$ and elliptic in a neighborhood of $\Lambda$ (i.e. $\rho_r(x,\xi) \geq C|\xi|^{-1}$ there). Then $\Lambda$ is called a radial source (respectively sink) for $P$, if the following conditions hold:
    \begin{enumerate}
        \item
        $$\rho_r^{m-1}H_p\rho_r\bigr|_{\Lambda} = \alpha_r\rho_r, \quad (\text{respectively}\,\, \rho_r^{m-1}H_p\rho_r\bigr|_{\Lambda} = -\alpha_r\rho_r),$$
        where $\alpha_r\in C^\infty(\Lambda)$ is homogeneous of degree $0$ and satisfies $\alpha_r>0$.
        \item There exists a homogeneous degree $0$ function $\rho_t \in C^\infty(V)$, defined in a conic neighborhood $V$ of $\Lambda$, which is a quadratic defining function of $\Lambda$ within $\Char(P)$, in the sense that 
        $$\Lambda = \{(x,\xi) \in V \,\,|\,\, p(x,\xi)=0,\, \rho_t(x,\xi)=0\}$$
        and $\rho_t$ vanishes quadratically at $\Lambda$, such that in $V$ we have
        $$\rho_r^{m-1}H_p\rho_t \geq \alpha_t\rho_t + F_3, \quad (\text{respectively}\,\, \rho_r^{m-1}H_p\rho_t \leq -\alpha_t\rho_t + F_3),$$
        where $\alpha_t \in C^\infty(V)$ is homogeneous of degree $0$ and satisfies $\alpha_t>0$, and $F_3 \in C^\infty(V)$ vanishes at least cubically at $\Lambda$.
    \end{enumerate}
\end{defn}

For propagation estimates at a radial source or sink, we require that the principal symbol of $P-P^*$ is homogeneous. Note that we are still assuming $p = \sigma_m(P)$ is homogeneous and real-valued, which implies that $P-P^* \in \Psi^{m-1}(M)$. The radial estimates come in two different flavors, high regularity and low regularity estimates, depending on the Sobolev regularity of $u\in H^s$. For $s$ above a certain threshold, regularity can be propagated out of radial sources or sinks, and for $s$ below this threshold, regularity can be propagated into sources or sinks. The threshold regularity depends on the subprincipal symbol of $P$. Specifically, it is given in terms of a function $\alpha_s \in C^\infty(\Lambda)$, such that
\begin{equation}
\label{threshold_regularity}
\sigma_{m-1}\Bigl(\frac{1}{2i}(P-P^*)\Bigr)\Bigr|_\Lambda = \pm\alpha_s\alpha_r\rho_r^{-m+1}
\end{equation}
with $\alpha_r$, $\rho_r$ as in Definition \ref{def_source_sink}, where the plus sign corresponds to a radial source and the minus sign to a radial sink. We then have the following results, see \cite[Theorem 9.9]{hintz_microlocal_notes}.

\begin{prop}[High Regularity Radial Estimate]
\label{high_reg_radial_estimate}
    Let $P\in\Psi^m(M)$ be properly supported and have a real-valued homogeneous principal symbol. Let further $P-P^*$ have a homogeneous principal symbol. Assume that $\Lambda \subset T^*M\setminus\{0\}$ is a radial source or radial sink for $P$. Let $G\in\Psi^0(M)$ be compactly supported with $\Lambda \subset \Ell(G)$ and let $s'\in\R$ satisfy $s' > \frac{m-1}{2} + \alpha_s$ on $\Lambda$. Then for any $s,N\in\R$ with $s\geq s'$, there exists $C>0$, $\chi\in C^\infty_c(M)$ and $B\in\Psi^0(M)$ compactly supported with $\Lambda\subset \Ell(B)$, such that $\forall u \in H^{-N}_{\mathrm{loc}}(M)$ with $Gu \in H^{s'}_c(M)$ and $GPu \in H^{s-m+1}_c(M)$, we have $Bu \in H^s_c(M)$, and the following estimate holds:
    $$\|Bu\|_{H^s} \leq C\bigl(\|GPu\|_{H^{s-m+1}} + \|\chi u\|_{H^{-N}}\bigr).$$
\end{prop}

\begin{prop}[Low Regularity Radial Estimate]
\label{low_reg_radial_estimate}
    Let $P$ and $\Lambda$ be as in Proposition \ref{high_reg_radial_estimate}. Let $G\in\Psi^0(M)$ be compactly supported with $\Lambda \subset \Ell(G)$. Then for any $s,N\in\R$ with $s < \frac{m-1}{2} + \alpha_s$ on $\Lambda$, there exists $C>0$, $\chi\in C^\infty_c(M)$ and $B,E\in\Psi^0(M)$ compactly supported with $\Lambda\subset \Ell(B)$ and $\WF(E)\subset \Ell(G)\setminus\Lambda$, such that $\forall u \in H^{-N}_{\mathrm{loc}}(M)$ with $GPu \in H^{s-m+1}_c(M)$ and $Eu \in H^{s}_c(M)$, we have $Bu \in H^s_c(M)$, and the following estimate holds:
    $$\|Bu\|_{H^s} \leq C\bigl(\|GPu\|_{H^{s-m+1}} + \|Eu\|_{H^s} + \|\chi u\|_{H^{-N}}\bigr).$$
\end{prop}

\subsection{Hyperbolic estimates on spaces of extendable/supported distributions}
\label{section_hyperbolic_estimates}

In this subsection, we state certain estimates that hold for strictly hyperbolic second order differential operators. We follow \cite[Appendix E.5]{dyatlov_zworski_scattering_book}. For a more thorough introduction to strictly hyperbolic operators see \cite[Section 23.2]{hörmander3}. We begin by introducing the Sobolev spaces of extendable and supported distributions, following \cite[Appendix B.2]{hörmander3}.

Let $X \subset \R^n$ be an open set with smooth boundary. For any $s\in\R$, we define the Sobolev space of extendable distributions as
$$\Bar{H}^s(X) = \{u\in H^s(X),\, u = v|_X \,\text{ for some }\, v\in H^s(\R^n)\}.$$
Likewise, we define the Sobolev space of supported distributions as
$$\Dot{H}^s(\Bar{X}) = \{u \in H^s(\R^n),\, \supp(u) \subset \Bar{X}\}.$$
Note that $\Dot{H}^s(\Bar{X})$ is a closed subspace of $H^s(\R^n)$, and is thus a Hilbert space equipped with the norm inherited from $H^s(\R^n)$. The kernel of the restriction $v \in H^s(\R^n) \to v|_X \in \Bar{H}^s(X)$ is precisely the space $\Dot{H}(\R^n \setminus X)$. Thus, $\Bar{H}^s(X)$ can be viewed as the quotient space
$$\Bar{H}^s(X) \cong H^s(\R^n) / \Dot{H}(\R^n \setminus X),$$
and forms a Hilbert space equipped with the quotient norm
$$\|u\|_{\Bar{H}^s(X)} = \inf\{\|v\|_{H^s(\R^n)},\, v\in H^s(\R^n), u = v|_X\}, \quad \forall u \in \Bar{H}^s(X).$$

Defining analogously the spaces $\Bar{C}^\infty(X)$ and $\Dot{C}^\infty(\Bar{X})$ of extendable, respectively supported, smooth functions, we note that the inclusions
$$\Bar{C}^\infty(X) \subset \Bar{H}^s(X), \quad \Dot{C}^\infty(\Bar{X}) \subset \Dot{H}^s(\Bar{X})$$
are dense. The $L^2$ pairing
$$(u,v) \in \Bar{C}^\infty(X) \times \Dot{C}^\infty(\Bar{X}) \to \int_X u(x)\Bar{v}(x)\,dx$$
extends by density to a pairing $\Bar{H}^s(X) \times \Dot{H}^{-s}(\Bar{X}) \to \C$, which provides an isomorphism of dual spaces
$$\bigr(\Bar{H}^s(X)\bigl)^* = \Dot{H}^{-s}(\Bar{X}).$$

Let now $X \subset \R^n$ have a compact smooth boundary $\partial X$ and consider $\Bar{X}$ as a manifold with boundary. Assume that for some function $x\in C^\infty(\Bar{X})$ and some $a,b\in\R$ with $a<b$, we have a product decomposition near the boundary:
$$x\times F:\,\, x^{-1}([a,b)) \xlongrightarrow{\sim} [a,b) \times \partial X,$$
where $\partial X = x^{-1}(\{a\})$. In the following, we denote by $y$ coordinates on $\partial X$ and by $\xi$ and $\eta$ the fiber coordinates on $T^*X$ associated to $x$ and $y$ respectively.

Let $P \in \mathrm{Diff}^2(X)$ be a second order differential operator whose coefficients are smooth up to the boundary, that is, they lie in the space $\Bar{C}^\infty(X)$. Then $P$ defines a bounded operator
$$P: \Bar{H}^s(X) \to \Bar{H}^{s-2}(X), \quad P: \Dot{H}^s(\Bar{X}) \to \Dot{H}^{s-2}(\Bar{X}).$$

\begin{defn}[Strictly Hyperbolic Operator]
\label{def_hyperbolic}
    We say that $P\in \mathrm{Diff}^2(X)$ is strictly hyperbolic with respect to $x$ on $x^{-1}([a,b))$ if its principal symbol $p(x,\xi,y,\eta)$ is real-valued for $x\in [a,b)$ and for each $x\in [a,b)$, $(y,\eta) \in T^*\partial X \setminus\{0\}$ the polynomial
    $$\xi \,\to\, p(x,\xi,y,\eta)$$
    has two distinct real roots.
\end{defn}

We can now state the pertinent estimates, which will be used in the proof of the Fredholm property for the Kerr spectral family. See \cite[Theorem E.56]{dyatlov_zworski_scattering_book} for a proof.

\begin{prop}[Hyperbolic Estimate]
\label{prop_hyperbolic_estimate}
    Assume that $P$ is strictly hyperbolic with respect $x$ on $x^{-1}([a,b))$. Let $\chi_1,\chi_2,\chi_3 \in \Bar{C}^\infty(X)$ be cutoff functions satisfying
    \begin{equation*}
        \supp(\chi_1) \subset x^{-1}\bigl([a,b)\bigr), \quad \chi_2 = 1 \,\,\text{on}\,\, x^{-1}\bigl([a,b)\bigr), \quad \chi_3 = 1 \,\,\text{near}\,\, x^{-1}\bigl(\{b\}\bigr).
    \end{equation*}
    Let $s\in\R$. Then there exists $C>0$ such that for all $u\in \Bar{H}^s(X)$ with $Pu \in \Bar{H}^{s-1}(X)$ the following estimate holds:
    \begin{equation}
    \label{hyperbolic_estimate_extendable}
        \|\chi_1 u\|_{\Bar{H}^s(X)} \leq C\bigl(\|\chi_2 Pu\|_{\Bar{H}^{s-1}(X)} + \|\chi_3 u\|_{\Bar{H}^s(X)}\bigr).
    \end{equation}
    Furthermore, there exists $C>0$ such that for all $v\in \Dot{H}^s(\Bar{X})$ with $Pv \in \Dot{H}^{s-1}(\Bar{X})$ the following estimate holds:
    \begin{equation}
    \label{hyperbolic_estimate_supported}
        \|\chi_1 v\|_{\Dot{H}^s(\Bar{X})} \leq C\|\chi_2 Pv\|_{\Dot{H}^{s-1}(\Bar{X})}
    \end{equation}
\end{prop}

\begin{rmk}
    Proposition \ref{prop_hyperbolic_estimate} can be understood in the context of the Cauchy problem for hyperbolic equations. In the estimate \eqref{hyperbolic_estimate_extendable}, the equation $Pu=f$ is solved with initial data given on the Cauchy surface $\{x=b\}$. The function $\chi_3$ cuts off to a small neighborhood of this Cauchy surface and $\|\chi_3 u\|_{\Bar{H}^s(X)}$ characterizes the size of the initial data. The solution $u$ can be controlled by the forcing $f$ and the initial data. In the estimate \eqref{hyperbolic_estimate_supported}, the equation is solved with initial data given on the Cauchy surface $\{x=a\} = \partial X$. The fact that $v\in \Dot{H}^s(\Bar{X})$ is a supported distribution should be interpreted as the vanishing of the initial data. Thus, the solution can be controlled by the forcing term alone.
\end{rmk}

\section{Quasinormal modes from the complex scaled Kerr spectral family}
\label{section_fredholm}

\subsection{The Kerr metric}
\label{section_kerr_metric}

Kerr spacetime \cite{kerr} is a 4-dimensional Lorentzian manifold, solving the Einstein vacuum equation, i.e. with vanishing Ricci curvature. It describes a rotating black hole with mass $m$ and specific angular momentum $a$. We use the convention $(+,-,-,-)$ for the metric signature and take $a\in(-m,m)$, which is the subextremal range of angular momenta. The region exterior to the black hole is given by ${M_0 = \R_t\times(r_+,\infty)_r\times \Sph^2}$, and the metric in Boyer-Lindquist coordinates on $M_0$ takes the form
\begin{equation}
\begin{split}
\label{metric_boyer_lindquist}
    g = \frac{1}{\rkerr^2}\bigl(\mu-a^2\sin^2(\theta)\bigr)dt^2 &+ \frac{1}{\rkerr^2}4mar\sin^2(\theta)dtd\varphi - \frac{\rkerr^2}{\mu}dr^2 - \rkerr^2d\theta^2 \\ 
    &- \frac{\sin^2(\theta)}{\rkerr^2}\bigl((r^2+a^2)^2 - a^2\sin^2(\theta)\mu\bigr)d\varphi^2,
    \end{split}
\end{equation}
where $(\theta,\varphi) \in (0,\pi)\times(0,2\pi)$ are spherical coordinates on $\Sph^2$. Here $\rkerr=\rkerr(r,\theta)$, $\mu=\mu(r)$ are  defined as
\begin{equation*}
    \rkerr^2 = r^2+a^2\cos^2(\theta), \qquad \mu = r^2-2mr+a^2.
\end{equation*}
Note that with $a\in(-m,m)$, $\mu(r)$ has two real roots $r_- < r_+$ given by 
$$r_+ = m+\sqrt{m^2-a^2}, \qquad r_- = m-\sqrt{m^2-a^2}.$$
The dual metric is
\begin{equation}
    g^{-1} = \frac{1}{\rkerr^2}\Bigl(\bigl(\frac{(r^2+a^2)^2}{\mu}-a^2\sin^2(\theta)\bigr)\partial_t^2 + \frac{4mar}{\mu}\partial_t\partial_\varphi - \mu\partial_r^2 - \partial_\theta^2 - \bigl(\frac{1}{\sin^2(\theta)}-\frac{a^2}{\mu}\bigr)\partial_\varphi^2\Bigr)
\end{equation}
The hypersurface at $r=r_+$ is the black hole horizon. Here, the Boyer-Lindquist coordinates break down. We can, however, extend the Kerr metric across the horizon using a different choice of coordinates. To this end, we define $t_*, \varphi_*$ on $M_0$ by
\begin{equation}
\label{coordinate_change}
    t_* = t + \int_{3m}^r\Bigl(\frac{r'^2+a^2}{\mu(r')} + h(r')\Bigr)\,dr', \qquad \varphi_* = \varphi + \int_{3m}^r\frac{a}{\mu(r')}\,dr',
\end{equation}
where $h\in C^\infty(\R_+)$ is a smooth bounded function, which will be chosen below so that $dt_*$ is everywhere timelike. Note that the choice of $r=3m$ for the coincidence of $t_*$ with $t$ and $\varphi_*$ with $\varphi$ is arbitrary. In the coordinates $(t_*,r,\theta,\varphi_*)$ the dual metric becomes
\begin{equation}
\label{metric_extended}
\begin{split}
    g^{-1} = -\frac{1}{\rkerr^2}\Bigl(&\mu\partial_r^2 + \partial_\theta^2 + \frac{1}{\sin^2(\theta)}\partial_{\varphi_*}^2 + 2a\partial_r\partial_{\varphi_*} + 2\bigl(r^2+a^2+\mu h\bigr)\partial_t\partial_r \\
    &+ 2a(1+h)\partial_t\partial_{\varphi_*}
    + \bigl(\mu h^2 + 2(r^2+a^2)h + a^2\sin^2(\theta)\bigr)\partial_t^2
    \Bigr).
\end{split}
\end{equation}
This expression has a well-defined extension to $r > r_-$ and defines a smooth Lorentzian metric on $M = \R_{t_*}\times(r_0,\infty)_r\times\Sph^2$, where we take an arbitrary $r_0 \in (m,r_+)$ as the position of the boundary of $M$ within the horizon.

We will choose $h\in C^\infty(\R_+)$ so that $dt_*$ becomes everywhere timelike on $M$. In addition, we take $h=-\frac{r^2+a^2}{\mu}$ in $r>R_0$, for some $R_0>r_+$. In this way, $dt$ and $dt_*$ agree outside the ball of radius $R_0$.
\begin{lemma}
\label{timelike}
    Let $R_0 > r_+$. There exists a function $h\in C^\infty(\R_+)$, satisfying 
    $$h(r) = -\frac{r^2+a^2}{\mu(r)}, \quad \text{for}\quad r>R_0,$$
    such that the one-form $dt_*$, with $t_*$ as in \eqref{coordinate_change}, is everywhere timelike, i.e. 
    $$g^{-1}(dt_*,dt_*)>0 \quad\text{on}\quad M.$$
\end{lemma}
\begin{proof}
    From \eqref{metric_extended} we see that
    $$g^{-1}(dt_*,dt_*) = -\frac{1}{\rkerr^2}\bigl(\mu h^2 + 2(r^2+a^2)h + a^2\sin^2(\theta)\bigr).$$
    Since $\rkerr^2>0$, the condition $g^{-1}(dt_*,dt_*)>0$ on $M$ is equivalent to
    \begin{equation}
    \label{timelike_condition}
        \mu h^2 + 2(r^2+a^2)h + a^2 < 0, \quad \forall r\in(r_0,\infty).
    \end{equation}
    Note that $h(r)=-1$ fulfills this condition everywhere. Indeed,
    $$\mu - 2(r^2+a^2) + a^2 = -(r^2+2mr) < 0, \quad \forall r\in(r_0,\infty).$$
    In the region $\mu>0$, i.e. for $r>r_+$, $h$ must lie between the roots of the quadratic polynomial in $h$ given by \eqref{timelike_condition}. Thus, we need
    \begin{equation*}
        \Bigl|h(r) + \frac{r^2+a^2}{\mu} \Bigr| < \frac{1}{\mu}\sqrt{r^2(r^2+a^2) + 2mr}, \quad \text{for}\quad r>r_+.
    \end{equation*}
    Note that for $\mu>0$, we have $-\frac{r^2+a^2}{\mu} < -1$ and $r^2+a^2 > 2mr$. So for $h$ in the range 
    $$h(r) \in \Bigl[-\frac{r^2+a^2}{\mu},-1\Bigr],$$
    we find
    $$\Bigl|h(r) + \frac{r^2+a^2}{\mu} \Bigr| \leq \frac{r^2+a^2}{\mu} - 1 = \frac{2mr}{\mu} < \frac{1}{\mu}\sqrt{r^2(r^2+a^2) + 2mr},$$
    and condition \eqref{timelike_condition} is satisfied in $r>r_+$ for $h$ in this range. Thus, for some arbitrary $r_+ < \Tilde{r} < R_0$, we can choose $h(r) = -1$ in $r<\Tilde{r}$ and $h(r) = -\frac{r^2+a^2}{\mu}$ in $r>R_0$, and let $h(r)$ interpolate smoothly between these values in $r\in(\Tilde{r},R_0)$.
\end{proof}

Let $\Box_g = |\det(g)|^{-\frac{1}{2}}\partial_i(|\det(g)|^{\frac{1}{2}}g^{ij}\partial_j)$ be the wave operator associated to the Kerr metric $g$ on $M$. We define the Kerr spectral family $P(\sigma)$ by
$$\Box_g(e^{-i\sigma t_*}u(r,\theta,\varphi_*)) = e^{-i\sigma t_*}P(\sigma)u(r,\theta,\varphi_*).$$
Thus, $P(\sigma)$ is a family of operators on the spatial slice $X = (r_0,\infty)_r\times\Sph^2$ parameterized by $\sigma \in \C$. It is obtained from $\Box_g$ by replacing $\partial_{t_*}$ with $-i\sigma$:
\begin{equation}
\label{operator_full}
\begin{split}
    P(\sigma) = &\rkerr^{-2}\Bigl(D_r\mu D_r + \frac{1}{\sin(\theta)}D_\theta\sin(\theta)D_\theta + \frac{1}{\sin^2(\theta)}D_{\varphi_*}^2 + 2aD_rD_{\varphi_*}\Bigl) \\
    - \sigma&\rkerr^{-2}\Bigl(D_r(r^2+a^2+\mu h) + (r^2+a^2+\mu h)D_r + 2a(1+h)D_{\varphi_*}\Bigr) \\
    + \sigma^2&\rkerr^{-2}\Bigr(\mu h^2 + 2(r^2+a^2)h + a^2\sin^2(\theta)\Bigl).
\end{split}
\end{equation}
We will consider the action of $P(\sigma)$ on $\Bar{H}^s(X)$, the Sobolev spaces of extendable distributions defined in Section \ref{section_hyperbolic_estimates}, where we use the density $\rkerr^2\sin(\theta)dr d\theta d\varphi_*$ on $X$. Since $r^2 \leq \rkerr^2 \leq 2r^2$, the norm defined in this way is actually equivalent to the standard Sobolev norm  with respect to the Lebesgue measure on $X \subset \R^3$. Note that $P(\sigma)$ defines a bounded operator 
$$P(\sigma): \Bar{H}^s(X) \to \Bar{H}^{s-2}(X), \quad \forall s\in\R.$$
By the mode stability of the Kerr wave equation, see Proposition \ref{mode_stability} below, there are no quasinormal modes with $\sigma$ in the upper half-plane. However, if such unstable quasinormal mode solutions existed, they would be square integrable on the spatial slice $X$ and could be characterized through the kernel of $P(\sigma)$ on the above Sobolev spaces for $\Im(\sigma)>0$, see the discussion in the Section \ref{section_introduction}. In order to access quasinormal modes in the lower half-plane, we will deform $P(\sigma)$ by the complex scaling procedure.

\subsection{The complex scaled operator}
\label{section_scaled_operator}
We will now apply the complex scaling to the operator $P(\sigma)$, see Section \ref{section_complex_scaling}. To this end, let
$$F_\beta: X = \{x\in\R^3, |x|>r_0\} \to \C^3, \quad F_\beta(x) = f_\beta(|x|)\frac{x}{|x|},$$
where $f_\beta$ is the function constructed in Definition \ref{def_phase_function}. We denote the image of this map as $X_\beta = F_\beta(X)$. Thus, $X_\beta$ agrees with $X$ in the ball $B_{R_1}$, whereas, outside this ball, $X_\beta$ is a deformation of $X$ into $\C^3$.

\begin{figure}[b]
    \centering
    \includegraphics{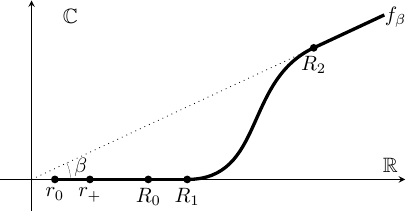}
    \caption{Depiction of the image of $f_\beta$ in the complex plane. The complex deformation begins at $|x|=R_1$, far from the horizon at $|x|=r_+$. The analytic extension of $P(\sigma)$ is already possible in $|x|>R_0$. For $|x|>R_2$ the complex scaled contour coincides with $e^{i\beta}\R^3$. Note that $|x|=r_0$ is the (topological) boundary of $X_\beta$ inside the horizon.}
    \label{fig_complex_scaling}
\end{figure}

To apply complex scaling, we must show that, for some $R_0$ large enough, $P(\sigma)$ has an analytic extension from $\Omega = \R^n \setminus \Bar{B}_{R_0}$ to an open set $U \subset \C^3$. Moreover, $U$ should include the deformed contours
$$\Gamma_\beta = F_\beta(\Omega) = X_\beta \cap \{z \in \C^3,\, |z| > R_0\},$$
where we take the radius $R_1$ where we actually start deforming $X$ larger than $R_0$. The complex scaled operator $P_\beta(\sigma)$ can then be defined by restriction of $P(\sigma)$ to the maximally totally real submanifold $\Gamma_\beta$.

For $(x_1,x_2,x_3)\in \Omega$, we have
$$F_\beta(x_1,x_2,x_3) = (e^{i\phi_\beta(r)}x_1, e^{i\phi_\beta(r)}x_2, e^{i\phi_\beta(r)}x_3),$$
where $r = \sqrt{x_1^2+x_2^2+x_3^2}$ and $\phi_\beta(r) \in [0,\beta]$, respectively $[\beta,0]$ for $\beta<0$. Denoting by $(z_1,z_2,z_3)$ coordinates on $\C^3$, we see that $(z_1,z_2,z_3) \in \Gamma_\beta$ satisfy 
$$\arg(z_1) = \arg(z_2) = \arg(z_3) \,\,\,\text{and}\,\,\, \arg(z_1^2+z_2^2+z_3^2) \in [0,2\beta], \,\,\text{respectively}\,\, [2\beta,0].$$
Furthermore, 
$$|z_1^2+z_2^2+z_3^2| = |z_1|^2+|z_2|^2+|z_3|^2 > R_0^2, \quad \text{on}\quad \Gamma_\beta.$$
For some $\delta>0$ small, let $U$ be the following open subset of $\C^3$:
\begin{equation}
\label{open_set_complex}
\begin{split}
    U = \{(z_1,z_2,z_3)\in\C^3,\,\, |z_1^2+z_2^2+z_3^2| > R_0^2,\,\, \arg(z_1^2+z_2^2+z_3^2) \in (-2\delta,2\pi-2\delta), \\
    |z|^2 < 2|x^2+y^2+z^2| \}.
\end{split}
\end{equation}
Then the discussion above shows that
$$\bigcup_{\beta\in[0,\pi-\delta)}\Gamma_\beta \subset U.$$
We could also scale in the other direction, that is, choosing $\beta$ negative. In this case we would replace $\arg(z_1^2+z_2^2+z_3^2) \in (-2\delta,2\pi-2\delta)$ in the definition of $U$ by 
$$\arg(z_1^2+z_2^2+z_3^2) \in (-2\pi+2\delta,2\delta),$$
and the resulting open subset contains the union $\bigcup_{\beta\in[0,\pi-\delta)}\Gamma_\beta$. We will now show that the operator $P(\sigma)$ in \eqref{operator_full} has an analytic continuation to $U$.

\begin{lemma}
\label{analytic_coefficients}
    Choosing $R_0$ large enough, the Kerr spectral family $P(\sigma)$ defines a differential operator with analytic coefficients on the open set $U\subset\C^3$ as in \eqref{open_set_complex}. The same holds if the condition $\arg(z_1^2+z_2^2+z_3^2) \in (-2\delta,2\pi-2\delta)$ is replaced by $\arg(z_1^2+z_2^2+z_3^2) \in (-2\pi+2\delta,2\delta)$ in the definition of $U$.
\end{lemma}
\begin{proof}
We choose $R_0$, so that $h(r) = -\frac{r^2+a^2}{\mu}$ in $U\cap\R^3$, see Lemma \ref{timelike}. In this domain, rewriting the Kerr spectral family, given in equation \eqref{operator_full}, in Cartesian coordinates, we obtain:
    \begin{equation*}
    \begin{split}
        P(\sigma) = -\frac{r^2}{\rkerr^2}\Bigl(&\partial_{x_1}^2+\partial_{x_2}^2+\partial_{x_3}^2 - \frac{2m}{r^3}(x_1\partial_{x_1}+x_2\partial_{x_2}+x_3\partial_{x_3})^2 \\
        &+ \frac{a^2}{r^4} \bigl((x_1\partial_{x_1}+x_2\partial_{x_2}+x_3\partial_{x_3})^2 - (x_1\partial_{x_1}+x_2\partial_{x_2}+x_3\partial_{x_3})\bigr) \\
        &+\frac{2a}{r^3}(x_1\partial_{x_1}+x_2\partial_{x_2}+x_3\partial_{x_3})(x_1\partial_{x_2}-x_2\partial_{x_1})\Bigr) \\
        &-i\sigma \frac{4amr}{\rkerr^2\mu}(x_1\partial_{x_2}-x_2\partial_{x_1}) - \sigma^2(1 + 2mr\frac{r^2+a^2}{\rkerr^2\mu})
    \end{split}
    \end{equation*}
Note that the square root of $x_1^2+x_2^2+x_3^2$ extends to an analytic function
$$r = \sqrt{z_1^2+z_2^2+z_3^2},$$
on $U$ and satisfies $|r| > R_0$. This is where restricting the argument of $z_1^2+z_2^2+z_3^2$ is essential. We thereby also obtain an analytic extension to $U$ of the functions
$$\rkerr^2 = r^2 + a^2\frac{z^2}{r^2}, \quad\text{and}\quad \mu = r^2 - 2mr + a^2.$$
Thus, the coefficients of $P(\sigma)$ certainly extend to meromorphic functions on $U$, and it remains to check that they have no poles inside of $U$. To this end, notice that on $U$, we have
\begin{align*}
    |\rkerr^2| &\geq |r^2| - a^2\frac{|z|^2}{|r^2|} > R_0^2 - 2a^2,\\
    |\mu| &\geq |r^2| - 2m|r| - a^2 > R_0^2 - 2mR_0 - a^2.
\end{align*}
Choosing $R_0$ large enough, we can ensure that $|\rkerr^2|, |\mu| > 0$ on $U$.
\end{proof}

\begin{rmk}
    Note that for any fixed $\beta \in (-\pi,\pi)$, we can choose $\delta$ small enough so that both $\Gamma_0 = \Omega \subset \R^3$ and $\Gamma_\beta$ are contained in an open subset of $\C^3$ on which $P(\sigma)$ is analytic. This will allow us, in Proposition \ref{prop_qnm_scaling_independent} below, to relate solutions of $P(\sigma)|_{\Gamma_\beta}u = 0$ to solutions for the original Kerr spectral family. The requirement of finding an open subset of $\C^3$ that includes both $\Gamma_\beta$ and $\Gamma_0$, on which the square root $\sqrt{z_1^2+z_2^2+z_3^2}$ is analytic, restricts the scaling angle to lie in the interval $(-\pi,\pi)$. This is the reason why the statement in Theorem \ref{thm_analytic_continuation} only concerns $\sigma$ lying in the first sheets of the logarithmic cover. Scaling beyond $\pi$, and thus exploring further sheets of the logarithmic cover in Theorem \ref{thm_analytic_continuation}, may be possible, but would require a more sophisticated analysis.
\end{rmk}

By restricting to the maximally real submanifolds $\Gamma_\beta \subset U$, we now obtain a differential operator $P(\sigma)|_{\Gamma_\beta}$ on $\Gamma_\beta$ for each $\beta \in (-\pi,\pi)$. We define the complex scaled operator $P_\beta(\sigma)$ on $X_\beta$ by
\begin{equation*}
    P_\beta(\sigma) = \begin{cases} P(\sigma) &\text{on } X_\beta \setminus \Gamma_\beta\\ P(\sigma)|_{\Gamma_\beta} &\text{on } \Gamma_\beta \end{cases}
\end{equation*}
Note that, with this definition, $P_0(\sigma) = P(\sigma)$.

Using $F_\beta^{-1}: X_\beta \to X$ as a coordinate chart and additionally working with the coordinates $(r,\theta,\varphi_*)$ on $X$, the complex scaled operator is given in $\Gamma_\beta$, i.e. for $r>R_0$, by
\begin{equation}
\label{complex_scaled_operator}
\begin{split}
    P_\beta(\sigma) = &\frac{1}{\rkerr_\beta^{2}}\Bigl(\frac{1}{f_\beta'}D_r \frac{\mu_\beta}{f_\beta'} D_r + \frac{1}{\sin(\theta)}D_\theta\sin(\theta)D_\theta + \frac{1}{\sin^2(\theta)}D_{\varphi_*}^2 + \frac{2a}{f_\beta'}D_rD_{\varphi_*}\Bigl) \\
    &+ \sigma \frac{4maf_\beta}{\rkerr_\beta^2\mu_\beta}D_{\varphi_*}
    - \sigma^2\Bigr(1 + \frac{2mf_\beta(f_\beta^2+a^2)}{\rkerr_\beta^2\mu_\beta}\Bigl),
\end{split}
\end{equation}
where 
$$\rkerr_\beta^2 = \rkerr(f_\beta(r), \theta)^2 = f_\beta(r)^2 + a^2\cos^2(\theta), \quad \mu_\beta = \mu(f_\beta(r)) = f_\beta(r)^2 - 2mf_\beta(r) + a^2,$$ with $f_\beta(r)$ as in Definition \ref{def_phase_function}.

We let $P_\beta(\sigma)$ act on $\Bar{H}^s(X_\beta)$ and $P_\beta(\sigma)^*$ on $\Dot{H}^s(\Bar{X}_\beta)$, the Sobolev spaces of extendable distributions respectively supported distributions on $X_\beta$, where the norms are defined with derivatives taken in the coordinate chart $F_\beta^{-1}$, and the pushforward measure $(F_\beta)_*(\rkerr^2\sin(\theta)drd\theta d\varphi_*)$ is used. With this definition, the pullback induces unitary equivalences
$$F_\beta^*: \Bar{H}^s(X_\beta) \to \Bar{H}^s(X), \quad F_\beta^*: \Dot{H}^s(\Bar{X}_\beta) \to \Dot{H}^s(\Bar{X}).$$

Although we initially defined the complex scaled operator as the analytic continuation of $P(\sigma)$ restricted to the family of submanifolds $X_\beta$, we can also view $P_\beta(\sigma)$ from a different perspective. Namely, as a family of operators $F_\beta^*P_\beta(\sigma)(F_\beta^{-1})^*$ acting on the same space $\Bar{H}^s(X)$. By a slight abuse of notation, we continue to denote this family as $P_\beta(\sigma)$. It coincides with the coordinate expression given in \eqref{complex_scaled_operator}. From this perspective, the formal adjoint of $P_\beta(\sigma)$ is given by
\begin{equation*}
\label{adjoint_scaling}
\begin{split}
    P_\beta(\sigma)^* = P_{-\beta}(\Bar{\sigma}) &+ 2\frac{\mu_{-\beta}}{\rkerr^2f'_{-\beta}}D_r\Bigl(\frac{\rkerr^2}{\rkerr^2_{-\beta}f'_{-\beta}}\Bigr)D_r + \frac{2a}{\rkerr^2}D_r\Bigl(\frac{\rkerr^2}{\rkerr^2_{-\beta}f'_{-\beta}}\Bigr)D_{\varphi_*} + \frac{2}{\rkerr^2}D_\theta\Bigl(\frac{\rkerr^2}{\rkerr_{-\beta}^2}\Bigr)D_\theta \\ &+ \frac{1}{\rkerr^2\sin(\theta)}D_\theta\Bigl(\sin(\theta)D_\theta\Bigl(\frac{\rkerr^2}{\rkerr_{-\beta}^2}\Bigr)\Bigr)
    + \frac{1}{\rkerr^2}D_r\Bigr(\frac{\mu_{-\beta}}{f'_{-\beta}}D_r\Bigl(\frac{\rkerr^2}{\rkerr^2_{-\beta}f'_{-\beta}}\Bigr)\Bigl).
\end{split}
\end{equation*}
Since
$$D_r\Bigl(\frac{\rkerr^2}{\rkerr^2_{-\beta}f'_{-\beta}}\Bigr) = \mathcal{O}(r^{-3}), \quad D_\theta\Bigl(\frac{\rkerr^2}{\rkerr_{-\beta}^2}\Bigr) = \mathcal{O}(r^{-2}), \quad \text{as}\,\, r\to\infty,$$
we have
$$P_\beta(\sigma)^* - P_{-\beta}(\Bar{\sigma}) \in r^{-3}\text{Diff}^1(X),$$
i.e. a first order differential operator with coefficients decaying as $r^{-3}$.

\subsection{Fredholm property of the complex scaled spectral family}
\label{section_fredholm_property}

We now state the central result of this section, namely that the complex scaled Kerr spectral family, $P_\beta(\sigma)$, defines a Fredholm operator on Sobolev spaces of high enough regularity, for $\sigma$ in a $\beta$-dependent open half-plane of $\C$.

In order to prove Fredholm estimates, we must modify the domain of our operators. Instead of working with the full Sobolev space of supported distributions, we let $P_\beta(\sigma)$ act on
\begin{equation*}
    \mathcal{X}^s_\beta = \{u \in \Bar{H}^s(X_\beta) \,\,|\,\, P_\beta(0)u \in \Bar{H}^{s-1}(X_\beta)\},
\end{equation*}
endowed with the norm
\begin{equation*}  
    \|u\|_{\mathcal{X}^s_\beta} = \|u\|_{\Bar{H}^s(X_\beta)} + \|P_\beta(0)u\|_{\Bar{H}^{s-1}(X_\beta)}.
\end{equation*}
Note that $P_\beta(\sigma): \mathcal{X}^s_\beta \to \Bar{H}^{s-1}(X_\beta)$ defines a bounded operator for each $\sigma \in \C$, since $P_\beta(\sigma) - P_\beta(0) \in \mathrm{Diff}^1(X_\beta)$.
Similarly, we define
$$\mathcal{Y}^s_\beta = \{v \in \Dot{H}^s(\Bar{X}_\beta) \,\,|\,\, P_\beta(0)^*v \in \Dot{H}^{s-1}(\Bar{X}_\beta)\},$$
which will serve as the domain for the operator $P_\beta(\sigma)^*$.

We will also make use of the weighted Sobolev spaces
$$\Bar{H}^{s,r}(X) = \brac{x}^{-r}\Bar{H}^s(X), \quad \Dot{H}^{s,r}(\Bar{X}) = \brac{x}^{-r}\Dot{H}^s(\Bar{X})$$
for $r \in \R$. Note that the inclusions $\Bar{H}^{s,r}(X) \subset \Bar{H}^{s',r'}(X)$ and $\Dot{H}^{s,r}(\Bar{X}) \subset \Dot{H}^{s',r'}(\Bar{X})$ are compact for $s>s'$, $r>r'$. On the deformed manifold $X_\beta$, we define
$$\Bar{H}^{s,r}(X_\beta) = (F_\beta^{-1})^*\Bar{H}^{s,r}(X), \quad \Dot{H}^{s,r}(\Bar{X}_\beta) = (F_\beta^{-1})^*\Dot{H}^{s,r}(\Bar{X}).$$

\begin{prop}
\label{fredholm_prop}
    Let
    \begin{equation*}
        \Lambda_\beta = \bigl\{\sigma \in \C\setminus\{0\} \,\,|\,\, \arg(\sigma) \in (-\beta,\pi-\beta)\bigr\},
    \end{equation*}
    and
        \begin{equation*}
            \Omega_s = \Bigl\{\sigma \in \C \,\,|\,\, \Im(\sigma) > \frac{1}{\alpha}\Bigl(\frac{1}{2} - s\Bigr)\Bigr\},
        \end{equation*}
    where
        $$\alpha = 2\Bigl(m + \frac{m^2}{\sqrt{m^2-a^2}}\Bigr).$$
    Then
    $$P_\beta(\sigma): \mathcal{X}^s_\beta \to \Bar{H}^{s-1}(X_\beta)$$
    is an analytic family of Fredholm operators for $\sigma \in \Lambda_\beta \cap \Omega_s$.
\end{prop}

Proposition \ref{fredholm_prop} will follow from Fredholm estimates for $P_\beta(\sigma)$ and $P_\beta(\sigma)^*$. That is, we show that $\|u\|_{\Bar{H}^{s}(X_\beta)}$ is bounded by $\|P_\beta(\sigma)u\|_{\Bar{H}^{s-1}(X_\beta)}$ and $\|v\|_{\Dot{H}^{-s+1}(\Bar{X}_\beta)}$ is bounded by $\|P_\beta(\sigma)^*v\|_{\Dot{H}^{-s}(\Bar{X}_\beta)}$ modulo compact error terms. In the region where complex scaling takes place, we show that both operators are elliptic, and use the additional ellipticity at infinity, i.e. scattering ellipticity, to obtain such estimates. Close to the black hole horizon, we use the methods of \cite{vasy_KdS} based on the Hamiltonian flow of the principal symbol and the propagation and radial estimates of Section \ref{section_microlocal_estimates}. Finally, we close the Fredholm estimates inside the black hole horizon using the results of Section \ref{section_hyperbolic_estimates} on strictly hyperbolic operators.

\subsection{Elliptic estimates in the complex scaling region}
\label{section_scattering_elliptic}

We begin our discussion of Proposition \ref{fredholm_prop} by proving elliptic estimates for $P_\beta(\sigma)$ and its formal adjoint in the complex scaling region.
\begin{lemma}
\label{lemma_scattering_elliptic}
    $P_\beta(\sigma)$ and $P_\beta(\sigma)^*$ are elliptic on $\Gamma_\beta = X_\beta \cap \{x\in\C^3 \,|\, |x|>R_0\}$. Moreover, let $\sigma \in \Lambda_\beta$ and $\chi, \Tilde{\chi} \in C^\infty(X_\beta)$ with $\supp(\chi), \supp(\Tilde{\chi}) \subset \{|x|>R_0\}$ and $\Tilde{\chi}=1$ on $\supp(\chi)$. Then for each $N\in\N$, there exists $C>0$, such that the following estimates hold for all $u \in \Bar{H}^s(X_\beta)$ and $v \in \Dot{H}^s(\Bar{X}_\beta)$:
    \begin{equation*}
    \begin{split}
        \|\chi u\|_{\Bar{H}^{s}(X_\beta)} &\leq C\bigl(\|\Tilde{\chi}P_\beta(\sigma)u\|_{\Bar{H}^{s-2}(X_\beta)} + \|\Tilde{\chi} u\|_{\Bar{H}^{-N, -1}(X_\beta)}\bigr), \\
        \|\chi v\|_{\Dot{H}^s(\Bar{X}_\beta)} &\leq C\bigl(\|\Tilde{\chi}P_\beta(\sigma)^*v\|_{\Dot{H}^{s-2}(\Bar{X}_\beta)} + \|\Tilde{\chi} v\|_{\Dot{H}^{-N,-1}(\Bar{X}_\beta)}\bigr).
    \end{split}
    \end{equation*}
\end{lemma}
\begin{proof}
Denote by $p_\beta \in C^\infty(T^*\Gamma_\beta)$ the principal symbol of $P_\beta(\sigma)$. Working in the coordinate chart given by $F_\beta^{-1}: X_\beta \to X \subset \R^3$ and using spherical coordinates on $\R^3$, as in \eqref{complex_scaled_operator}, we find for $r>R_0$:
    \begin{equation}
    \label{symbol_scaled}
    \begin{split}
        p_\beta(r,\theta,{\varphi_*},\xi,\eta,\nu) = \frac{r^2}{\rkerr(f_\beta(r),\theta)^2} \Bigl(\frac{\mu(f_\beta(r))}{f'_\beta(r)^2r^2} \,\xi^2 + \frac{\eta^2}{r^2} + \frac{\nu^2}{r^2\sin(\theta)^2} + \frac{2a}{f'_\beta(r)r}\,\xi\frac{\nu}{r}\Bigr).
    \end{split}
    \end{equation}
Here, $\xi, \eta, \nu$ denote the fiber coordinates associated to $r, \theta, \varphi_*$ respectively.

Recall that $f_\beta(r)=e^{i\phi_\beta(r)}r$ with $|r\phi_\beta'(r)| < \pi\varepsilon$. Thus, choosing $R_0$ large enough, $\rkerr(f_\beta(r),\theta)^2 = f_\beta(r)^2 + a^2\cos^2(\theta)$ satisfies
$$\frac{r^2}{2} \leq r^2 - a^2\cos^2(\theta) \leq |\rkerr(f_\beta(r),\theta)^2| \leq r^2 + a^2\cos^2(\theta) \leq 2r^2, \quad \forall r>R_0.$$ So we can estimate the prefactor in \eqref{symbol_scaled} by
$$\frac{1}{2} \leq \frac{r^2}{\rkerr(f_\beta(r),\theta)^2} \leq 2.$$

Since $f_\beta'(r) = e^{i\phi_\beta(r)}(1 + ir\phi_\beta'(r))$ satisfies $|f_\beta'(r)| \geq 1$, we can estimate the last term in $p_\beta$ by
$$\Bigl| \frac{2a}{f'_\beta(r)r}\,\xi\frac{\nu}{r} \Bigr| \leq \frac{a}{r}\Bigl(\xi^2 + \frac{\nu^2}{r^2}\Bigr) \leq \frac{a}{R_0}\Bigl(\xi^2 + \frac{\eta^2}{r^2} + \frac{\nu^2}{r^2\sin(\theta)^2}\Bigr), \quad \forall r>R_0.$$

For the first term in \eqref{symbol_scaled}, we write
$$\frac{\mu(f_\beta(r))}{f'_\beta(r)^2r^2} = \frac{f_\beta(r)^2 - 2mf_\beta(r) + a^2}{f'_\beta(r)^2r^2} = 1 + \frac{f_\beta(r)^2 - f'_\beta(r)^2r^2}{f'_\beta(r)^2r^2} + \frac{2mf_\beta(r)}{f'_\beta(r)^2r^2} + \frac{a^2}{f'_\beta(r)^2r^2}.$$
Then we can estimate the above terms as follows:
\begin{align*}
    \Bigl|\frac{f_\beta(r)^2 - f'_\beta(r)^2r^2}{f'_\beta(r)^2r^2}\Bigr| &= \frac{|r^2 - (1 + ir\phi_\beta'(r))^2r^2|}{|f'_\beta(r)^2|r^2} \leq 2|r\phi_\beta'(r)| + (r\phi_\beta'(r))^2  < 2\pi\varepsilon + \pi^2\varepsilon^2, \\
    \Bigl|\frac{2mf_\beta(r)}{f'_\beta(r)^2r^2}\Bigr| &\leq \frac{2m}{r} < \frac{2m}{R_0}, \quad\qquad
    \Bigl|\frac{a^2}{f'_\beta(r)^2r^2}\Bigr| \leq \frac{a^2}{r} < \frac{a^2}{R_0}.
\end{align*}

Combining the above estimates, we find that
$$\bigl|p_\beta(r,\theta,\varphi_*,\xi,\eta,\nu)\bigr| \geq C\Bigl(\xi^2 + \frac{\eta^2}{r^2} + \frac{\nu^2}{r^2\sin(\theta)^2}\Bigr),$$
uniformly for $r>R_0$, that is, $C$ depends only on $R_0$ and $\varepsilon$. Note that the right hand side is just the Euclidean norm on the fibers of the cotangent bundle. Written in Cartesian coordinates, the estimate becomes
\begin{equation}
\label{elliptic_estimate_scaling}
    \bigl|p_\beta(x,y,z,\xi_x,\xi_y,\xi_z)\bigr| \geq C\bigl(\xi_x^2 + \xi_y^2 + \xi_z^2\bigr).
\end{equation}

We should remark that the coordinate system used above breaks down at $\theta \in \{0,\pi\}$. However, the principal symbol $p_\beta$ is smooth on $X_\beta$, compare Lemma \ref{analytic_coefficients}, so the estimate \eqref{elliptic_estimate_scaling} extends by continuity to all of $X_\beta \cap \{|x|>R_0\}$ and shows that $P_\beta(\sigma)$ is uniformly elliptic there.

Let $\chi_1, \chi_2 \in C^\infty(X_\beta)$ be cutoff functions with 
$$\supp(\chi_1),\, \supp(\chi_2) \subset \{\Tilde{\chi}=1\},\quad \chi_2 = 1 \quad\text{on}\,\, \supp(\chi_1) \quad\text{and}\quad \chi_1 = 1 \quad\text{on}\,\, \supp(\chi).$$
We can extend $F_\beta^*(\chi_1 P_\beta(\sigma))(F_\beta^{-1})^*$ to all of $\R^3$ as an operator in $\Psi^2(\R^3)$. The estimate \eqref{elliptic_estimate_scaling} shows that this operator is uniformly elliptic on $\supp(\chi)$. For any $u\in \Bar{H}^s(X_\beta)$, we extend $F_\beta^*\chi_2 u$ to an element of $H^s(\R^3)$, and apply the uniform elliptic estimate, Proposition \ref{uniform_elliptic_estimate}. For simplicity, we will drop the pullback map $F_\beta^*$ from our notation in what follows, but one should keep in mind that all functions are pullbacks to $\R^3$ of functions defined on $X_\beta$. Then the uniform elliptic estimate gives
\begin{equation*}
\begin{split}
    \|\chi u\|_{H^s(\R^3)} &\leq C\bigl(\|\chi_1 P_\beta(\sigma) \chi_2 u\|_{H^{s-2}(\R^3)} + \|\chi_2 u\|_{H^{-N}(\R^3)}\bigr) \\
    &\leq C\bigl(\|\chi_2 P_\beta(\sigma) u\|_{H^{s-2}(\R^3)} + \|\chi_2 u\|_{H^{-N}(\R^3)}\bigr).
\end{split}
\end{equation*}

We can now use the large $r$ asymptotic behavior of $P_\beta(\sigma)$ to improve the error term $\|\chi_2 u\|_{H^{-N}(\R^3)}$ above to an error term measured in the weighted Sobolev space $H^{-N,-1}(\R^3)$. To this end, notice from \eqref{complex_scaled_operator} that the complex scaled operator approaches $e^{-2i\beta}\Delta - \sigma^2$ as $r\to\infty$, that is,
\begin{equation*}
    F_\beta^*P_\beta(\sigma)(F_\beta^{-1})^* = e^{-2i\beta}\bigl(\Delta - (e^{i\beta}\sigma)^2\bigr) + \frac{1}{r}Q,
\end{equation*}
where $\Delta$ is the Laplacian on $\R^3$ and $Q\in\mathrm{Diff}^2(X)$ is a differential operator with bounded coefficients as $r\to\infty$.

For $\sigma\in\Lambda_\beta$, $(e^{i\beta}\sigma)^2$ is in the resolvent set of the Laplacian. Thus, there is a bounded inverse
$$\bigl(\Delta - (e^{i\beta}\sigma)^2\bigr)^{-1}: H^{-N}(\R^3) \to H^{-N+2}(\R^3).$$
So we can estimate:
\begin{equation*}
\begin{split}
    \|\chi_2 u\|_{H^{-N}(\R^3)} &\leq C\|e^{-2i\beta}(\Delta - (e^{i\beta}\sigma)^2\bigr)\chi_2 u\|_{H^{-N-2}(\R^3)} \\
    &\leq C\bigl(\|\chi_2 P_\beta(\sigma)u\|_{H^{-N-2}(\R^3)} + \|[P_\beta(\sigma),\chi_2]u\|_{H^{-N-2}(\R^3)} + \|r^{-1}Q\chi_2 u\|_{H^{-N-2}(\R^3)}\bigr) \\
    &\leq C\bigl(\|\Tilde{\chi}P_\beta(\sigma)u\|_{H^{s}(\R^3)} + \|\Tilde{\chi}u\|_{H^{-N,-1}(\R^3)}\bigr),
\end{split}
\end{equation*}
where in the last line we used that $[P_\beta,\chi_2] \in \mathrm{Diff}^1(X_\beta)$ vanishes for $r$ large enough.

The estimate in the lemma follows, since Lebesgue measure and the density coming from the Kerr metric induce equivalent norms on $H^s(\R^3)$, and since $F_\beta^*$ maps $\Bar{H}^s(X_\beta)$ to $\Bar{H}^s(X)$ unitarily. Note also that for functions supported away from the boundary of $X$ at $r=r_0$ the norm on $\Bar{H}^s(X)$ is just the usual Sobolev norm.

Consider now the adjoint $P_\beta(\sigma)^*$. Abusing notation and denoting by $P_\beta(\sigma)$ and $P_\beta(\sigma)^*$ the local coordinate expressions on $X$, we have by \eqref{adjoint_scaling}:
$$P_\beta(\sigma)^* = P_{-\beta}(\Bar{\sigma}) + r^{-3}\Tilde{Q},$$
where $\Tilde{Q}\in\mathrm{Diff}^1(X)$ has bounded coefficients as $r\to\infty$. Thus, the principal symbol of $P_\beta(\sigma)^*$ is given by $p_{-\beta}$ and the arguments above show the uniform ellipticity of $P_\beta(\sigma)^*$ on $\{r > R_0\}$. Furthermore, the term $r^{-3}\Tilde{Q}$ does not influence the $r\to\infty$ asymptotic behavior and $P_\beta(\sigma)^*$ asymptotically approaches $e^{2i\beta}\bigl(\Delta - (e^{-i\beta}\Bar{\sigma})^2\bigr)$ up to an error term in $r^{-1}\mathrm{Diff}^2(X)$. Once again, $(e^{-i\beta}\Bar{\sigma})^2$ lies in the resolvent set of the Laplacian for $\sigma\in\Lambda_\beta$, so the arguments above go through to give the desired estimate.
\end{proof}

\begin{rmk}
    Proposition \ref{lemma_scattering_elliptic} fits naturally into the context of scattering pseudodifferential operators, see \cite{vasy_minicourse}. The operator $P_\beta(\sigma)$ can be viewed as a scattering operator, whose principal symbol at spatial infinity is $e^{-2i\beta}|\xi|^2 - \sigma^2$, i.e. that of the operator $e^{-2i\beta}\Delta - \sigma^2$ featured in the proof above. Thus, $P_\beta(\sigma)$ is scattering elliptic outside the ball $B_{R_0}$ for $\sigma\in\Lambda_\beta$. This perspective immediately leads to the improvement of the error terms in the estimates of Proposition \ref{lemma_scattering_elliptic} to $\|\Tilde{\chi}u\|_{\Bar{H}^{-N,-N}(X_\beta)}$ and $\|\Tilde{\chi}v\|_{\Dot{H}^{-N,-N}(\Bar{X}_\beta)}$ for any $N$. Since all our operators are scattering elliptic, and thus in some sense trivial, at spatial infinity, we chose not to invoke this formalism. However, note that the propagation and radial estimates of Section \ref{section_microlocal_estimates} have counterparts at spatial infinity in the scattering calculus, which could for instance be used to study the operator $P_0(\sigma)$ for $\sigma$ on the real line, see \cite{melrose_scattering_calculus} for an application of these ideas to the Laplacian on asymptotically Euclidean spaces. The method of complex scaling avoids these issues by replacing $P_0(\sigma)$ with the operator $P_\beta(\sigma)$, which is scattering elliptic at spatial infinity for all $\sigma \in \Lambda_\beta$.
\end{rmk}

\subsection{Dynamics of the bicharacteristic flow}
\label{section_hamiltonian_flow}

We now consider the dynamics of the Hamiltonian flow associated to $P_\beta(\sigma)$. In $\{r<R_0\}$ the operator takes the form \eqref{operator_full}. Note that away from complex scaling the formal adjoint satisfies $P_\beta(\sigma)^* = P_\beta(\Bar{\sigma})$. Thus, in this region, the principal symbols of $P_\beta(\sigma)$ and $P_\beta(\sigma)^*$ agree.

To simplify the formulas, we study instead the operators $\rkerr^2P_\beta(\sigma)$ and $\rkerr^2P_\beta(\sigma)^*$. Since $\rkerr^2$ is just a bounded smooth function in $\{r<R_0\}$, this will not influence the estimates we obtain. The principal symbol is
$$p = \sigma_2\bigl(\rkerr^2P_\beta(\sigma)\bigr) = \mu\xi^2 + \eta^2 + \frac{\nu^2}{\sin^2(\theta)} + 2a\xi\nu,$$
and its Hamiltonian vector field is given by
$$H_p = 2\bigl(\mu\xi + a\nu\bigr)\partial_r + 2\eta\partial_\theta + 2\Bigl(\frac{\nu}{\sin^2(\theta)} + 2a\xi\Bigr)\partial_{\varphi_*} -2(r-m)\xi^2\partial_\xi + 2\frac{\cos(\theta)}{\sin^3(\theta)}\nu^2\partial_\eta.$$

The principal symbol is conserved under the Hamiltonian flow, and evidently $H_p$ annihilates $\nu$. Slightly less immediate is the annihilation of the expression $\eta^2 + \frac{\nu^2}{\sin^2(\theta)}$, leading to the following conserved quantities for the flow of $H_p$:
$$p, \quad \nu, \quad \kappa = \eta^2 + \frac{\nu^2}{\sin^2(\theta)}.$$

Denote the characteristic set by $\Sigma = \{p=0\}$. Using Young's inequality we find that on $\Sigma$:
$$a^2\sin^2(\theta)\xi^2 + \frac{\nu^2}{\sin^2(\theta)} \geq |2a\xi\nu| = |\mu\xi^2 + \eta^2 + \frac{\nu^2}{\sin^2(\theta)}| \geq \mu\xi^2 +\eta^2 + \frac{\nu^2}{\sin^2(\theta)}.$$
In particular, $a^2\sin^2(\theta)\xi^2 \geq \mu\xi^2$ on $\Sigma$, so the characteristic set is contained in the ergoregion:
$$\Sigma \subset \{\mu(r) - a^2\sin^2(\theta) \leq 0\} \subset \{r \leq 2m\}.$$

Notice that $\xi=0$ and $p=0$ imply that $\eta=\nu=0$. Since the zero section is excluded from $\Sigma$, we have $\xi\neq 0$ on $\Sigma$. Thus, the characteristic set consists of two connected components
$$\Sigma = \Sigma_+\cup\Sigma_-, \quad\text{where}\quad \Sigma_+ = \Sigma\cap\{\xi>0\}, \quad \Sigma_- = \Sigma\cap\{\xi<0\}.$$
In fact, another application of Young's inequality leads to a stronger statement. On $\Sigma$, we have
$$2a^2\sin^2(\theta)\xi^2 + \frac{1}{2}\frac{\nu^2}{\sin^2(\theta)} \geq |2a\xi\nu| = |\mu\xi^2 + \eta^2 + \frac{\nu^2}{\sin^2(\theta)}| \geq \eta^2 + \frac{\nu^2}{\sin^2(\theta)} - |\mu|\xi^2.$$
Using that $\mu$ is bounded on $\Sigma$, this shows that for some constant $c>0$:
\begin{equation}
\label{xi_bounded_below}
    \xi^2 > c\Bigl(\eta^2 + \frac{\nu^2}{\sin^2(\theta)}\Bigr) \quad\text{on}\,\,\Sigma.
\end{equation}

Notice that our analysis above used the coordinates $\theta,{\varphi_*}$ on $\Sph^2$, which are ill-defined at the poles of the sphere, i.e. at $\theta = 0,\pi$. However, both $\nu$ and $\kappa$ extend to smooth functions on all of $T^*\Sph^2$. Indeed, using coordinates 
$$u=\sin(\theta)\cos(\varphi_*),\quad w=\sin(\theta)\sin(\varphi_*)$$
near either of the poles, we have
$$\nu = u\xi_w - w\xi_u, \quad \kappa = \xi_u^2 + \xi_w^2 - (u\xi_u + w\xi_w)^2,$$
where $\xi_u, \xi_w$ are the associated coordinates on the fibers of the cotangent bundle. Thus, $\nu$ vanishes at the poles, while $\kappa$ is strictly positive on $T^*\Sph^2$ away from the zero section. The inequality \eqref{xi_bounded_below} extends by continuity to the poles in the form $\xi^2 > c\kappa$. Note that at the poles, the principal symbol becomes $p = \mu\xi^2 + \kappa$, and the only characteristic set over the poles is located at $\mu = 0, \kappa = 0$, i.e. at the radial sets over the black hole horizon, see Lemma \ref{lemma_radial_set} below. In the following, we will use the functions $\kappa$ and $\nu$ in our analysis, so that all statements apply also to the poles of $\Sph^2$.

Following these preliminary observations, we now show that the characteristic set contains a radial source and a radial sink, located over the black hole horizon. Furthermore, outside these radial sets, the Hamiltonian flow tends towards the source or sink in one direction and towards the boundary at $r=r_0$ in the other direction.

\begin{lemma}
\label{lemma_radial_set}
    Let $\Lambda_+ \subset \Sigma_+$ and $\Lambda_- \subset \Sigma_-$ be defined by
    $$\Lambda_+ = \{\mu=0,\, \kappa=0,\, \xi>0\}, \quad \Lambda_- = \{\mu=0,\, \kappa=0,\, \xi<0\}.$$
    Then $\Lambda_+$ is a radial source and $\Lambda_-$ a radial sink for the Hamiltonian flow of $p$, in the sense of Definition \ref{def_source_sink}.
\end{lemma}
\begin{proof}
    Note that on $\Sigma$, $\kappa=0$ implies $\mu=0$. Thus, we have
    $$\Lambda_{\pm} = \{\kappa=0\}\cap\Sigma_{\pm}.$$
    Since $\kappa$ is a conserved quantity, $\Lambda_\pm$ is an invariant submanifold for the flow of $H_p$.
    
    By \eqref{xi_bounded_below}, $\rho_r = |\xi|^{-1}$ is a well-defined smooth function in a conic neighborhood of $\Sigma$, which is elliptic there, and we have
    $$\rho_rH_p\rho_r\bigr|_{\Lambda_+} = 2(r_+-m)\rho_r, \quad \rho_rH_p\rho_r\bigr|_{\Lambda_-} = -2(r_+-m)\rho_r.$$
    Thus, the first condition in Definition \ref{def_source_sink} is satisfied with $$\alpha_r = 2(r_+-m) = 2\sqrt{m^2-a^2}.$$

    As a homogeneous degree $0$ quadratic defining function of $\Lambda_\pm$ within $\Sigma_\pm$, we take
    $$\rho_t = \frac{\eta^2}{\xi^2} + \frac{\nu^2}{\sin^2(\theta)\xi^2} = \rho_r^2\kappa.$$
    This satisfies
    $$\rho_rH_p\rho_t = 2(H_p\rho_r)\rho_r^2\kappa = \pm4(r-m)\rho_t, \quad \text{on}\,\, \Sigma_\pm.$$
    Thus, the second condition in Definition \ref{def_source_sink} is in fact satisfied with no cubic error term.
\end{proof}

\begin{lemma}
\label{lemma_propagation}
    Let $\Lambda_\pm \subset U_\pm$ be any neighborhoods and let $\delta > 0$. For all $(x,\xi) \in \Sigma_+ \setminus \Lambda_+$, there exists $T_1,T_2>0$, such that $e^{-T_1H_p}(x,\xi) \in U_+$ and $e^{T_2H_p}(x,\xi) \in \{r<r_0+\delta\}$. Similarly, for all $(x,\xi) \in \Sigma_- \setminus \Lambda_-$, there exists $T_1,T_2>0$, such that $e^{T_1H_p}(x,\xi) \in U_-$ and $e^{-T_2H_p}(x,\xi) \in \{r<r_0+\delta\}$.
\end{lemma}
\begin{proof}
    We prove the statement concerning the $\Sigma_+$ component. Let $\rho_t, \rho_r$ be as in the proof of Lemma \ref{lemma_radial_set}. Denote $\rho_t(t) = \rho_t(\exp(tH_p)(x,\xi))$, and similarly for $\rho_r(t)$ and $r(t)$, where $(x,\xi)\in\Sigma_+\setminus\Lambda_+$. Then
    $$\frac{d}{dt}\rho_t(t) = H_p\rho_t(t) = 4(r(t)-m)\rho_r(t)\rho_t(t).$$
    Since the integral curve $\exp(tH_p)(x,\xi)$ remains in the characteristic set, we have $r(t)\leq 2m$ for all $t$. Note that $\kappa \neq 0$ at $(x,\xi)\notin\Lambda_\pm$ and remains constant along the integral curve. Thus, by \eqref{xi_bounded_below}, we have $\xi(t) \geq C$ on the integral curve, or $\rho_r(t) \leq C^{-1}$ for some $C>0$. So for some constant $c_1>0$ we find
    $$\frac{d}{dt}\rho_t(t) \leq c_1\rho_t(t),$$
    and by Grönwall's inequality this shows that 
    $$\rho_t(t) \leq \rho_t(0)e^{c_1 t}.$$
    Any open set $\Lambda_+ \subset U_+$ contains a set of the form $\{\rho_t < \varepsilon\}$ for some $\varepsilon>0$. So the inequality above shows that for $t$ large enough $\exp(-tH_p)(x,\xi) \in U_+$.

    For the other part of the statement, note that $H_p\rho_r = 2(r-m)\rho_r$, so $\rho_r$ is increasing along the Hamiltonian flow, and we can bound $\rho_r(t)\geq\rho_r(0)$. On the domain of definition of the integral curve, that is, before $\exp(tH_p)(x,\xi)$ reaches the boundary at $r=r_0$, we have $r(t)>r_0>m$. So for some constant $c_2>0$, we have
    $$\frac{d}{dt}\rho_t(t) \geq c_2\rho_t(t),$$
    and again by Grönwall's inequality we find
    $$\rho_t(t) \geq \rho_t(0)e^{c_2 t}.$$
    Now notice that $\frac{p}{\xi^2} = \mu + a\frac{\nu}{\xi} + \rho_t$. Thus, on the characteristic set, we have
    $$\rho_t = -\mu - 2a\frac{\nu}{\xi} \leq -\mu + 2a\Bigl|\frac{\nu}{\xi}\Bigr| \leq -\mu + 2a^2 + \frac{1}{2}\rho_t,$$
    where we applied Young's inequality and bounded $|\frac{\nu^2}{\xi^2}|$ by $\rho_t$. As long as $t$ is in the domain of definition of the integral curve, i.e. while $r(t)$ remains bounded away from $r_0$, we have
    $$\mu(r(t)) \leq  2a^2 - \frac{1}{2}\rho_t(t).$$
    Thus, as $t$ increases, eventually $\mu(r(t)) < \mu(r_0 + \delta)$, and therefore also $r(t)<r_0+\delta$, as $\mu(r)$ is strictly increasing for $r\in (r_0,\infty)$.

    The statement for $(x,\xi)\in \Sigma_-\setminus\Lambda_-$ follows analogously, with $\rho_t(t)$ exponentially decreasing there.
\end{proof}

In order to apply the radial estimates, Propositions \ref{high_reg_radial_estimate} and \ref{low_reg_radial_estimate}, we must calculate the threshold regularity, see \eqref{threshold_regularity}. Recall that we are considering the Hamiltonian flow for the rescaled operator $\rkerr^2P_\beta(\sigma)$. Away from complex scaling $P_\beta(\sigma)^* = P_\beta(\Bar{\sigma})$, so we have
\begin{equation*}
    \sigma_1\bigl(\frac{1}{2i}(\rkerr^2P_\beta(\sigma) - (\rkerr^2P_\beta(\sigma))^*)\bigr) = \sigma_1\bigl(\frac{1}{2i}(\rkerr^2P_\beta(\sigma) - \rkerr^2P_\beta(\Bar{\sigma}) + [P_\beta(\Bar{\sigma}),\rkerr^2])\bigr).
\end{equation*}
Note that the principal symbol of $[P_\beta(\Bar{\sigma}),\rkerr^2]$ vanishes at the radial sets. Restricting to the radial sets, we find
\begin{equation*}
    \sigma_1\bigl(\frac{1}{2i}(\rkerr^2P_\beta(\sigma) - (\rkerr^2P_\beta(\sigma))^*)\bigr)\bigl|_{\Lambda_\pm} = -2\Im(\sigma)(r_+^2+a^2)\xi = \pm \alpha_s\alpha_r\rho_r^{-1},
\end{equation*}
with the threshold regularity given in terms of the black hole parameters as
\begin{equation}
\label{threshold_reg_kerr}
    \alpha_s = -2\Im(\sigma)\bigl(m+\frac{m}{\sqrt{m^2-a^2}}\bigr).
\end{equation}

\subsection{Fredholm estimates for the complex scaled operator}
\label{section_fredholm_estimates}

We are now ready to prove the Fredholm estimates for $P_\beta(\sigma)$ and its adjoint.
\begin{prop}
\label{prop_fredholm_estimates}
    Let $\sigma\in\C\setminus\{0\}$ satisfy $\arg(\sigma)\in(-\beta,\pi-\beta)$ and let $s > \frac{1}{2} - \alpha\Im(\sigma)$, where $\alpha = 2\bigl(m + \frac{m^2}{\sqrt{m^2-a^2}}\bigr)$. Then the following estimates hold for any $N\in\N$ and some $C = C_{s,N} > 0$:
    \begin{align}
        \|u\|_{\Bar{H}^s(X_\beta)} &\leq C\bigl(\|P_\beta(\sigma)u\|_{\Bar{H}^{s-1}(X_\beta)} + \|u\|_{\Bar{H}^{-N,-1}(X_\beta)}\bigr), \quad \forall u\in\mathcal{X}^s_\beta, \label{semi_fredholm_P}\\
        \|v\|_{\Dot{H}^{-s+1}(\Bar{X}_\beta)} &\leq C\bigl(\|P_\beta(\sigma)^*v\|_{\Dot{H}^{-s}(\Bar{X}_\beta)} + \|v\|_{\Dot{H}^{-N,-1}(\Bar{X}_\beta)}\bigr), \quad \forall v\in\mathcal{Y}^{-s+1}_\beta. \label{semi_fredholm_adjoint}
    \end{align}
\end{prop}
\begin{proof}
    Consider first $P_\beta(\sigma)$. In $\Sigma_+$, we will propagate estimates forward along the Hamiltonian flow from the radial source towards the boundary at $r=r_0$, while in $\Sigma_-$ we propagate estimates backward along the flow from the radial sink towards $r=r_0$.

    Since by assumption $s$ is above the threshold regularity, see \eqref{threshold_reg_kerr}, we can apply the high regularity radial estimates, Proposition \ref{high_reg_radial_estimate}, to find $B^+$ and $B^-$ with $\Lambda_\pm \subset \Ell(B^\pm)$ satisfying
    $$\|B^\pm u\|_{H^s} \leq C\bigl(\|\chi^\pm\rkerr^2P_\beta(\sigma)u\|_{H^{s-1}} + \|\chi^\pm u\|_{H^{-N}}\bigr),$$
    for some $\chi^\pm \in C^\infty_c(X)$. Note the presence of $\rkerr^2$, which stems from examining the flow associated to $\rkerr^2P_\beta(\sigma)$ instead of $P_\beta(\sigma)$. However, $\rkerr^2$ is just a bounded smooth function on the compact sets $\supp(\chi^\pm)$, so 
    $$\|\chi^\pm\rkerr^2P_\beta(\sigma)u\|_{H^{s-1}} \leq C \|\chi^\pm P_\beta(\sigma)u\|_{H^{s-1}}.$$
    We will use this fact without further comment in the estimates below.

    Let now $\chi \in C^\infty_c(X_\beta)$ be any cutoff function with $\supp(\chi) \subset \{r<R_0\}$, i.e. supported away from the region where complex scaling takes place. We denote 
    $$V = \{(x,\xi)\in T^*X_\beta\setminus\{0\} \,\,|\,\, x\in\supp(\chi)\}$$
    and set $U^\pm_0 = \Ell(B^\pm)$. By Lemma \ref{lemma_propagation}, the Hamiltonian flow started from any point in $\Sigma_\pm\cap V$ eventually enters the open set $U^\pm_0$ in the forward, respectively backward, direction. By the continuous dependence of the flow on the initial point, we can find conic open covers $U^+_1,\dots U^+_{k^+}$ of $\Sigma_+\cap V$ and $U^-_1,\dots U^-_{k^-}$ of $\Sigma_-\cap V$ such that, for any $j$, all $(x,\xi)\in U^+_j$ enter $U^+_0$ in finite time when propagated forward along the flow and all $(x,\xi)\in U^-_j$ enter $U^-_0$ in finite time when propagated backward. Taking $\Tilde{U} \subset \Ell(P_\beta(\sigma))$, we then have a conic open cover of $V$:
    $$U^+_0,\dots,U^+_{k^+}, U^-_0,\dots,U^-_{k^-}, \Tilde{U}.$$
    By Proposition \ref{microlocal_partition_of_unity} there is a microlocal partition of unity
    $$A^+_0,\dots,A^+_{k^+}, A^-_0,\dots,A^-_{k^-}, \Tilde{A} \in \Psi^0(X_\beta)$$
    consisting of compactly supported operators and satisfying 
    $$\WF(A^\pm_j) \subset U^\pm_j,\quad \WF(\Tilde{A})\subset \Tilde{U},\quad \WF\bigl(Id - \textstyle\sum_j A^+_j - \textstyle\sum_j A^-_j - \Tilde{A}\bigr)\cap V = \emptyset.$$
    Let $\Tilde{\chi}\in C^\infty_c(X_\beta)$ be a cutoff function satisfying 
    $$\supp(A^\pm_j),\, \supp(\Tilde{A}) \subset \{\Tilde{\chi}=1\}\times\{\Tilde{\chi}=1\}, \quad \forall j,$$
    where we denote by $\supp(A)$ the support of the Schwartz kernel of $A$.

    Since $\WF(A^\pm_0) \subset \Ell(B^\pm)$, we can apply the microlocal elliptic estimate, Proposition \ref{elliptic_estimate}, to find
    $$\|A^\pm_0 u\|_{H^s} \leq C\bigl(\|B^\pm u\|_{H^s} +\|\Tilde{\chi} u\|_{H^{-N}}\bigr) \leq C\bigl(\|\Tilde{\chi} P_\beta(\sigma)u\|_{H^{s-1}} + \|\Tilde{\chi} u\|_{H^{-N}}\bigr).$$
    Since the Hamiltonian flow started from any point in $\WF(A^\pm_j)$ enters $\Ell(B^\pm)$ in finite time, we can apply the propagation estimate, Proposition \ref{propagation_estimate}, for each $j\geq 1$ to find
    $$\|A^\pm_j u\|_{H^s} \leq C\bigl(\|\Tilde{\chi} P_\beta(\sigma)u\|_{H^{s-1}} + \|B^\pm u\|_{H^s} +\|\Tilde{\chi} u\|_{H^{-N}}\bigr) \leq C\bigl(\|\Tilde{\chi} P_\beta(\sigma)u\|_{H^{s-1}} + \|\Tilde{\chi} u\|_{H^{-N}}\bigr).$$
    Finally, $\WF(\Tilde{A})\subset\Ell(P_\beta(\sigma))$, so the microlocal elliptic estimate gives
    $$\|\Tilde{A}u\|_{H^s} \leq C\bigl(\|\Tilde{\chi} P_\beta(\sigma)u\|_{H^{s-2}} + \|\Tilde{\chi} u\|_{H^{-N}}\bigr) \leq C\bigl(\|\Tilde{\chi} P_\beta(\sigma)u\|_{H^{s-1}} + \|\Tilde{\chi} u\|_{H^{-N}}\bigr).$$

    Combining these estimates, we have
    \begin{equation}
    \label{estimate_near_horizon}
    \begin{split}
        \|\chi u\|_{H^s} &\leq \sum_j\|\chi A^+_j u\|_{H^s} + \sum_j\|\chi A^-_j u\|_{H^s} + \|\chi \Tilde{A}u\|_{H^s} + \|\chi\bigl(Id - \textstyle\sum_j A^+_j - \textstyle\sum_j A^-_j - \Tilde{A}\bigr)u\|_{H^s} \\
        &\leq C\bigl(\|\Tilde{\chi} P_\beta(\sigma)u\|_{H^{s-1}} + \|\Tilde{\chi} u\|_{H^{-N}}\bigr),
    \end{split}
    \end{equation}
    where we used that
    $$\chi\bigl(Id - \textstyle\sum_j A^+_j - \textstyle\sum_j A^-_j - \Tilde{A}\bigr) = \chi\bigl(Id - \textstyle\sum_j A^+_j - \textstyle\sum_j A^-_j - \Tilde{A}\bigr)\Tilde{\chi},$$
    and that the wavefront set of this operator is empty.

    The above estimate concerns functions supported in a fixed compact subset of $X_\beta$. Modulo a change of constant, we can hence apply it to the norms $\|\cdot\|_{\Bar{H}^s(X_\beta)}$. We now take a cutoff function $\psi\in C^\infty(X_\beta)$ supported in $\{r>r_0+\delta\}$ and satisfying $\psi=1$ on $\{r>r_0+2\delta\}$. Let further $\psi_0 \in C^\infty(X_\beta)$ satisfy $\supp(\psi_0) \subset \{r>R_0\}$ and $\psi_0 = 1$ near $\{r>R_1\}$. Then we have
    $$\|\psi u\|_{\Bar{H}^s(X_\beta)} \leq \|\psi(1-\psi_0) u\|_{\Bar{H}^s(X_\beta)} + \|\psi_0 u\|_{\Bar{H}^s(X_\beta)}.$$
    Taking $\psi(1-\psi_0)$ as the cutoff function $\chi$ in \eqref{estimate_near_horizon}, we can estimate
    \begin{equation*}
    \begin{split}
        \|\psi(1-\psi_0) u\|_{\Bar{H}^s(X_\beta)} &\leq C\bigl(\|\Tilde{\chi} P_\beta(\sigma)u\|_{\Bar{H}^{s-1}(X_\beta)} + \|\Tilde{\chi} u\|_{\Bar{H}^{-N}(X_\beta)}\bigr) \\
        &\leq C\bigl(\|\Tilde{\chi} P_\beta(\sigma)u\|_{\Bar{H}^{s-1}(X_\beta)} + \|\Tilde{\chi} u\|_{\Bar{H}^{-N,-1}(X_\beta)}\bigr),
    \end{split}
    \end{equation*}
    where we used that, since $\Tilde{\chi}$ is compactly supported, $\|\Tilde{\chi} u\|_{\Bar{H}^{-N}(X_\beta)} \leq C\|r^{-1}\Tilde{\chi} u\|_{\Bar{H}^{-N}(X_\beta)}$.
    Applying Lemma \ref{lemma_scattering_elliptic}, we find
    \begin{equation*}
    \begin{split}
        \|\psi_0 u\|_{\Bar{H}^s(X_\beta)} &\leq C\bigl(\|\Tilde{\psi} P_\beta(\sigma) u\|_{\Bar{H}^{s-2}(X_\beta)} + \|\Tilde{\psi} u\|_{\Bar{H}^{-N,-1}(X_\beta)}\bigr) \\
        &\leq C\bigl(\|\Tilde{\psi} P_\beta(\sigma) u\|_{\Bar{H}^{s-1}(X_\beta)} + \|\Tilde{\psi} u\|_{\Bar{H}^{-N,-1}(X_\beta)}\bigr),
    \end{split}
    \end{equation*}
    where we take $\Tilde{\psi}$ supported away from the boundary at $r=r_0$ and satisfying $\Tilde{\psi}=1$ on $\{r>r_0+\delta\}$. Altogether, we have
    \begin{equation}
    \label{estimate_unclosed}
        \|\psi u\|_{\Bar{H}^s(X_\beta)} \leq C\bigl(\|\Tilde{\psi} P_\beta(\sigma) u\|_{\Bar{H}^{s-1}(X_\beta)} + \|\Tilde{\psi} u\|_{\Bar{H}^{-N,-1}(X_\beta)}\bigr).
    \end{equation}

    This is almost the desired estimate. However, in order to control $u$ on some set, the microlocal estimates require the terms on the right-hand side to be controlled on a slightly larger set. We can close the estimate by applying the results on strictly hyperbolic operators of Section \ref{section_hyperbolic_estimates}.
    
    Notice that $P_\beta(\sigma)$ is strictly hyperbolic with respect to $r$ on $\{r_0\leq r < r_+\}$, in the sense of Definition \ref{def_hyperbolic}. Indeed, for $r\in [r_0,r_+)$, we have $\mu<0$, while $\kappa > 0$ on $T^*\Sph^2$ away from the zero section. Thus, viewing the principal symbol
    $$p = \mu\xi^2 + 2\nu\xi + \kappa$$
    as a polynomial in $\xi$, we find the two distinct real roots
    $$\xi_\pm = \frac{1}{\mu}\bigl(\nu \pm \sqrt{\nu^2 - \mu\kappa}\bigr).$$

    Choosing $\delta$ small enough, we have $\supp(1-\psi) \subset \{r<r_+\}$. Applying the hyperbolic estimate, Proposition \ref{prop_hyperbolic_estimate}, we find
    $$\|(1-\psi)u\|_{\Bar{H}^s(X_\beta)} \leq C\bigl(\|P_\beta(\sigma)u\|_{\Bar{H}^{s-1}(X_\beta)} + \|\psi u\|_{\Bar{H}^s(X_\beta)}\bigr).$$
    Together with the estimate for $\|\psi u\|_{\Bar{H}^s(X_\beta)}$ in \eqref{estimate_unclosed}, the first part of the Proposition, i.e. \eqref{semi_fredholm_P}, follows.

    We now consider the adjoint $P_\beta(\sigma)^*$. Recall that the principal symbol, and thus the Hamiltonian flow, for $P_\beta(\sigma)$ and its adjoint agree away from complex scaling. The proof is in the same vein as for $P_\beta(\sigma)$. However, we will propagate estimates in the opposite direction, that is, from a neighborhood of the boundary at $r=r_0$ backward along the Hamiltonian flow towards the radial source in $\Sigma_+$ and forward along the flow towards the radial sink in $\Sigma_-$. Note that $-s+1 < \frac{1}{2} + \alpha\Im(\sigma)$ is below the threshold regularity for $P_\beta(\sigma)^*$ at the radial sets. We can thus use the low regularity radial estimates, Proposition \ref{low_reg_radial_estimate}, to propagate estimates into the radial sets.

    Let $\psi \in C^\infty(X_\beta)$ be as above, i.e. $\supp(\psi)\in\{r>r_0+\delta\}$ and $\psi=1$ for $r>r_0+2\delta$. Then $1-\psi$ is supported in the region where $P_\beta(\sigma)^*$ is strictly hyperbolic and Proposition \ref{prop_hyperbolic_estimate} gives
    $$\|(1-\psi)v\|_{\Dot{H}^{-s+1}(\Bar{X}_\beta)} \leq C \|P_\beta(\sigma)^*v\|_{\Dot{H}^{-s}(\Bar{X}_\beta)}, \quad \forall v\in \mathcal{Y}^{-s+1}_\beta,$$
    where no error term is present, since $v$ lies in the space of supported distributions.

    By Lemma \ref{lemma_propagation}, the Hamiltonian flow started from any point in $\Sigma_+\setminus\Lambda_+$ eventually enters the region $\{\psi=0\}$, i.e. the elliptic set of $1-\psi$, and similarly for $\Sigma_-\setminus\Lambda_-$ when flowing backward. As above, we can thus cover the characteristic set, away from the radial source and sink, by open sets entering $\Ell(1-\psi)$ in finite time under the Hamiltonian flow. The propagation estimate of Proposition \ref{propagation_estimate} then allows us to control the $H^s$ norm of $u$ microlocally there in terms of $\|(1-\psi)v\|_{\Dot{H}^{-s+1}(\Bar{X}_\beta)}$.

    Away from the characteristic set, we use the microlocal elliptic estimates of Proposition \ref{elliptic_estimate} and, in the complex scaling region, we once again use Lemma \ref{lemma_scattering_elliptic}. This gives the desired estimate \eqref{semi_fredholm_adjoint} microlocally away from the radial source and sink. Finally, we use the low regularity radial estimate of Proposition \ref{low_reg_radial_estimate} to propagate the estimate into the radial sets. 
\end{proof}

The Fredholm property of $P_\beta(\sigma)$ now follows almost immediately from the estimates of Proposition \ref{prop_fredholm_estimates}. Since the formulation in terms of the spaces $\mathcal{X}^s_\beta$ and $\mathcal{Y}^s_\beta$ may not be quite standard, we provide a brief proof.

\begin{proof}[Proof of Proposition \ref{fredholm_prop}]
    $P_\beta(\sigma)$ clearly defines an analytic family of bounded operators from $\mathcal{X}^s_\beta$ to $\Bar{H}^{s-1}(X_\beta)$, being a polynomial in $\sigma$ whose coefficients are bounded operators. 
    
    The condition $\sigma \in \Lambda_\beta \cap \Omega_s$ is precisely the range where Proposition \ref{prop_fredholm_estimates} holds. The space $\Bar{H}^s(X_\beta)$ is compactly included in $\Bar{H}^{-N,-1}(X_\beta)$ and, since the inclusion $\mathcal{X}^s_\beta \xhookrightarrow{} \Bar{H}^s_\beta(X_\beta)$ is continuous, the same holds for $\mathcal{X}^s_\beta$. Thus, the estimate \eqref{semi_fredholm_P} shows, using a standard compactness argument, that $\mathrm{im}_{\mathcal{X}^s_\beta}(P_\beta(\sigma))$ is closed and $\ker_{\mathcal{X}^s_\beta}(P_\beta(\sigma))$ is finite-dimensional. Similarly, the estimate \eqref{semi_fredholm_adjoint} shows that $\ker_{\mathcal{Y}^{-s+1}_\beta}(P_\beta(\sigma)^*)$ is finite-dimensional. However, since $P_\beta(\sigma)^*: \mathcal{Y}^{-s+1}_\beta \to \Dot{H}^{-s}(\Bar{X}_\beta)$ is not quite the dual of ${P_\beta(\sigma): \mathcal{X}^s_\beta \to \Bar{H}^s(X_\beta)}$, we cannot conclude yet.

    Denote $K = \ker_{\mathcal{Y}^{-s+1}_\beta}(P_\beta(\sigma)^*)$. We claim that if $f \in \Bar{H}^{s-1}(X_\beta)$ satisfies $\pair{v}{f} = 0$ for all $v\in K$, then $f = P_\beta(\sigma)u$ for some $u\in\Bar{H}^s(X_\beta)$, and hence $f \in \mathrm{im}_{\mathcal{X}^s_\beta}(P_\beta(\sigma))$. This implies that $\mathrm{Ann}(K) \subset \mathrm{im}_{\mathcal{X}^s_\beta}(P_\beta(\sigma))$, where $\mathrm{Ann}(K)$ is the annihilator of $K$ in $\Bar{H}^{s-1}(X_\beta)$, and thus there is a surjection 
    $$\Bar{H}^{s-1}(X_\beta) / \mathrm{Ann}(K) \to \Bar{H}^{s-1}(X_\beta) / \mathrm{im}_{\mathcal{X}^s_\beta}(P_\beta(\sigma)).$$ Since $\Bar{H}^{s-1}(X_\beta) / \mathrm{Ann}(K) \simeq K^*$, this shows that the cokernel is finite-dimensional.

    To see the claim, let $L \subset \mathcal{Y}^s_\beta$ be a closed complementary subspace to $K$. Then a standard argument by contradiction shows that 
    $$\|v\|_{\Dot{H}^{-s+1}(\Bar{X}_\beta)} \leq \|v\|_{\mathcal{Y}^{-s+1}_\beta} \leq C\|P_\beta(\sigma)^*v\|_{\Dot{H}^{-s}(\Bar{X}_\beta)}, \quad \forall v\in L.$$
    Thus, the linear functional $P_\beta(\sigma)^*v \to \pair{v}{f}$ is well-defined on $\mathrm{im}_{\mathcal{Y}^{-s+1}_\beta}(P_\beta(\sigma)^*)$ and satisfies
    $$\bigl|\pair{v}{f}\bigr| = \bigl|\pair{v_L}{f}\bigr| \leq C\|P_\beta(\sigma)^*v_L\|_{\Dot{H}^{-s}(\Bar{X}_\beta)} = C\|P_\beta(\sigma)^*v\|_{\Dot{H}^{-s}(\Bar{X}_\beta)} \quad \forall v\in\mathcal{Y}^{-s+1}_\beta,$$
    where we wrote $v=v_L+v_K$ with $v_L\in L$ and $v_K\in K$.
    Applying the Hahn-Banach theorem, this functional can be extended to all of $\Dot{H}^{-s}(X_\beta)$ and is represented by an element $u \in \Dot{H}^{-s}(X_\beta)^* = \Bar{H}^s(X_\beta)$. Since $C^\infty_c(X_\beta) \subset \mathcal{Y}^{-s+1}_\beta$, we then find 
    $$\pair{v}{f} = \pair{P_\beta(\sigma)^*v}{u} = \pair{v}{P_\beta(\sigma)u}, \quad \forall v\in C^\infty_c(X_\beta),$$ showing $f=P_\beta(\sigma)u$ by density.
\end{proof}

\begin{rmk}
\label{remark_resonant_states_smooth}
Notice that Proposition \ref{prop_fredholm_estimates} also implies the smoothness of resonant states, that is, of elements in $\ker_{\mathcal{X}^s_\beta}(P_\beta(\sigma))$ as long as $s$ is above the threshold regularity. Indeed, the estimate \eqref{semi_fredholm_P} holds in the strong sense that if $s > \frac{1}{2} - \alpha\Im(\sigma)$ and the right hand side is finite, then the left hand side is finite. Thus, if $\sigma \in \Lambda_\beta$, $u \in \Bar{H}^{s'}(X_\beta)$ for some $s' > \frac{1}{2} - \alpha\Im(\sigma)$ and $P_\beta(\sigma)u = 0$, then \eqref{semi_fredholm_P} shows that $u\in\Bar{H}^s(X_\beta)$ for all $s$, and hence $u \in \Bar{C}^\infty(X_\beta)$. This also implies that the space of resonant states, $\ker(P_\beta(\sigma))$, is independent of the regularity of the space $\mathcal{X}^s_\beta$ on which we let $P_\beta(\sigma)$ act.
\end{rmk}

\subsection{Definition of quasinormal modes}
\label{section_def_QNM}
In this section, we first provide a definition of quasinormal modes for the Kerr black hole in terms of the complex scaled operators $P_\beta(\sigma)$ and then prove Theorem \ref{thm_analytic_continuation}, showing that quasinormal modes can also be obtained as poles of the analytically continued cutoff resolvent. We will use the Fredholm property proved in Proposition \ref{fredholm_prop} and the analytic Fredholm theorem to show that $P_\beta(\sigma)$ is invertible for $\sigma \in \Lambda_\beta$ at all but a discrete set of points. 

\begin{prop}
\label{prop_qnm}
    For any $\beta \in (-\pi,\pi)$ and $s\in\R$, let $\Lambda_\beta$ and $\Omega_s$ be as in Proposition \ref{fredholm_prop}. Then the analytic family of Fredholm operators
    $$\sigma \in \Lambda_\beta\cap\Omega_s \longrightarrow P_\beta(\sigma): \mathcal{X}^s_\beta \to \Bar{H}^{s-1}(X_\beta)$$
    has index $0$ and is invertible away from a closed discrete subset of $\Lambda_\beta\cap\Omega_s$, giving rise to a meromorphic family of operators
    $$\sigma \in \Lambda_\beta\cap\Omega_s \longrightarrow P_\beta(\sigma)^{-1}: \Bar{H}^{s-1}(X_\beta) \to \Bar{H}^{s}(X_\beta)$$
    with poles of finite rank.
\end{prop}

In the following, we will define the quasinormal modes contained in $\Lambda_\beta$ as the set of poles of the inverse $P_\beta(\sigma)^{-1}$. In order for this definition to make sense, we must show that the set of poles agrees for different values of $\beta$. The fundamental result in this direction is the following Proposition, which makes use of Lemma \ref{deformation_lemma}. The proof essentially follows \cite[Lemma 3.4]{sjostrand_zworski_complex_scaling}, but we include it here for completeness.

\begin{prop}
\label{prop_qnm_scaling_independent}
    Let $\beta_1,\beta_2 \in [0,\pi)$ or $\beta_1,\beta_2 \in (-\pi,0]$. Let further $\sigma \in \Lambda_{\beta_1} \cap \Lambda_{\beta_2}$ and $s > \frac{1}{2} - \alpha\Im(\sigma)$. Then 
    $$\dim(\ker_{\mathcal{X}^s_{\beta_1}}(P_{\beta_1}(\sigma))) = \dim(\ker_{\mathcal{X}^s_{\beta_2}}(P_{\beta_2}(\sigma))).$$
\end{prop}
\begin{proof}
    We prove the result for $|\beta_1-\beta_2|$ sufficiently small. That is, for any $\beta_1$ with $\sigma \in \Lambda_{\beta_1}$, there is a $\delta = \delta_{\beta_1} > 0$, such that $\dim(\ker(P_{\beta_1}(\sigma))) = \dim(\ker(P_{\beta_2}(\sigma)))$ for all $\beta_2 \in (\beta_1-\delta,\beta_1+\delta)$. The result for general $\beta_1 < \beta_2$ follows by covering the interval $[\beta_1,\beta_2]$ by such open intervals.
    
    Assume without loss of generality that $\beta_1 < \beta_2$ and let $u_1 \in \Bar{H}^s(X_{\beta_1})$ satisfy $P_{\beta_1}(\sigma)u_1 = 0$. Note from Remark \ref{remark_resonant_states_smooth} that $u_1\in \Bar{C}^\infty(X_{\beta_1})$ and, in fact, $\ker(P_{\beta_1}(\sigma))$ has a well-defined meaning independent of the regularity $s$. We will first use Lemma \ref{deformation_lemma} to analytically continue $u_1$ in the region of $\C^3$ where $P(\sigma)$ defines an operator with analytic coefficients. Restricting to $X_{\beta_2}$ will provide an element $u_2 \in \Bar{C}^\infty(X_{\beta_2})$ satisfying $P_{\beta_2}(\sigma)u_2 = 0$. However, this does not control the growth of $u_2$ at infinity. Thus, in a second step, we will show that in fact $u_2 \in \Bar{H}^s(X_{\beta_2})$, finishing the proof. The opposite direction, where we start from an element in $\ker(P_{\beta_2}(\sigma))$ follows entirely analogously.

    Denote $\Omega = \{x\in\R^3 \,\,|\,\, |x|>R_0\}$ and let $\Gamma_{\beta_1} = F_{\beta_1}(\Omega)$, $\Gamma_{\beta_2} = F_{\beta_2}(\Omega)$. Choose $\delta$ small enough so that $\beta_1,\beta_2 \in [0,\pi-\delta)$ or $(-\pi+\delta,0]$ respectively. In either case, by Lemma \ref{analytic_coefficients}, $P(\sigma)$ defines a differential operator with analytic coefficients in an open subset $U\subset\C^3$, which includes both contours $\Gamma_{\beta_1}, \Gamma_{\beta_2}$.
    Lemma \ref{deformation_lemma} only applies to contour deformations that take place in a compact set. We will thus patch $u_2$ together from a family of such deformations. To this end, let 
    $$\Tilde{\chi} \in C^\infty_c(\R), \quad \Tilde{\chi}=1 \,\,\text{on}\,\, [1,2],\quad \supp(\Tilde{\chi}) \subset \bigl(\tfrac{1}{2}, 4\bigr), \quad 0\leq\Tilde{\chi}\leq 1.$$ Define the phase function $\phi^{\beta_1,\beta_2}_{R,s}: (0,\infty) \to \R$ by
    \begin{equation}
    \label{phase_function_deformation}
        \phi^{\beta_1,\beta_2}_{R,s}(r) = \phi_{\beta_1}(r) + s\Tilde{\chi}\bigl(\frac{r}{R}\bigr)(\phi_{\beta_2}(r) - \phi_{\beta_1}(r)) = \bigl(\beta_1 + s\Tilde{\chi}\bigl(\frac{r}{R}\bigr)(\beta_2 - \beta_1)\bigr)\psi(\log(r)),
    \end{equation}
    with $\phi_\beta(r) = \beta\psi(\log(r))$ as in Definition \ref{def_phase_function}. Assuming that $R>2R_0$, we set
    $$F^{\beta_1,\beta_2}_R: [0,1]\times \Omega \to \C^3, \quad F^{\beta_1,\beta_2}_R(s,x) = e^{i\phi^{\beta_1,\beta_2}_{R,s}(|x|)}x, \quad\text{and}\quad \Gamma^{\beta_1,\beta_2}_{R,s} = F^{\beta_1,\beta_2}_R(\{s\}\times\Omega).$$
    Then the contour $\Gamma^{\beta_1,\beta_2}_{R} := \Gamma^{\beta_1,\beta_2}_{R,1}$ coincides with $\Gamma_{\beta_2}$ in $\{R \leq |x|\leq 2R\}$ and with $\Gamma_{\beta_1}$ outside of $\{\frac{R}{2} \leq |x| \leq 4R\}$. The family of contours $\Gamma^{\beta_1,\beta_2}_{R,s}$ interpolates between $\Gamma^{\beta_1,\beta_2}_{R,0}=\Gamma_{\beta_1}$ and $\Gamma^{\beta_1,\beta_2}_{R}$. 

    Note that all the $\Gamma^{\beta_1,\beta_2}_{R,s}$ are contained in the open set $U$, where $P(\sigma)$ has analytic coefficients, and $P(\sigma)\bigr|_{\Gamma_{\beta_1}}u_1 = 0$ trivially extends to an analytic function on $U$. Choosing $\varepsilon$, the bound on $|\psi'|$, small enough, the same argument as in the proof of Lemma \ref{lemma_scattering_elliptic} shows that $P(\sigma)\bigr|_{\Gamma^{\beta_1,\beta_2}_{R,s}}$ is elliptic for all $s\in[0,1]$. Furthermore, for all $s$, $F^{\beta_1,\beta_2}_R(s,x)=F^{\beta_1,\beta_2}_R(0,x)$ outside the compact set $\{\frac{R}{2} \leq |x| \leq 4R\}$ and, in fact, the only points of intersection of the $\Gamma^{\beta_1,\beta_2}_{R,s}$ for different $s$ are contained in this set, where all the contours agree. Thus, we can apply Lemma \ref{deformation_lemma} to the family of contours $\Gamma^{\beta_1,\beta_2}_{R,s}$, see also Remark \ref{remark_multivalued}, and obtain an analytic function $\Tilde{u}^{\beta_1,\beta_2}_{R}$ defined in a neighborhood of $\bigcup_{s\in[0,1]}\Gamma^{\beta_1,\beta_2}_{R,s}$, such that $\Tilde{u}^{\beta_1,\beta_2}_{R}\bigr|_{\Gamma_{\beta_1}} = u_1$. 
    
    Notice that $\Tilde{u}^{\beta_1,\beta_2}_{R}$ automatically satisfies $P(\sigma)\Tilde{u}^{\beta_1,\beta_2}_{R} = 0$, since $P(\sigma)\Tilde{u}^{\beta_1,\beta_2}_{R}$ is an analytic function that vanishes on $\Gamma_{\beta_1}$. For the same reason, the analytic extensions $\Tilde{u}^{\beta_1,\beta_2}_{R}$ agree for different choices of $R$ in the intersection of their domains. We will denote by ${u^{\beta_1,\beta_2}_{R} \in C^\infty(\Gamma^{\beta_1,\beta_2}_{R})}$ the restriction $u^{\beta_1,\beta_2}_{R} = \Tilde{u}^{\beta_1,\beta_2}_{R}\bigr|_{\Gamma^{\beta_1,\beta_2}_{R}}$. Finally, we define the desired element $u_2 \in \Bar{C}^\infty(X_{\beta_2})$. Recall that $\phi_\beta(r) = 0$ for $r<R_1$, and thus the $\Gamma^{\beta_1,\beta_2}_{R,s}$ all agree in $\{|x|<R_1\}$. Choose $R_1>2R_0$ and define $u_2$ by
    \begin{equation}
    \label{u_2_construction}
        u_2(x) = u_1(x) \,\,\text{for}\,\, |x| < R_1, \quad u_2(x) = u^{\beta_1,\beta_2}_{2^jR_1}(x) \,\,\text{for}\,\, 2^jR_1\leq |x|\leq 2^{j+1}R_1.
    \end{equation}
    Note that $u_2 \in \Bar{C}^\infty(X_{\beta_2})$ is well-defined in this way and satisfies $P_{\beta_2}(\sigma)u_2 = 0$.

    We will now show that, in fact, $u_2 \in \Bar{H}^s(X_{\beta_2})$. To this end, consider first the operator $e^{-2i\beta_1}\bigl(\Delta - (e^{i\beta_1}\sigma)^2\bigr)$ on $\R^3$. Since by assumption $(e^{i\beta_1}\sigma)^2$ is in the resolvent set of the Laplacian, we have for all $v\in H^s(\R^3)$:
    $$\|v\|_{H^s(\R^3)} \leq C\|e^{-2i\beta_1}\bigl(\Delta - (e^{i\beta_1}\sigma)^2\bigr)v\|_{H^{s-2}(\R^3)}.$$
    Taking a family of cutoff function $\chi_R \in C^\infty_c(\R^3)$ with 
    $$\chi_R = 1 \,\,\text{for}\,\, |x| \in [\tfrac{R}{2},4R], \quad \supp(\chi_R) \subset \{|x| \in (\tfrac{R}{4},8R)\}$$
    and satisfying $|\nabla^j\chi_R| \leq c$ for all $j<s+1$ and some $c>0$ independent of $R$ for $R\geq R_1$, we find that for all $u \in C^\infty(\R^3)$:
    \begin{equation}
    \label{dyadic_cover_estimate}
    \begin{split}
        \|u\|_{H^s(\{\frac{R}{2}<|x|<4R\})} &\leq \|\chi_R u\|_{H^s(\R^3)} \leq C\|e^{-2i\beta_1}\bigl(\Delta - (e^{i\beta_1}\sigma)^2\bigr)\chi_R u\|_{H^{s-2}(\R^3)} \\
        &\leq C\bigl(\|\chi_R e^{-2i\beta_1}\bigl(\Delta - (e^{i\beta_1}\sigma)^2\bigr)u\|_{H^{s-2}(\R^3)} + \|[\Delta,\chi_R] u\|_{H^{s-2}(\R^3)}\bigr) \\
        &\leq C\bigl(\|e^{-2i\beta_1}\bigl(\Delta - (e^{i\beta_1}\sigma)^2\bigr)u\|_{H^{s-2}(\{\frac{R}{4}<|x|<8R\})} + \|u\|_{H^{s}(\{\frac{R}{4}<|x|<\frac{R}{2}\}\cup\{4R<|x|<8R\})}\bigr).
    \end{split}
    \end{equation}
    with $C$ independent of $R$.

    Denote by 
    $$P^{\beta_1,\beta_2}_{R}(\sigma) = F^{\beta_1,\beta_2}_{R}(1,\,\cdot\,)^*P(\sigma)\bigr|_{\Gamma^{\beta_1,\beta_2}_{R}}(F^{\beta_1,\beta_2}_{R}(1,\,\cdot\,)^{-1})^*$$
    the local coordinate expression of the restriction of $P(\sigma)$ to the contour $\Gamma^{\beta_1,\beta_2}_{R}$, viewed as a differential operator on $\Omega \subset \R^3$. Then $P^{\beta_1,\beta_2}_{R}(\sigma)$ asymptotically approaches $${e^{-2i\beta_1}\bigl(\Delta - (e^{i\beta_1}\sigma)^2\bigr)}$$ as $|x|\to\infty$. In fact, for any $\varepsilon>0$, we can choose $|\beta_2-\beta_1|$ small enough and $R$ large enough, so that for all $u \in C^\infty(\R^3)$:
    \begin{equation}
    \label{asymptotically_laplacian_estimate}
        \|\bigl(P^{\beta_1,\beta_2}_{R}(\sigma) - e^{-2i\beta_1}(\Delta - (e^{i\beta_1}\sigma)^2)\bigr)u\|_{H^{s-2}(\{\frac{R}{4}<|x|<8R\})} \leq \varepsilon \|u\|_{H^{s}(\{\frac{R}{4}<|x|<8R\})},
    \end{equation}
    with $\varepsilon$ independent of $R$. Indeed, the coefficients of the second order differential operator on the left hand side consist of terms that are of order $|x|^{-1}$ uniformly in $R$ and terms that are of order
    $$\bigl|e^{i\phi^{\beta_1,\beta_2}_{R,1}(r)} - e^{i\beta_1}\bigr| \leq \bigl|e^{i\beta_2} - e^{i\beta_1}\bigr|.$$
    Combining \eqref{dyadic_cover_estimate} and \eqref{asymptotically_laplacian_estimate}, splitting the norm on the right hand side of \eqref{asymptotically_laplacian_estimate} as
    $$\|u\|_{H^{s}(\{\frac{R}{4}<|x|<8R\})} = \|u\|_{H^{s}(\{\frac{R}{2}<|x|<4R\})} + \|u\|_{H^{s}(\{\frac{R}{4}<|x|<\frac{R}{2}\}\cup\{4R<|x|<8R\})},$$
    and absorbing $\varepsilon \|u\|_{H^{s-2}(\{\frac{R}{2}<|x|<4R\})}$ into the left hand side of \eqref{dyadic_cover_estimate}, we obtain for all $u\in C^\infty(\R^3)$ and $R$ large enough:
    \begin{equation*}
         \|u\|_{H^s(\{\frac{R}{2}<|x|<4R\})} \leq C\bigl(\|P^{\beta_1,\beta_2}_{R}(\sigma)u\|_{H^{s-2}(\{\frac{R}{4}<|x|<8R\})} + \|u\|_{H^{s}(\{\frac{R}{4}<|x|<\frac{R}{2}\}\cup\{4R<|x|<8R\})}\bigr).
    \end{equation*}

    Applying this estimate to $F^{\beta_1,\beta_2}_{R}(1,\,\cdot\,)^*u^{\beta_1,\beta_2}_{R}$, the pullback to $\Omega\subset\R^3$ of the analytically continued solution, which we continue to denote by $u^{\beta_1,\beta_2}_{R}$ for simplicity, we find
    \begin{equation*}
         \|u^{\beta_1,\beta_2}_{R}\|_{H^s(\{\frac{R}{2}<|x|<4R\})} \leq C \|u^{\beta_1,\beta_2}_{R}\|_{H^{s}(\{\frac{R}{4}<|x|<\frac{R}{2}\}\cup\{4R<|x|<8R\})}.
    \end{equation*}
    Since $u^{\beta_1,\beta_2}_{R}$ agrees with $u_1$ in $\{\frac{R}{4}<|x|<\frac{R}{2}\}\cup\{4R<|x|<8R\}$ and agrees with $u_2$ in $\{R<|x|<2R\}$, we have, uniformly in $R$ for all $R$ large enough,
    $$\|u_2\|_{H^s(\{R<|x|<2R\})} \leq C \|u_1\|_{H^{s}(\{\frac{R}{4}<|x|<\frac{R}{2}\}\cup\{4R<|x|<8R\})}.$$
    Applying this to $R = 2^kR_1$ for all $k\in\N$, as in the definition of $u_2$, shows that $u_2 \in \Bar{H}^s(X_{\beta_2})$.

    Note that analytically extending two different elements of $\ker(P_{\beta_1}(\sigma))$ must lead to two different elements of $\ker(P_{\beta_2}(\sigma))$, since otherwise one would end up with two different analytic functions agreeing on $\Gamma_{\beta_2}$. Thus, the dimension of the kernels are equal.
\end{proof}

\begin{rmk}
    A simpler version of the proof of Proposition \ref{prop_qnm_scaling_independent} shows that the dimension of the kernel of $P_\beta(\sigma)$ is also independent of the exact form used for complex scaling. If $P_\beta(\sigma)$ and $\Tilde{P}_\beta(\sigma)$ are the operators obtained from complex scaling with different phase functions, then $P_\beta(\sigma)$ and $\Tilde{P}_\beta(\sigma)$ agree outside a compact set, so the deformation result, Lemma \ref{deformation_lemma}, allows us to analytically continue an element of $\ker(P_\beta(\sigma))$ to an element of $\ker(\Tilde{P}_\beta(\sigma))$.
\end{rmk}

In order to apply the analytic Fredholm theorem, we must show that $P_\beta(\sigma)$ is at least invertible for some $\sigma$. This will follow from Proposition \ref{prop_qnm_scaling_independent} together with the analysis of the original Kerr spectral family $P_0(\sigma)$. It is fairly straightforward to show, using standard energy estimates for the Kerr wave equation, that $P_0(\sigma)$ is invertible for all $\sigma$ with $\Im(\sigma)$ large enough. A different argument, based on semiclassical analysis as in Section \ref{section_semiclassical}, can be obtained by adapting \cite[Section 7]{vasy_KdS} to the Kerr metric. However, for the Kerr wave equation a much stronger result is available, namely mode stability, which implies the invertibility of the Kerr spectral family in the upper half-plane. The mode stability of Kerr was first proved in \cite{whiting} and then extended to the closed upper half-plane in \cite{shlapentokh-rothman_mode_stability}. Note that both of these results only treat fully separated modes, i.e. show the absence of elements of $\ker(P(\sigma))$ of the form $u(r,\theta,\phi_*) = v(r)w(\theta)e^{im\phi_*}$. This is upgraded to mode stability in the sense used here in \cite[Theorem 1.7]{hintz_mode_stability}. (Since the operator underlying the ODE for $\theta$ is not self-adjoint when $\sigma \notin \R$, it is not a priori clear that any mode solution could be expanded into fully separated ones.)

\begin{prop}[\cite{whiting,shlapentokh-rothman_mode_stability,hintz_mode_stability}]
\label{mode_stability}
    For $s>\frac{1}{2}$, the original Kerr spectral family 
    $$P_0(\sigma): \mathcal{X}^s_0 \to \Bar{H}^{s-1}(X)$$
    is invertible on the upper half-plane $\Lambda_0 = \{\sigma\in\C \,\,|\,\, \Im(\sigma)>0\}$.
\end{prop}

Using the invertibility of $P_0(\sigma)$ in the upper half-plane, we can now prove Proposition \ref{prop_qnm}.

\begin{proof}[Proof of Proposition \ref{prop_qnm}]
    We first show that $P_\beta(\sigma)$ has Fredholm index zero. Due to its invertibility, $P_0(\sigma)$ certainly has index zero in $\Lambda_0$. So the claim essentially follows from the local constancy of the Fredholm index and the continuous dependence of $P_\beta(\sigma)$ on $\beta$. However, the situation is slightly more subtle, since the operators $P_\beta(\sigma)$ are defined on different spaces for each $\beta$. Recall that the domain is
    $$\mathcal{X}_\beta^s = \{u\in\Bar{H}^s(X_\beta) \,\,|\,\, P_\beta(0)u \in \Bar{H}^{s-1}(X_\beta)\}.$$
    Using the unitary operator $F_\beta^*$, we can identify $F_\beta^*\Bar{H}^s(X_\beta) = \Bar{H}^s(X)$ and the operators with their local coordinate expression $F_\beta^*P_\beta(\sigma)(F_\beta^*)^{-1}$, which we continue to denote $P_\beta(\sigma)$ by a slight abuse of notation. Away from complex scaling, i.e. in the ball $B_{R_1}$, the operators $P_\beta(\sigma)$ all agree for different $\beta$ and we have
    $$P_\beta(0)u\big|_{B_{R_1}} = P_0(0)u\big|_{B_{R_1}}, \quad \forall u \in \Bar{H}^s(X).$$
    Outside the ball $B_{R_0}$ the operators are elliptic, so
    $$P_\beta(0)u \in \Bar{H}^{s-1}\bigl(\R^3\setminus \overline{B_{R_0}}\bigr) \iff u \in \Bar{H}^{s+1}\bigl(\R^3\setminus \overline{B_{R_0}}\bigr) \iff  P_0(0)u \in \Bar{H}^{s-1}\bigl(\R^3\setminus \overline{B_{R_0}}\bigr).$$
    Thus, the spaces $\mathcal{X}_\beta^s$ are all the same, that is, $F_\beta^*\mathcal{X}_\beta^s = \mathcal{X}_0^s$.

    From the local coordinate expression for $P_\beta(\sigma)$ in \eqref{complex_scaled_operator}, we see that $\beta \to P_\beta(\sigma)$ is continuous in operator norm from $\Bar{H}^{s+1}(X)$ to $\Bar{H}^{s-1}(X)$. However, we need to show continuity in operator norm from $\mathcal{X}_0^s(X)$ to $\Bar{H}^{s-1}(X)$. Choosing a cutoff function $\chi \in C^\infty_c(X)$ with $\supp(\chi) \subset \{r>R_0\}$ and $\chi = 1$ near $\{r \geq R_1\}$, and using the fact that the $P_\beta(\sigma)$ agree on $B_{R_1}$, we find for all $u \in \mathcal{X}^s_0$ and all $\beta_1, \beta_2$ with $|\beta_1-\beta_2|$ sufficiently small
    \begin{equation*}
    \begin{split}
        \|\bigl(P_{\beta_1}(\sigma) - P_{\beta_2}(\sigma)\bigr)u\|_{\Bar{H}^{s-1}(X)} &= \|\bigl(P_{\beta_1}(\sigma) - P_{\beta_2}(\sigma)\bigr)\chi u\|_{\Bar{H}^{s-1}(X)} \\
        &\leq C|\beta_1-\beta_2|\|\chi u\|_{\Bar{H}^{s+1}(X)} \\
        &\leq C|\beta_1-\beta_2|\bigl(\|\Tilde{\chi}P_0(0)u\|_{\Bar{H}^{s-1}(X)} + \|\Tilde{\chi}u\|_{\Bar{H}^{s}(X)}\bigr) \\
        &\leq C|\beta_1-\beta_2| \|u\|_{\mathcal{X}^s_0},
    \end{split}
    \end{equation*}
    where we used the uniform ellipticity of $P_0(0)$ in $\{r>R_0\}$. This shows that the $P_\beta(\sigma)$ depend continuously on $\beta$ as operators from $\mathcal{X}^s_0$ to $\Bar{H}^{s-1}$ and the index zero property follows. Note that the unitary pullback map $F_\beta^*$ does not affect the index.

    For $\beta \in (-\pi,\pi)$, we choose $\sigma_0 \in \Lambda_\beta \cap \Lambda_0$. Note that the intersection is non-empty. Since $P_0(\sigma_0)$ is invertible, Proposition \ref{prop_qnm_scaling_independent} shows that $P_\beta(\sigma_0)$ has trivial kernel and, as its index is zero, the operator is invertible. By the analytic Fredholm theorem, see \cite[Theorem C.8]{dyatlov_zworski_scattering_book}, $P_\beta(\sigma)$ is thus invertible at all but a discrete set of points in $\Lambda_\beta \cap \Omega_s$ and $\sigma \to P_\beta(\sigma)^{-1}$ is meromorphic with poles of finite rank.
\end{proof}

We can now provide a rigorous definition of quasinormal modes for the Kerr spacetime as a discrete subset of logarithmic cover of the complex plane.
\begin{defn}
\label{def_qnm}
    The set of quasinormal modes of a Kerr black hole is the discrete subset
    $$\mathrm{QNM}_{m,a} \subset \bigl\{\sigma\in\Lambda \,\,|\,\, \arg(\sigma) \in (-\pi,2\pi)\bigr\},$$
    consisting of those $\sigma$ for which $\ker_{\mathcal{X}^s_\beta}(P_\beta(\sigma))$ is non-trivial for some $\beta$ with $\sigma \in \Lambda_\beta$ and some $s > \frac{1}{2} - \alpha\Im(\sigma)$.
\end{defn}

We now address the analytic continuation of the cutoff resolvent for the original Kerr spectral family. To this end, we first show that the action of the resolvent $P_\beta(\sigma)^{-1}$ away from complex scaling does not depend on $\beta$. Notice that the operators $\chi P_\beta(\sigma)^{-1}\chi$ can all be viewed as acting on the original undeformed space $\Bar{H}^{s-1}(X)$. We claim that they in fact agree on this space. This will follow in a similar manner as Proposition \ref{prop_qnm_scaling_independent}.

\begin{prop}
\label{prop_cutoffs_agree}
    Let $\beta_1,\beta_2 \in [0,\pi)$ or $\beta_1,\beta_2 \in (-\pi,0]$. Let further $\chi \in \Bar{C}^\infty(X)$ have compact support in $\Bar{X}$ and choose $R_1$, the start of complex scaling, large enough so that ${\supp(\chi) \subset \{r<R_1\}}$. Then for any ${\sigma \in \Lambda_{\beta_1}\cap\Lambda_{\beta_2}}$ satisfying $\sigma \notin \mathrm{QNM}_{m,a}$, we have
    $$\chi P_{\beta_1}(\sigma)^{-1}\chi = \chi P_{\beta_2}(\sigma)^{-1}\chi$$
    as operators on $\Bar{H}^{s-1}(X)$, where $s>\frac{1}{2}-\alpha\Im(\sigma)$.
\end{prop}
\begin{proof}
    Let $f \in \Bar{H}^{s-1}(X)$. Then $\chi f$ is supported away from complex scaling. We denote
    $$u_1 = P_{\beta_1}(\sigma)^{-1}\chi f \in \Bar{H}^s(X_{\beta_1}).$$
    Note that $P_{\beta_1}(\sigma)u_1$ is identically zero in $\{r>R_1\}$. By the ellipticity of $P_{\beta_1}(\sigma)$, $u_1$ is thus smooth in this region. By the exact same arguments as in Proposition \ref{prop_qnm_scaling_independent}, we can analytically continue $u_1$ in $\{r>R_1\}$ and obtain an element of $u_2 \in \Bar{H}^s(X_{\beta_2})$, agreeing with $u_1$ in $\{r<R_1\}$ and satisfying $P_{\beta_2}(\sigma)u_2 = \chi f$. The injectivity of $P_{\beta_2}(\sigma)$ for $\sigma \notin \mathrm{QNM}_{m,a}$ then shows that we must have
    $$u_2 = P_{\beta_2}(\sigma)^{-1}\chi f.$$
    Since $\chi u_2 = \chi u_1$, the proposition follows.
\end{proof}

\begin{rmk}
    Note that for $\beta_1 \in [0,\pi)$ and $\beta_2 \in (-\pi,0]$ the cutoff resolvents satisfy
    $$\chi P_{\beta_1}(\sigma)^{-1}\chi = \chi P_0(\sigma)^{-1}\chi = \chi P_{\beta_2}(\sigma)^{-1}\chi, \quad \forall\sigma \in \Lambda_{\beta_1}\cap\Lambda_{\beta_2}\cap\{\sigma\in\C \,\,|\,\, \Im(\sigma)>0\}.$$
    However, in $\Im(\sigma)<0$ we do not expect $\chi P_{\beta_1}(\sigma)^{-1}\chi$ and $\chi P_{\beta_2}(\sigma)^{-1}\chi$ to match up. This makes it necessary to view the analytic continuation of the cutoff resolvent for the original Kerr spectral family as a multi-valued analytic function of $\sigma$, or rather a function defined on the logarithmic cover of the complex plane. This can also be interpreted in terms of the expected logarithmic singularity of $\chi P_0(\sigma)^{-1}\chi$ at $\sigma=0$.
\end{rmk}

Theorem \ref{thm_analytic_continuation} now follows.
\begin{proof}[Proof of Theorem \ref{thm_analytic_continuation}]
    We perform complex scaling with $R_1$, the start of the complex deformation, chosen so that $\supp(\chi)$ is contained in the ball $B_{R_1}$.
    Then for each $\beta \in (-\pi,\pi)$, we have by Proposition \ref{prop_cutoffs_agree}
    $$\chi P_{\beta}(\sigma)^{-1}\chi = \chi P_0(\sigma)^{-1}\chi$$
    on the non-empty set $\Lambda_\beta\cap\Lambda_0$. Thus, the cutoff resolvents for various $\beta$ provide a meromorphic continuation of $\chi P_0(\sigma)^{-1}\chi: \Bar{H}^{s-1}(X) \to \Bar{H}^{s}(X)$ to the set
    $$\bigcup_{\beta \in (-\pi,\pi)} \bigl(\Lambda_\beta \cap \Omega_s\bigr) = \Bigl\{\sigma\in\Lambda \,\,|\,\, \arg(\sigma) \in (-\pi,2\pi), \, \Im(\sigma) > \tfrac{1}{\alpha}\bigl(\tfrac{1}{2} - s\bigr) \Bigr\}.$$
    The poles of this meromorphic continuation are contained in the set $\mathrm{QNM}_{m,a}$ of Definition \ref{def_qnm}.
    
    We will now show that for $\supp(\chi)$ large enough the poles of the meromorphically continued cutoff resolvent in fact coincide with $\mathrm{QNM}_{m,a}$. Thus, let $\chi = 1$ on a ball $B_R$ with $R > R_0$, where $R_0$ is the radius from Lemma \ref{analytic_coefficients}. The claim will follow from the fact that for any $\beta \in (-\pi,\pi)$, $\sigma \in \C$: if $u \in \Bar{C}^\infty(X_\beta)$ satisfies $P_\beta(\sigma)u = 0$ and $\chi u = 0$ then $u = 0$. Indeed, applying Lemma \ref{deformation_lemma}, as in the proof of Proposition \ref{prop_qnm_scaling_independent}, shows that such a $u$ extends from $X_\beta \setminus B_{R_0}$ to an analytic function on an open neighborhood in $\C^3$. But $B_R$ has non-empty intersection with this open set in $\C^3$ and $u$ satisfies $u = \chi u = 0$ on $B_R$, so we must have $u=0$. Note that we do not require $\sigma \in \Lambda_\beta$.

    Let $\sigma_0$ be a pole of $P_\beta(\sigma)^{-1}$ and take $s > \frac{1}{2} - \alpha\Im(\sigma_0)$. We will show below that
    $$\{\mathrm{res}_{\sigma=\sigma_0}P_\beta(\sigma)^{-1}\chi f(\sigma) \,\,|\,\, f:\C \to \Bar{H}^{s-1}(X_\beta) \text{ polynomial}\} \neq \{0\},$$
    where $\mathrm{res}_{\sigma=\sigma_0}$ denotes the residue at $\sigma_0$. Assume for the moment that this holds.
    Then there is a largest $k \in \N$ such that there exists $f_k \in \Bar{H}^{s-1}(X_\beta)$ with
    $$u_k = \mathrm{res}_{\sigma=\sigma_0}P_\beta(\sigma)^{-1}\chi (\sigma-\sigma_0)^k f_k \neq 0.$$
    Writing $P_\beta(\sigma) = P_\beta(0) - \sigma Q - \sigma^2 R$ for $Q,R \in \mathrm{Diff}^1(X_\beta)$, see \eqref{complex_scaled_operator}, we have on $\Bar{H}^{s-1}(X_\beta)$:
    $$P_\beta(\sigma_0)P_\beta(\sigma)^{-1} = 1 + (\sigma - \sigma_0)(Q + 2\sigma_0 R)P_\beta(\sigma)^{-1} + (\sigma - \sigma_0)^2 R P_\beta(\sigma)^{-1}.$$
    Thus,
    $$P_\beta(\sigma_0)u_k = (Q + 2\sigma_0 R)\mathrm{res}_{\sigma=\sigma_0}P_\beta(\sigma)^{-1}\chi (\sigma-\sigma_0)^{k+1} f_k + R\,\mathrm{res}_{\sigma=\sigma_0}P_\beta(\sigma)^{-1}\chi (\sigma-\sigma_0)^{k+2} f_k = 0$$
    by our assumption on $k$. Now $\chi u_k = 0$ would imply that $\supp(u_k) \subset X_\beta \setminus B_R$ so the ellipticity of $P_\beta(\sigma)$ on $\supp(u_k)$ would imply $u_k \in C^\infty(X_\beta)$. Thus, by the discussion above, we must have $\chi u_k \neq 0$ showing that
    $$\{\mathrm{res}_{\sigma=\sigma_0}\chi P_\beta(\sigma)^{-1}\chi f(\sigma) \,\,|\,\, f:\C \to \Bar{H}^{s-1}(X_\beta) \text{ polynomial}\} \neq \{0\},$$
    i.e. $\sigma_0$ is a pole of the cutoff resolvent.
    
    To prove the claim above, we introduce a pairing
    $$\pair{\cdot\,}{\cdot}: \Bar{H}^s(X_\beta) \times \Dot{H}^{-s}(X_{-\beta}) \to \C
    .$$
    Note that $X_{-\beta}$ is the complex conjugate of $X_\beta$. Thus, for $u\in \Bar{C}^\infty(X_\beta)$ and $v \in \Dot{C}^\infty(X_{-\beta})$ the following contour integral is well-defined
    $$\pair{u}{v} = \int_{X_{\beta}} u(z) \overline{v(\overline{z})} \Bigl(1 + \frac{a^2 z_3^2}{z_1^2 + z_2^2 + z_3^2}\Bigr)dz_1\wedge dz_2 \wedge dz_3.$$
    Extending by density, this pairing defines an isomorphism
    $$\Dot{H}^{-s}(X_{-\beta}) \xrightarrow{\sim} \bigl(\Bar{H}^s(X_\beta)\bigr)^*.$$
    Identifying $C^\infty(X_\beta)$ and $C^\infty(X_{-\beta})$ with $C^\infty(X)$ via $F_\beta^*$ respectively $F_{-\beta}^*$, i.e. working in local coordinates, the contour integral becomes
    $$\pair{u}{v} = \int_{X} u(r,\theta,\varphi)\overline{v(r,\theta,\varphi)}\rkerr_\beta^2(r,\theta)f_\beta'(r)\sin(\theta)\,drd\theta d\varphi.$$
    An explicit calculation using \eqref{complex_scaled_operator} now shows that
    $$\pair{P_\beta(\sigma)u}{v} = \pair{u}{P_{-\beta}(\overline{\sigma})v}, \quad \forall u \in \Bar{H}^s(X_\beta), v\in \Dot{H}^{-s}(X_{-\beta}).$$
    Now $P_\beta(\sigma)^{-1}\chi: \Bar{H}^{s-1}(X_\beta) \to \Bar{H}^s(X_\beta)$ is analytic near $\sigma_0$ if and only if 
    $$\pair{P_\beta(\sigma)^{-1}\chi f}{v} = \pair{f}{\chi P_{-\beta}(\overline{\sigma})^{-1}v}$$
    is analytic near $\sigma_0$ for all $f \in \Bar{H}^{s-1}(X_\beta)$, $v \in \Dot{H}^{-s}(X_{-\beta})$, which is the case if and only $\chi P_{-\beta}(\sigma)^{-1}$ is analytic near $\overline{\sigma_0}$. Here we used that $\chi$ is real-valued and supported away from the complex deformation. A similar argument as above shows that the set
    $$\{\mathrm{res}_{\sigma=\overline{\sigma_0}}P_{-\beta}(\sigma)^{-1} v(\sigma) \,\,|\,\, v:\C \to \Dot{H}^{-s}(X_\beta) \text{ polynomial}\}$$
    contains a non-trivial element $g$ satisfying $P_{-\beta}(\overline{\sigma_0})g = 0$. Again arguing as above, we must have $\chi g \neq 0$, so $\chi P_{-\beta}(\sigma)^{-1}$ is not analytic at $\overline{\sigma_0}$.
\end{proof}

\section{High energy estimates}
\label{section_high_energy}

In this section, we study the behavior of the resolvent $P_\beta(\sigma)^{-1}$ in the high energy limit $|\Re(\sigma)| \to \infty$ with $\sigma$ confined to a strip of the form $|\Im(\sigma)| < \gamma$ for some fixed $\gamma>0$. We will show that there is such a strip containing no quasinormal modes for $|\Re(\sigma)|$ large enough, as in Theorem \ref{thm_high_energy}. For $\Re(\sigma)>0$, we choose some small positive $\beta$ so that, by Proposition \ref{prop_qnm}, $P_\beta(\sigma)^{-1}$ is indeed well-defined as a meromorphic family in $\{\sigma\in\C \,\,|\,\, |\Im(\sigma)| < \gamma,\, \Re(\sigma) > C\}$ for some $C>0$. Similarly, for $\Re(\sigma)<0$, we take $\beta$ small and negative.

We will transform this high energy regime into a semiclassical problem with semiclassical parameter $h=|\sigma|^{-1}$. Note that the differential operator $P_\beta(\sigma)$ has the form
$$P_\beta(\sigma) = \sum_{|\alpha|\leq 2} \sigma^{2-|\alpha|}a_\alpha(x)D_x^\alpha.$$
Thus, writing $\sigma = h^{-1}z$ with $z\in\C$ satisfying $|z| = 1$, we have
$$P_\beta(\sigma) = h^{-2}\sum_{|\alpha|\leq 2} z^{2-|\alpha|}a_\alpha(x)(hD_x)^\alpha.$$
We take as our semiclassical operator
\begin{equation}
\label{semiclassical_operator_kerr}
    P_\hbar = h^2P_\beta(h^{-1}z).
\end{equation}
Note that this belongs to the semiclassical operator algebra introduced in Section \ref{section_semiclassical}. We suppress the dependence on $\beta$ and $z$ from our notation. The condition $|\Im(\sigma)|\leq\gamma$ implies $|\Im(z)| \leq \gamma h$ and $|\Re(z)| = 1 +\mathcal{O}(h)$, so for the purpose of the semiclassical principal symbol, we have $z = \pm 1$. We will study the action of $P_\hbar$ on the semiclasical Sobolev spaces $\Bar{H}^s_h(X_\beta)$, which are akin to the spaces in Section \ref{section_fredholm} but with derivatives weighted by a factor of $h$, see Section \ref{section_semiclassical}. The main result of this section is the following semiclassical resolvent estimate.

\begin{prop}
\label{prop_semiclassical_estimate}
    For $\Re(z)>0$, choose $\beta>0$ small, and for $\Re(z)<0$, choose $\beta<0$ small. Then there exists $\gamma > 0$ such that for $|\Im(z)| \leq \gamma h$ and $h$ small enough the following estimate holds for every $s>\frac{1}{2}+\alpha\gamma$, $N>0$ and some $C = C_{s,N}$:
    \begin{equation*}
        \|u\|_{\Bar{H}_h^s(X_\beta)} \leq C\bigl(h^{-2}\|P_\hbar u\|_{\Bar{H}_h^{s-1}(X_\beta)} + h^N\|u\|_{\Bar{H}_h^{-N}(X_\beta)}\bigr), \quad \forall u\in\mathcal{X}_\beta^s,
    \end{equation*}
    where $P_\hbar$ is the operator in \eqref{semiclassical_operator_kerr}.
\end{prop}

The proof proceeds via semiclassical symbolic estimates, which we briefly review in the next subsection, see also \cite[Appendix E]{dyatlov_zworski_scattering_book}. Theorem \ref{thm_high_energy} now follows easily from Proposition \ref{prop_semiclassical_estimate}.

\begin{proof}[Proof of Theorem \ref{thm_high_energy}]
    In terms of $\sigma = h^{-1}z$ and $P_\beta(\sigma) = h^2P_\hbar$, Proposition \ref{prop_semiclassical_estimate} sates that for $|\Im(\sigma)| < \gamma$, we have the estimate
    \begin{equation*}
        \|u\|_{\Bar{H}_{|\sigma|^{-1}}^s(X_\beta)} \leq C\bigl(\|P_\beta(\sigma) u\|_{\Bar{H}_{|\sigma|^{-1}}^{s-1}(X_\beta)} + |\sigma|^{-N}\|u\|_{\Bar{H}_{|\sigma|^{-1}}^{-N}(X_\beta)}\bigr).
    \end{equation*}
    Choosing $|\sigma|$ large enough, we can absorb the error term on the right into $\|u\|_{\Bar{H}_{|\sigma|^{-1}}^s(X_\beta)}$ and obtain
    $$\|u\|_{\Bar{H}_{|\sigma|^{-1}}^s(X_\beta)} \leq C\|P_\beta(\sigma) u\|_{\Bar{H}_{|\sigma|^{-1}}^{s-1}(X_\beta)}$$
    when $|\Re(\sigma)| > c$ for some constant $c$. Thus, the kernel of $P_\beta(\sigma)$ is trivial in the strip $\{\sigma\in\C \,\,|\,\, |\Im(\sigma)| < \gamma,\,|\Re(\sigma)|>c\}$. Together with the absence of quasinormal modes in the upper half-plane, see Proposition \ref{mode_stability}, this proves the result.
\end{proof}

\begin{rmk}
    Note that from the proof of Theorem \ref{thm_high_energy} we also obtain an estimate on the resolvent in the strip $\{|\Im(\sigma)| < \gamma,\,|\Re(\sigma)|>c\}$:
    $$\|P_\beta(\sigma)^{-1}f\|_{\Bar{H}_{|\sigma|^{-1}}^s(X_\beta)} \leq C\|f\|_{\Bar{H}_{|\sigma|^{-1}}^{s-1}(X_\beta)}, \quad \forall f\in \Bar{H}^{s-1}(X_\beta).$$
    This estimate holds uniformly as $|\Re(\sigma)| \to \infty$, but in terms of a $|\sigma|$-dependent Sobolev norm, where derivatives are weighted with factors of $|\sigma|^{-1}$.
\end{rmk}

\subsection{Semiclassical analysis}
\label{section_semiclassical}

Here, we review some elements of semiclassical analysis and provide references to the semiclassical symbolic estimates that will be used in the proof of Proposition \ref{prop_semiclassical_estimate}. This exposition roughly follows \cite[Appendix E]{dyatlov_zworski_scattering_book}, see also the monograph \cite{zworski_semiclassical_book}.

The semiclassical operator algebra consists of pseudodifferential operators depending on a small parameter. We will use the subscript $\hbar$ to denote semiclassical objects, while the semiclassical parameter will be denoted by $h$. This operator algebra is a microlocalization of the space of semiclassical differential operators $\mathrm{Diff}_\hbar^k(M)$, which take the form
$$\sum_{|\alpha|\leq k} a_\alpha(x)(hD_x)^\alpha$$
with derivatives weighted by factors of $h$.

Semiclassical pseudodifferential operators of order $m$, denoted $\Psi_\hbar^m(M)$, are obtained by quantizing symbols 
$$a \in C^\infty\bigl([0,1)_h,S^m(T^*M)\bigr),$$
i.e. $h$-dependent families of symbols lying in the symbol spaces of Section \ref{section_microlocal_estimates} (where the smoothness is with respect to the Fréchet topology on $S^m(T^*M)$). The quantization procedure now takes the following form on $\R^n$:
$$\mathrm{Op}_\hbar(a)u(x) = \frac{1}{(2\pi h)^n}\int_{\R^n}\int_{\R^n} e^{\frac{i}{h}(x-y)\cdot\xi}a(h,x,\xi)u(y)\,dyd\xi, \quad \forall u \in C^\infty_c(\R^n).$$
Note the factor of $h^{-1}$ in the exponential. With this definition, we have $\mathrm{Op}_\hbar(\xi^j) = hD_x^j$ on $\R^n$. As in Section \ref{section_microlocal_estimates}, the quantization procedure can be patched together from local coordinates to form the space $\Psi_\hbar^m(M)$. The Schwartz kernel of elements in $\Psi_\hbar^m(M)$ is smooth off the diagonal and decays superpolynomially in $h$ as $h\to 0$.

Notice that to $a \in C^\infty\bigl([0,1)_h,S^m(T^*M)\bigr)$, we can associate its $h$-dependent principal symbol, in the sense of Section \ref{section_microlocal_estimates}, 
\begin{equation}
\label{h_dependent_symbol}
    [a] \in C^\infty\bigl([0,1)_h,S^m(T^*M) / S^{m-1}(T^*M)\bigr),
\end{equation}
which characterizes the $|\xi| \to \infty$ asymptotics. In addition, $a$ carries a semiclassical principal symbol $a|_{h=0} \in S^m(T^*M)$, characterizing the $h\to 0$ asymptotic behavior. These notions are compatible, in the sense that $[a]|_{h=0}$ is just the equivalence class of $a|_{h=0}$ in $S^m(T^*M)) / S^{m-1}(T^*M)$. As in \cite{dyatlov_zworski_scattering_book}, we restrict to a subspace of symbols, where $[a]$ is in fact independent of $h$. In this case, the information from both principal symbols is carried by $a|_{h=0}$. More precisely, we work with the class of symbols, denoted simply as $S_\hbar^m(T^*M)$, which have an asymptotic expansion of the form
\begin{equation*}
    a(h,x,\xi) \sim \sum_{j=0}^\infty h^j a_j(x,\xi), \quad\text{with}\quad a_j(x,\xi) \in S_\mathrm{cl}^{m-j}(T^*M),
\end{equation*}
where $S_\mathrm{cl}^{m-j}(T^*M)$ is the space of classical (or polyhomogeneous) symbols of order $m-j$, see \cite[Section E.1.2]{dyatlov_zworski_scattering_book}. The resulting space of semiclassical pseudodifferential operators is denoted $\Psi_\hbar^m(M)$. As our space of residual operators we now take
$$h^\infty \Psi_\hbar^{-\infty}(M) = \bigcap_{N\in\R} h^N\Psi_\hbar^{-N}(M).$$
As before, properly supported semiclassical pseudodifferential operators form a graded algebra under composition.

The semiclassical principal symbol map, induced from the restriction to $h=0$, is multiplicative and fits into the short exact sequence
$$0 \rightarrow \Psi_\hbar^{m-1}(M) \rightarrow \Psi_\hbar^m(M) \xrightarrow[]{\sigma_\hbar} S^m(T^*M) / hS^{m-1}(T^*M) \rightarrow 0.$$
For our choice of symbol class, the principal symbol can also be viewed as a smooth function on the fiber-radially compactified cotangent bundle $\overline{T}^*M$, see \cite[Section E.1.3]{zworski_semiclassical_book}. This is a manifold with interior $T^*M$ and boundary $\partial \overline{T}^*M \cong S^*M$ the sphere bundle of Remark \ref{sphere_bundle}. Note that $\bdf(x,\xi) = \langle\xi\rangle^{-1}$ is a smooth defining function for fiber infinity, i.e. $\partial \overline{T}^*M$. For $A \in \Psi_\hbar^m(M)$, the rescaled principal symbol $\langle\xi\rangle^{-m}\sigma_\hbar(A)$ extends to an element of $C^\infty(\overline{T}^*M)$. The restriction of this function to fiber infinity is just the classical principal symbol, viewed as a function on $S^*M$ as in Remark \ref{sphere_bundle}, and carries all the information of \eqref{h_dependent_symbol}.

The semiclassical wavefront set and elliptic set should be viewed as subsets of $\overline{T}^*M$. The wavefront set $\WF(A)$ of an operator $A \in \Psi_\hbar^m(M)$ is the set of points in $\overline{T}^*M$ near which the full symbol is not of order $h^\infty \langle\xi\rangle^\infty$. That is, $(x_0,\xi_0) \in \overline{T}^*M$ is not contained in $\WF(A)$, if and only if there is a neighborhood $U \subset \overline{T}^*M$ of $(x_0,\xi_0)$ such that for all $\alpha,\beta \in \N_0^n$ and all $N\in\N$ we have
$$|\partial_x^\alpha \partial_\xi^\beta a(h,x,\xi)| \leq C_{\alpha,\beta,N} h^N\brac{\xi}^{-N}, \quad \forall (x,\xi)\in U,$$
where $A$ is locally given as the quantization of the symbol $a$. Note that, as in the classical case (Proposition \ref{microlocal_partition_of_unity}), we have the existence of semiclassical partitions of unity, see \cite[Proposition E.30]{dyatlov_zworski_scattering_book}.

The semiclassical elliptic set $\mathrm{Ell}_\hbar(A)$ consists of all $(x_0,\xi_0) \in \overline{T}^*M$ that have a neighborhood $U$ where the semiclassical principal symbol of $A$ satisfies
$$|\sigma_\hbar(A)(x,\xi)| \geq C\brac{\xi}^m, \quad \forall (x,\xi)\in U.$$
The complement $\mathrm{Char}_\hbar(A) = \overline{T}^*M\setminus\mathrm{Ell}_\hbar(A)$ is the semiclassical characteristic set. Note that this is just the zero set of $\brac{\xi}^{-m}\sigma_\hbar(A)(x,\xi)$ in $\overline{T}^*M$.

Operators in $\Psi_\hbar^m(M)$ act on semiclassical Sobolev spaces. On $\R^n$, we define the semiclassical Sobolev space $H^s_h(\R^n)$ for $s\in\R$ to agree with $H^s(\R^n)$ as a space, but equipped with the $h$-dependent norm
$$\|u\|_{H^s_h(\R^n)} = \|\mathrm{Op}_\hbar(\brac{\xi}^s)u\|_{L^2(\R^n)} = \|\brac{h\xi}^s\hat{u}(\xi)\|_{L^2(\R^n)}.$$
Note that for $k\in\N$, this norm is equivalent to
$$\|u\|_{H^k_h(\R^n)}^2 = \sum_{|\alpha|\leq k} \|(hD_x)^\alpha u\|_{L^2(\R^n)}^2,$$
the usual Sobolev norm with derivatives weighted by a factor of $h$. On a smooth manifold $M$, we define $H^s_{h,c}(M)$ and $H^s_{h,\mathrm{loc}}(M)$ by patching together the semiclassical Sobolev spaces from charts, as in Section \ref{section_microlocal_estimates}. We once again use the notation $\|u\|_{H_h^s}$ to denote a choice of norm on elements of $H^s_{h,c}(M)$ supported in a fixed compact set. We then have the following mapping property for $A \in \Psi_\hbar^m(M)$:
$$A: H^s_{h,c}(M) \to H^{s-m}_{h,\mathrm{loc}}(M)$$
and if $A$ is properly supported then the estimate
$$\|Au\|_{H_h^s} \leq C\|u\|_{H_h^s}$$
holds for all $u\in H_{h,c}^s(M)$ supported in a fixed compact set.

We now turn to the symbolic estimates needed for the proof of Proposition \ref{prop_semiclassical_estimate}. These are essentially the extensions of the estimates of Section \ref{section_microlocal_estimates} to the semiclassical setting, where principal symbols, as well as wavefront sets and elliptic sets, now live on $\overline{T}^*M$ as above. The precise statements and proofs can be found in \cite[Appendix E]{dyatlov_zworski_scattering_book} and we direct the reader there. Note that we take all operators except $P$ in these estimates to be compactly supported elements of $\Psi_\hbar^0(M)$.

Microlocally on the elliptic set of $P\in\Psi^m_\hbar(M)$, we have semiclassical elliptic estimates, see \cite[Theorem E.33]{dyatlov_zworski_scattering_book}. These take the form
$$\|Bu\|_{H^s_h} \leq C\bigl(\|GPu\|_{H^{s-m}_h} + h^N\|\chi u\|_{H^{-N}_h}\bigr)$$
for compactly supported $B,G \in \Psi^0_\hbar(M)$ with $\WF_\hbar(B) \subset \Ell_\hbar(P)\cap\Ell_\hbar(G)$. As in Section \ref{section_microlocal_estimates}, we can define the notion of uniform semiclassical pseudodifferential operators on $\R^n$ and of semiclassical ellipticity on an open subset uniformly as $|x| \to \infty$. We then have a version of the uniform elliptic estimate, Proposition \ref{uniform_elliptic_estimate}, in the semiclassical setting, where the error term now comes with a factor of $h^N$.

On the characteristic set of an operator $P\in\Psi^m_\hbar(M)$ with real-valued principal symbol we can use semiclassical propagation and radial estimates. The Hamiltonian flow of the semiclassical principal symbol is now on $\overline{T}^*M$ with the flow at fiber infinity corresponding to the Hamiltonian flow of the classical symbol, see Remark \ref{homogeneous_hamiltonian}. Semiclassical propagation estimates take the form
$$\|Bu\|_{H^s_h} \leq C\bigl(h^{-1}\|GPu\|_{H^{s-m+1}_h} + \|Eu\|_{H^s_h} + h^N\|\chi u\|_{H^{-N}_h}\bigr),$$
where estimates are propogated (forwards or backwards) along the Hamiltonian flow from $\Ell_\hbar(E)$ to $\WF_\hbar(B)$ while remaining in $\Ell_\hbar(G)$.

In our proof of semiclassical resolvent estimates for the complex scaled Kerr spectral family, we will need to propagate estimates into and out of the complex scaling region. This poses a problem, since the semiclassical principal symbol becomes complex-valued there. However, propagtion estimates continue to hold, for the Hamiltonian flow of the real part of the principal symbol, as long as the imaginary part of the principal symbol has a definite sign. More precisely, if $\Im(\sigma_\hbar(P))\leq 0$ then estimates can be propagated in the forward direction along the Hamiltonian flow of $\Re(\sigma_\hbar(P))$, and if $\Im(\sigma_\hbar(P))\geq 0$ then estimates can be propagated in the backward direction. See \cite[Theorem E.47]{dyatlov_zworski_scattering_book} for the precise statement.

In the semiclassical setting, radial sets should be viewed at fiber infinity, i.e. as subsets $L \subset \partial\overline{T}^*M$. Radial sources and sinks are defined by requiring the conditions of Definition \ref{def_source_sink} to hold in a neighborhood of $L$. Radial estimates then extend to the semiclassical setting. We will only need the high regularity semiclassical radial estimate, \cite[Theorem E.52]{dyatlov_zworski_scattering_book}, which takes the form
$$\|Bu\|_{H^s_h} \leq C\bigl(h^{-1}\|GPu\|_{H^{s-m+1}_h} + h^N\|\chi u\|_{H^{-N}_h}\bigr).$$
Here, the radial source or sink is contained in $\Ell_\hbar(B)$ and $\WF_\hbar(B) \subset \Ell_\hbar(G)$. Moreover, as in Proposition \ref{high_reg_radial_estimate} this can only be applied to $u$ that are in $H^{s',\mathrm{loc}}_h(M)$ microlocally at the radial set, where $s'$ is larger than the threshhold regularity, see \ref{threshold_regularity}.

As in Section \ref{section_fredholm_estimates}, we will close our estimates beyond the horizon using a semiclassical version of hyperbolic estimates. These will be applied to Sobolev spaces of extendable distributions, which are defined exactly as in Section \ref{section_hyperbolic_estimates}, but with respect to the semiclassical Sobolev norms on $\R^n$. Inside the horizon, i.e. in $\{r\in(r_0,r_+)\}$, our operator $P_\hbar$ will be semiclassically strictly hyperbolic with respect to $r$, see \cite[Definition E.55]{dyatlov_zworski_scattering_book}. For any $r_1 \in (r_0,r_+)$, we then have semiclassical hyperbolic estimates of the form
$$\|\chi_1 u\|_{\Bar{H}^s_h(X)} \leq C\bigl(h^{-1}\|\chi_2 P_\hbar u\|_{\Bar{H}^{s-1}_h(X)} + \|\chi_3 u\|_{\Bar{H}^s_h(X)}\bigr),$$
where $\supp{\chi_1} \subset \{r\in (r_0,r_1)\}$, $\chi_2=1$ on $\{r\in (r_0,r_1)\}$ and $\chi_3 = 1$ near $\{r=r_1\}$. See \cite[Theorem E.57]{dyatlov_zworski_scattering_book} for a proof.

The new feature in the proof of \ref{prop_semiclassical_estimate}, compared to Section \ref{section_fredholm_estimates}, is the presence of a trapped set for the semiclassical Hamiltonian flow, i.e. integral curves in the characteristic set that neither escape to infinity nor into the black hole. Note that this is related to the fact that the Kerr metric exhibits trapped null-geodesics, see Remark \ref{rmk_geodesics}. In general, we define the trapped set of an operator $P \in \Psi_\hbar^m(M)$ as the set of points in the characteristic set that remain in a compact subset of $T^*M$ under the Hamiltonian flow of $p=\sigma_\hbar(P)$. More precisely, let $(x,\xi) \in \Char_\hbar(P) \cap T^*M$ and denote by
$$\gamma: (T_{\mathrm{min}},T_{\mathrm{max}}) \to T^*M, \quad \gamma(t) = e^{tH_p}(x,\xi)$$
the maximally extended integral curve of the Hamiltonian vector field $H_p$, where of course $T_{\mathrm{min}}=-\infty$, $T_{\mathrm{max}}=\infty$ is possible. Then $(x,\xi)$ is in the trapped set $\Gamma \subset T^*M$, if and only if $\gamma((T_{\mathrm{min}},T_{\mathrm{max}}))$ is contained in a compact subset of $T^*M$. Similarly, we define the forward trapped set $\Gamma_s$ and the backward trapped set $\Gamma_u$ to consist of all $(x,\xi)\in T^*M$, such that $\gamma([0,T_{\mathrm{max}}))$, respectively $\gamma((T_{\mathrm{min}},0])$, is contained in a compact subset of $T^*M$. Note that $\Gamma_s$ and $\Gamma_u$ are invariant under the flow and $\Gamma = \Gamma_s\cap\Gamma_u$.

Since $\Gamma$ is invariant under the Hamiltonian flow, we cannot use the above propagation result to propagate estimates into $\Gamma$. However, if the trapped set is normally hyperbolic, and hence in some sense unstable, we can nonetheless obtain semiclassical estimates at trapping with a loss of a factor of $h^{-1}$ compared to standard propagation estimates. Such normally hyperbolic trapping estimates were pioneered in \cite{wunsch_zworski_trapping}, see also \cite{dyatlov_trapping}. We will call the trapped set normally hyperbolic if there exists a bounded neighborhood $U \subset T^*M$ of $\Gamma$ and smooth functions $\varphi_u,\varphi_s \in C_c^\infty(T^*M)$ which locally in $U$ are defining functions for $\Gamma_u$ respectively $\Gamma_s$ within the characteristic set, i.e. $\Gamma_{u/s}\cap U = \varphi_{u/s}^{-1}(0) \cap \mathrm{Char}_\hbar(P)\cap U$. Furthermore, $\phi_u,\phi_s$ are required to satisfy $\{\phi_u,\phi_s\} > 0$ on $U$ and
\begin{equation}
\label{trapping_condition}
    H_p\phi_u = w_u\phi_u, \quad H_p\phi_s = -w_s\phi_s \quad\text{on}\,\, U \quad\text{with}\quad w_u,w_s \in C^\infty(U),\quad w_u,w_s \geq \gamma > 0
\end{equation}
for some positive constant $\gamma$. Note that in a neighborhood of the trapped set this implies that $\Gamma_u,\Gamma_s$ are smooth manifolds and have codimension one within the characteristic set. Furthermore, \eqref{trapping_condition} shows that near trapping the Hamiltonian flow on $\Gamma_s$ exponentially approaches $\Gamma$ in the forward direction, whereas on $\Gamma_u$ the flow exponentially approaches $\Gamma$ in the backward direction.

When the trapping is normally hyperbolic, one can obtain a type of propagation estimate at the trapped set, see \cite[Theorem 4.7]{hintz_vasy_quasilinear}. This requires an additional condition on the imaginary part of the subleading symbol of $P$, namely
\begin{equation}
\label{expansion_rate}
    \sigma_\hbar\Bigl(\frac{1}{2ih}(P-P^*)\Bigr) < \frac{1}{2}\min\bigl(\inf_U(w_u),\inf_U(w_s)\bigr), \quad\text{on}\,\, \Gamma.
\end{equation}
We then have an estimate of the form
$$\|Bu\|_{H^s_h} \leq C\bigl(h^{-2}\|GPu\|_{H^{s-m+1}_h} + h^{-1}\|Eu\|_{H^s_h} + h^N\|\chi u\|_{H^{-N}_h}\bigr)$$
with $\Gamma \subset \Ell_\hbar(B)$, $\WF(B)\subset \Ell_\hbar(G)$ and $\WF_\hbar(E)\cap\Gamma_u = \emptyset$. Thus, we can control $u$ microlocally at the trapped set by the size of $u$ away from the backward trapped set.

We note that normally hyperbolic trapping can also be defined in terms of the more geometric and dynamical conditions of \cite{dyatlov_trapping}. In particular, the flow is required to be hyperbolic in the normal directions to $\Gamma$ within $\Char_\hbar(P)$. 
Note that the quantity $\min\bigl(\inf_U(w_u),\inf_U(w_s)\bigr)$ in condition \eqref{expansion_rate} is related to the minimal expansion rate of the linearized flow in the normal directions at trapping.

\subsection{Semiclassical principal symbol in the complex scaling region}
\label{section_definite_sign}
The proof of Proposition \ref{prop_semiclassical_estimate} requires a careful study of the semiclassical principal symbol $\sigma_\hbar(P_\hbar)$ and its Hamiltonian flow. Note that at fiber infinity, i.e. $\partial\Bar{T}^*X_\beta$, the semiclassical principal symbol coincides with the usual principal symbol of $P_\beta(\sigma)$, which was studied in Section \ref{section_hamiltonian_flow}. Thus, it remains to consider $\sigma_\hbar(P_\hbar)$ on the interior of $\Bar{T}^*X_\beta$, i.e. $T^*X_\beta$.

We begin in the complex scaling region $\{r > R_0\}$, where $\sigma_\hbar(P_\hbar)$ is complex-valued. We will show that $P_\hbar$ is elliptic as soon as complex scaling kicks in, i.e. when the phase function $\phi_\beta(r)>0$, see Definition \ref{def_phase_function}. In order to propagate estimates into or out of this region, according to \cite[Theorem E.47]{dyatlov_zworski_scattering_book}, we must ensure that the imaginary part of $\sigma_\hbar(P_\hbar)$ has a definite sign on the zero set of $\Re(\sigma_\hbar(P_\hbar))$.

\begin{lemma}
\label{semiclassical_complex_scaling}
    Let $\Sigma_\hbar = \{\Re(\sigma_\hbar(P_\hbar)) = 0\} \subset T^*X_\beta$. Then for $\beta>0$ ($z=1$), we have $\Im(\sigma_\hbar(P_\hbar)) \leq 0$ on $\Sigma_\hbar$, while for $\beta<0$ ($z=-1$), we have $\Im(\sigma_\hbar(P_\hbar)) \geq 0$ on $\Sigma_\hbar$. Moreover, $P_\hbar$ is semiclassically elliptic on $\{\phi_\beta(r)>0\}$ and for some $R$ large enough $P_\hbar$ is uniformly semiclassically elliptic in $\{r>R\}$.
\end{lemma}
\begin{proof}
    Note that the statement regarding the sign of $\Im(\sigma_\hbar(P_\hbar))$ is automatically true in $\{r\leq R_0\}$, where $\Im(\sigma_\hbar(P_\hbar)) = 0$. Working in local coordinates, as in Lemma \ref{lemma_scattering_elliptic}, the semiclassical principal symbol in $\{r>R_0\}$ is given by
    \begin{equation*}
        \sigma_\hbar(P_\hbar) = \frac{r^2}{\rkerr_\beta^2}\Bigl(\frac{\mu_\beta}{(f'_\beta r)^2}\xi^2 + \frac{\eta^2}{r^2} + \frac{\nu^2}{r^2\sin^2(\theta)} + \frac{2a}{f_\beta'r^2}\xi\nu\Bigr) + z\frac{4maf_\beta}{\rkerr_\beta^2\mu_\beta}\nu - \frac{2mf_\beta(f_\beta^2+a^2)}{\rkerr_\beta^2\mu_\beta} - 1,
    \end{equation*}
    where $f_\beta(r) = e^{i\phi_\beta(r)}r$, $\rkerr_\beta^2 = f_\beta(r)^2 + a^2\cos^2(\theta)$ and $\mu_\beta = f_\beta(r)^2 - 2mf_\beta(r) + a^2$.
    For all $r > R_0$, we have
    \begin{equation*}
    \begin{split}
        \frac{r^2}{\rkerr_\beta^2} &= e^{-2i\phi_\beta}\bigl(1 - \frac{a^2\cos^2(\theta)}{e^{2i\phi_\beta}r^2+a^2\cos^2(\theta)}\bigr) = e^{-2i\phi_\beta} + \mathcal{O}(R_0^{-2}), \\
        \frac{\mu_\beta}{(f_\beta'r)^2} &= 1 + \frac{f_\beta^2-(f_\beta'r)^2}{(f_\beta'r)^2} - \frac{2mf_\beta}{(f_\beta'r)^2} + \frac{a^2}{(f_\beta'r)^2} = 1 + \mathcal{O}(R_0^{-1}) + \mathcal{O}(\varepsilon),
    \end{split}
    \end{equation*}
    where we estimated $f_\beta^2-(f_\beta'r)^2$ in terms of $|r\phi_\beta'(r)| < \varepsilon$ as in Lemma \ref{lemma_scattering_elliptic}, see also Definition \ref{def_phase_function}. Similarly, using Young's inequality and writing $|\Vec{\xi}|^2 = \xi^2 + \frac{\eta^2}{r^2} + \frac{\nu^2}{r^2\sin^{2}(\theta)}$, we find
    \begin{equation*}
        \frac{2a}{f_\beta'r^2}\xi\nu = \mathcal{O}(R_0^{-1})|\Vec{\xi}|^2, \quad z\frac{4maf_\beta}{\rkerr_\beta^2\mu_\beta}\nu = \mathcal{O}(R_0^{-1}) + \mathcal{O}(R_0^{-1})|\Vec{\xi}|^2, \quad \frac{2mf_\beta(f_\beta^2+a^2)}{\rkerr_\beta^2\mu_\beta} = \mathcal{O}(R_0^{-1}).
    \end{equation*}
    Thus, the real part of the principal symbol satisfies
    \begin{equation*}
        \Re(\sigma_\hbar(P_\hbar)) = \bigl(\cos(2\phi_\beta) + \mathcal{O}(R_0^{-1}) + \mathcal{O}(\varepsilon)\bigr)|\Vec{\xi}|^2 + \mathcal{O}(R_0^{-1}) - 1.
    \end{equation*}
    Choosing $\beta$ small, $\cos(2\phi_\beta)$ is close to $1$ and, choosing $R_0^{-1}$ and $\varepsilon$ small, we find that
    \begin{equation}
    \label{real_part_char_bound}
        \frac{1}{2} \leq |\Vec{\xi}|^2 \leq 2, \quad\text{on}\,\, \Sigma_\hbar.
    \end{equation}

    Considering the imaginary part of $\sigma_\hbar(P_\hbar)$, more care is needed in estimating the above terms. Close to $\phi_\beta(r) = 0$, i.e. as we begin the complex deformation, ${\Im(e^{-2i\phi_\beta}) = -\sin(2\phi_\beta)}$ itself is small, and we cannot just throw away terms of order $R_0^{-1}$. Note that this is exactly the location, where the sign condition is relevant. In particular, we need to be careful with the terms involving $f_\beta'(r)$. Although the derivative of the phase function ${r\phi_\beta'(r) = \beta\psi'(\log(r))}$, see Definition \ref{def_phase_function}, is small, it cannot be bounded by $\phi_\beta(r)$. 

    Thus, consider the leading order term involving the derivative of the phase function. We find
    \begin{equation*}
        \Im\Bigl(\frac{f_\beta^2-(f_\beta'r)^2}{(f_\beta'r)^2}\Bigr) = \Im\Bigl(\frac{1}{(1+ir\phi_\beta')^2}\Bigr) = -\frac{2r\phi_\beta'}{(1+(r\phi_\beta')^2)^2},
    \end{equation*}
    which has the correct sign. The other terms of order $\psi'$ are subleading with respect to $R_0^{-1}$, so choosing $R_0$ large enough they do not affect the sign of $\Im(\sigma_\hbar(P_\hbar))$.
    Writing $\mathcal{O}(\phi_\beta)$ and $\mathcal{O}(\psi')$ for terms that vanish correspondingly as $\phi_\beta, \psi' \to 0$, we can estimate
    \begin{equation*}
    \begin{split}
        \Im\Bigl(\frac{r^2}{\rkerr_\beta^2}\Bigr) = -\sin(2\phi_\beta) + \mathcal{O}(R_0^{-2}\phi_\beta), \quad \Im\Bigl(-\frac{2mf_\beta}{(f_\beta'r)^2} + \frac{a^2}{(f_\beta'r)^2}\Bigl) = \mathcal{O}(R_0^{-1}\phi_\beta) + \mathcal{O}(R_0^{-1}\psi'), \\
        \Im\Bigl(z\frac{4maf_\beta}{\rkerr_\beta^2\mu_\beta}\nu\Bigr) = \mathcal{O}(R_0^{-1}\phi_\beta)|\Vec{\xi}|^2 + \mathcal{O}(R_0^{-1}\phi_\beta), \quad
        \Im\Bigl(\frac{2mf_\beta(f_\beta^2+a^2)}{\rkerr_\beta^2\mu_\beta}\Bigr) = \mathcal{O}(R_0^{-1}\phi_\beta).
    \end{split}
    \end{equation*}
    Finally, we bound
    \begin{equation*}
        \Im\Bigl(\frac{2a}{f_\beta'r^2}\xi\nu\Bigr) = \bigl(\mathcal{O}(R_0^{-1}\phi_\beta) + \mathcal{O}(R_0^{-1}\psi')\bigr)|\Vec{\xi}|^2.
    \end{equation*}
    Altogether, we find for the imaginary part of the principal symbol
    \begin{equation}
    \label{imaginary_part_semiclassical}
    \begin{split}
        \Im(\sigma_\hbar(P_\hbar)) = \Bigl(&-\sin(2\phi_\beta) - \frac{2\beta\psi'}{(1+(\beta\psi')^2)^2}\Bigr)|\Vec{\xi}|^2 \\
        &+ \bigl(\mathcal{O}(R_0^{-1}\phi_\beta) + \mathcal{O}(\varepsilon\phi_\beta) + \mathcal{O}(R_0^{-1}\psi')\bigr)|\Vec{\xi}|^2
        + \mathcal{O}(R_0^{-1}\phi_\beta).
    \end{split}
    \end{equation}
    Recalling that $|\Vec{\xi}|^2$ is bounded as in \eqref{real_part_char_bound} on $\Sigma_\hbar$ and choosing $R_0^{-1}$ and $\varepsilon$ small enough, we see that indeed $\Im(\sigma_\hbar(P_\hbar)) \leq 0$ for $\beta > 0$ and $\Im(\sigma_\hbar(P_\hbar)) \geq 0$ for $\beta < 0$. Equation \eqref{imaginary_part_semiclassical} also shows that, as soon as $\phi_\beta(r)\neq 0$, we have $\Im(\sigma_\hbar(P_\hbar)) \neq 0$ on $\{\Re(\sigma_\hbar(P_\hbar)) = 0\}$, and thus $P_\hbar$ is semiclassically elliptic for $\phi_\beta(r)\neq 0$. Note that the ellipticity at fiber infinity, follows from Lemma \ref{lemma_scattering_elliptic}.

    Further into the region of complex scaling, when $|\phi_\beta| > c|\beta|$ for some $c>0$, the simpler estimate
    \begin{equation*}
        \sigma_\hbar(P_\hbar) = \bigl(e^{-2i\phi_\beta} + \mathcal{O}(R_0^{-1}) + \mathcal{O}(\varepsilon)\bigr)|\Vec{\xi}|^2 - 1 + \mathcal{O}(R_0^{-1}),
    \end{equation*}
    shows that $P_\hbar$ is uniformly semiclassically elliptic there.
\end{proof}

By Lemma \ref{semiclassical_complex_scaling}, if an integral curve of the Hamlitonian vector field enters the complex scaling region, then we can propagate estimates forward along the Hamiltonian flow for $z=1$ and backward along the Hamiltonian flow for $z=-1$, but not the other way around.

\subsection{Hamiltonian flow of the semiclassical principal symbol}

We now turn to the region away from complex scaling. Here, the semiclassical principal symbol is real-valued and we will study its Hamiltonian flow on the semiclassical characteristic set. The essential feature of the Hamiltonian flow is the presence of a trapped set, see the end of Section \ref{section_semiclassical} for this notion. Trapping for the Kerr black hole has been studied in numerous previous works, see for instance \cite{wunsch_zworski_trapping, dyatlov_zworski_KdS_physics, dyatlov_ringdown}, where it is shown that the Kerr trapped set is normally hyperbolic, see also \cite{vasy_KdS, petersen_vasy_full_subextremal} for the related case of Kerr-de Sitter black holes. This inherent instability of trapping will allow us to obtain semiclassical symbolic estimates microlocally at the trapped set. Thus, Proposition \ref{prop_semiclassical_estimate} essentially follows from the analysis of the Hamiltonian flow performed in the above references. However, as noted in the previous subsection, our use of complex scaling imposes a constraint on the propagation direction. That is, when the Hamiltonian flow crosses into the complex scaling region, estimates can only be propagated in a definite direction. We thus require a slightly more detailed analysis of the Hamiltonian flow, in order to ensure that propagation of regularity can be applied in a manner compatible with these constraints.

As in Section \ref{section_hamiltonian_flow}, we work with the principal symbol of $\rkerr^2P_\hbar$ for convenience. Note that $\rkerr^2P_\hbar$ and $P_\hbar$ have the same characteristic set, denoted $\Sigma_\hbar$, and the Hamiltonian flow of $\sigma_\hbar(\rkerr^2P_\hbar) = \rkerr^2\sigma_\hbar(P_\hbar)$ restricted to $\Sigma_\hbar$ is just a reparametrization of the Hamiltonian flow of $\sigma_\hbar(P_\hbar)$.

We will focus on the flow in $T^*X_\beta \subset \Bar{T}^*X_\beta$, since the semiclassical Hamiltonian flow at fiber infinity, $\partial\Bar{T}^*X_\beta$, coincides with the classical Hamiltonian flow studied in Section \ref{section_hamiltonian_flow}. 
The semiclassical principal symbol is given away from complex scaling by
\begin{equation*}
    p_\hbar := \sigma_\hbar(\rkerr^2P_\hbar) = \mu(\xi-zh)^2 - 2((r^2+a^2)z-a\nu)(\xi-zh) + \Tilde{\kappa},
\end{equation*}
where
$$\Tilde{\kappa} = \eta^2 + \frac{1}{\sin^2(\theta)}(\nu - za\sin^2(\theta))^2.$$
Note that, as in Section \ref{section_hamiltonian_flow}, $\nu$ and $\Tilde{\kappa}$ extend to smooth functions on all of $T^*X_\beta$, even at $\theta = 0,\pi$, where the spherical coordinate system is ill-defined. The Hamiltonian vector field is given by
\begin{equation*}
\begin{split}
    H_{p_\hbar} &= 2\bigl(\mu(\xi-zh) - ((r^2+a^2)z-a\nu)\bigr)\partial_r + \frac{\partial\Tilde{\kappa}}{\partial\eta}\partial_\theta - \frac{\partial\Tilde{\kappa}}{\partial\theta}\partial_\eta + \bigl(2a(\xi-zh) + \frac{\partial\Tilde{\kappa}}{\partial\nu}\bigr)\partial_{\varphi_*} \\
    &-2\bigl((r-m)(\xi-zh)^2 - 2zr(\xi-zh) - zh'(\mu(\xi-zh) - ((r^2+a^2)z-a\nu))\bigr)\partial_\xi
\end{split}
\end{equation*}
Evidently, $p_\hbar$, $\nu$ and $\Tilde{\kappa}$ are conserved under the Hamiltonian flow. We note that the conserved quantity $\Tilde{\kappa}$ corresponds to Carter's constant \cite{carter}.

For fixed values of the conserved quantities $\nu$ and $\Tilde{\kappa}$, we can analyse the Hamiltonian flow in the $r-\xi$ plane (the dynamics on $T^*\Sph^2$ do not influence the flow in $r$ and $\xi$). Away from the black hole horizon at $r=r_+$, the characteristic set is given by
\begin{equation}
\label{semiclassical_xi_roots}
    \xi_{\pm}-zh = \frac{1}{\mu}\bigl((r^2+a^2)z-a\nu \pm \sqrt{\mu(V_\nu(r) - \Tilde{\kappa})}\bigr), \quad V_\nu(r) = \frac{((r^2+a^2)z-a\nu)^2}{\mu}.
\end{equation}
That is, given $\nu$ and $\Tilde{\kappa}$, the semiclassical characteristic set in $T^*X_\beta$ is non-empty over $r$, if and only if 
$$\mu(V_\nu(r) - \Tilde{\kappa}) \geq 0,$$
and, in that case, it is located at $\xi(r) = \xi_\pm(r)$. The action of the Hamiltonian vector field on the functions $r$ and $(\xi-zh)$ is given on the characteristic set by
\begin{equation}
\label{semiclassical_hamiltonian_r_xi}
\begin{split}
    H_{p_\hbar}r &= 2\bigl(\mu(\xi-zh) - ((r^2+a^2)z-a\nu)\bigr) = \pm 2\sqrt{\mu(V_\nu(r) - \Tilde{\kappa})}, \\
    H_{p_\hbar}(\xi-zh) &= -2(r-m)(\xi-zh)^2 + 4zr(\xi-zh).
\end{split}
\end{equation}
Recall that $h(r)$ is a bounded function, so $\xi-zh$ just amounts to a shift in the fiber variable $\xi$ by a finite $r$-dependent quantity. We see that the solution $\xi_+$ to \eqref{semiclassical_xi_roots} gives the outward moving branch, while $\xi_-$ gives the inward moving branch. Furthermore, $H_{p_\hbar}r=0$ precisely when $\Tilde{\kappa} = V_\nu(r)$. 

Over the black hole horizon, the semiclassical characteristic set in $T^*X_\beta$ is non-empty unless $(r_+^2+a^2)z-a\nu=0$ and is located at
\begin{equation}
\label{xi_over_horizon}
    \xi - zh = \frac{\Tilde{\kappa}}{2((r_+^2+a^2)z-a\nu)}.
\end{equation}
The action of the Hamiltonian vector field on $r$ is given at $r=r_+$ by
\begin{equation}
\label{r_dot_over_horizon}
    H_{p_\hbar}r = 2((r_+^2+a^2)z-a\nu).
\end{equation}

\subsubsection{Radial sets at fiber infinity over the horizon}
\label{section_semiclassical_radial_sets}
We still have the radial source and radial sink over the horizon, denoted $\Lambda_\pm$ in Section \ref{section_hamiltonian_flow}, but located now at fiber infinity in the semiclassical setting (where the semiclassical flow is the same as the classical one). That is, 
$$L_+ = \Lambda_+\cap\partial\overline{T}^*M \quad\text{and}\quad L_- = \Lambda_-\cap\partial\overline{T}^*M$$
are a semiclassical radial source and radial sink respectively, see \cite[Section 2.8]{vasy_KdS} or \cite[Section E.4.3]{zworski_semiclassical_book}. Note that $L_+$ can be defined within the characteristic set on $\brac{\xi}^{-1}\xi > 0$ by $|\xi|^{-1} = 0$, $|\xi|^{-2}\kappa = 0$ and the same holds for $L_-$ within the characteristic set on $\brac{\xi}^{-1}\xi < 0$. Here, 
$$\kappa = \Tilde{\kappa} + 2az\nu - a^2\sin^2(\theta)$$
is the conserved quantity of the classical Hamiltonian flow of Section \ref{section_hamiltonian_flow} and all expressions are understood in terms of their extension from $T^*M$ to $\overline{T}^*M$. For our purposes it will be more convenient to use $|\xi-zh|^{-1}$ and $|\xi-zh|^{-2}\Tilde{\kappa}$ as defining functions for the semiclassical radial sets. Notice that these expressions are indeed smooth functions in a neighborhood of $L_\pm$, whose vanishing defines $L_\pm$ within the characteristic set in such a neighborhood. The action of the Hamiltonian vector field is given in $\pm(\xi-zh) > 0$ by
\begin{equation*}
\begin{split}
    |\xi-zh|^{-1}H_{p_\hbar}|\xi-zh|^{-1} &= \bigl(\pm2(r-m)-4zr|\xi-zh|^{-1}\bigr)|\xi-zh|^{-1}, \\
    |\xi-zh|^{-1}H_{p_\hbar}(|\xi-zh|^{-2}\Tilde{\kappa}) &= \bigl(\pm4(r-m) - 8zr|\xi-zh|^{-1}\bigr)|\xi-zh|^{-2}\Tilde{\kappa}.
\end{split}
\end{equation*}
Thus, once $(\xi-zh)$ is sufficiently large the Hamiltonian flow within $\Sigma_\hbar$ must exponentially approach $L_+$ in the backward direction, and once $-(\xi-zh)$ is sufficiently large the flow within $\Sigma_\hbar$ must exponentially approach $L_-$ in the forward direction.

\subsubsection{The characteristic set within the horizon}

Inside the black hole horizon, there are always two distinct real solutions for $\xi$ in \eqref{semiclassical_xi_roots}. Thus, the operator $P_\hbar$ is strictly semiclassically hyperbolic with respect to $r$ in this region, in the sense of \cite[Definition E.55]{dyatlov_zworski_scattering_book}. Furthermore, $r$ is strictly increasing or strictly decreasing along the integral curves of the Hamiltonian vector field wihtin $\Sigma_\hbar$.
\begin{lemma}
\label{lemma_semiclassically_hyperbolic}
    $P_\hbar$ is strictly semiclassically hyperbolic with respect to $r$ in $\{r<r_+\}$. Moreover, $H_{p_\hbar}r \neq 0$ on the semiclassical characteristic set within $\{r<r_+\}$.
\end{lemma}
\begin{proof}
    Since $\mu<0$ in $r<r_+$, we certainly have
    $$\mu(V_\nu(r) - \Tilde{\kappa}) = ((r_+^2+a^2)z-a\nu)^2 - \mu\Tilde{\kappa} \geq 0, \quad\text{for}\quad r<r_+.$$
    For equality to hold, we would need $\Tilde{\kappa} = 0$ and $a\nu = (r^2+a^2)z$ for some $r<r_+$. However, note that the vanishing of $\Tilde{\kappa}$ implies $a\nu = za^2\sin^2(\theta)$, which is impossible, since
    $$(r^2+a^2)z - za^2\sin^2(\theta) = z(r^2+a^2\cos^2(\theta)) \neq 0.$$
    Thus, $p_\hbar$ always possesses two distinct real roots in $\{r<r_+\}$. Note that the homogeneous degree $2$ part of $p_\hbar$ is just the symbol analysed in Section \ref{section_hamiltonian_flow}. Since  $\mu(V_\nu(r) - \Tilde{\kappa}) > 0$ for all $r<r_+$, equation \eqref{semiclassical_hamiltonian_r_xi} shows that $H_{p_\hbar}r \neq 0$ on $\Sigma_\hbar$.
\end{proof}

\subsubsection{Location of the characteristic set near the horizon}
\label{section_char_set_near_horizon}

We now examine the behavior of the characteristic set as $r \to r_+$ from either side of the horizon. This is sensitive to the sign of the conserved quantity $(r_+^2+a^2)z - a\nu$. Note that
$$\mu(V_\nu(r) - \Tilde{\kappa}) = ((r_+^2+a^2)z - a\nu)^2 - \mu\Tilde{\kappa},$$
so for $(r_+^2+a^2)z - a\nu \neq 0$ there are always two distinct roots in \eqref{semiclassical_xi_roots} sufficiently close to the horizon.
Writing
$$\sqrt{\mu(V_\nu(r) - \Tilde{\kappa})} = \sqrt{((r^2+a^2)z-a\nu)^2 - \mu\Tilde{\kappa}} = |(r^2+a^2)z-a\nu|\sqrt{1 - \frac{\Tilde{\mu\kappa}}{((r^2+a^2)z-a\nu)^2}},$$
we see from \eqref{semiclassical_xi_roots} that near the black hole horizon the characteristic set in $T^*X_\beta$ is located at
\begin{equation}
\label{xi_roots_horizon}
    \xi_{\pm}-zh = \frac{(r^2+a^2)z-a\nu}{\mu} \pm \Bigl(\frac{|(r^2+a^2)z-a\nu|}{\mu} - \frac{\Tilde{\kappa}}{2|(r^2+a^2)z-a\nu|}\Bigr) + \mathcal{O}(\mu).
\end{equation}
Thus, depending on the sign of $(r_+^2+a^2)z-a\nu$, either the outward or inward moving branch tends to fiber infinity at the horizon. The other branch crosses the horizon within $T^*X_\beta$ with $\xi-zh$ approaching the single root of $p_\hbar$ over the horizon, which we denote by
$$\zeta_{\nu,\Tilde{\kappa}} = \frac{\Tilde{\kappa}}{2((r_+^2+a^2)z-a\nu)},$$
see \eqref{xi_over_horizon}. The situation can be summarized as follows:
\begin{equation}
\label{xi_limits_horizon}
\begin{alignedat}{4}
    &\text{for}\,\,\, (r_+^2+a^2)z-a\nu>0:\quad &&\xi_+ - zh \,\to\, +\infty, \quad &&\xi_- - zh \,\to\, \zeta_{\nu,\Tilde{\kappa}} \quad&&\text{as}\quad r\searrow r_+ \\
    &\quad &&\xi_+ - zh \,\to\, -\infty, \quad &&\xi_- - zh \,\to\, \zeta_{\nu,\Tilde{\kappa}} \quad&&\text{as}\quad r\nearrow r_+ \\
    &\text{for}\,\,\, (r_+^2+a^2)z-a\nu<0:\quad &&\xi_+ - zh \,\to\, \zeta_{\nu,\Tilde{\kappa}}, \quad &&\xi_- - zh \,\to\, -\infty \quad&&\text{as}\quad r\searrow r_+ \\
    &\quad &&\xi_+ - zh \,\to\, \zeta_{\nu,\Tilde{\kappa}}, \quad &&\xi_- - zh \,\to\, +\infty \quad&&\text{as}\quad r\nearrow r_+.
\end{alignedat}
\end{equation}
If $(r_+^2+a^2)z-a\nu = 0$, then the principal symbol at $r=r_+$ is given by $p_\hbar = \Tilde{\kappa}$ and we must have $\Tilde{\kappa}>0$, since $\Tilde{\kappa} = r_+^2+a^2-za\nu = 0$ would imply $za\nu = a^2\sin^2(\theta) = (r_+^2+a^2)$, which is impossible. Thus, the characteristic set in $T^*X_\beta$ is empty over the horizon (although there is of course still the characteristic set at fiber infinity). Moreover, we have $V_\nu(r_+) = 0$ in this case, see Lemma \ref{lemma_potential} below, so there is no characteristic set in $r>r_+$ near the horizon. Within the horizon, there are still the two branches of the characteristic set posited in Lemma \ref{lemma_semiclassically_hyperbolic}. Writing $(r^2+a^2)z-a\nu = z(r^2-r_+^2)$ and $\mu = (r-r_-)(r-r_+)$, we see from \eqref{semiclassical_xi_roots} that
$$\xi_\pm - zh = z\frac{r+r_+}{r-r_-} \pm \frac{-1}{\sqrt{-\mu}}\sqrt{\Tilde{\kappa}-V_\nu(r)}.$$
So as $r$ approaches the horizon from below, we find
\begin{equation*}
    \xi_+ - zh \,\to\, -\infty, \quad \xi_- - zh \,\to\, +\infty \quad\text{as}\quad r\nearrow r_+.
\end{equation*}

\subsubsection{Hamiltonian flow in the exterior region}
\label{section_semiclassical_flow}

To understand the Hamiltonian flow in ${\{r>r_+\}}$, we must first examine the function $V_\nu(r)$ in \eqref{semiclassical_xi_roots}. Note that the proof of the following lemma is similar to the corresponding calculation in the Kerr-de Sitter setting performed in \cite[Theorem 3.2]{petersen_vasy_full_subextremal}.

\begin{lemma}
\label{lemma_potential}
    Consider the function
    $$V_\nu(r) = \frac{((r^2+a^2)z-a\nu)^2}{\mu},$$
    for fixed $\nu\in\R$ and $z = \pm 1$. If $\nu$ is such that $(r_+^2+a^2)z-a\nu = 0$, then $V_\nu(r_+) = 0$ and $V_\nu$ is strictly increasing in $(r_+,\infty)$. Otherwise, $V_\nu$ has exactly one critical point $r_\mathrm{min}(\nu)$ in $(r_+,\infty)$ and $V''_\nu(r_\mathrm{min}) > 0$.
\end{lemma}
\begin{proof}
    The derivative of $V_\nu$ can be written
    $$V_\nu'(r) = -\frac{(r^2+a^2)z-a\nu}{\mu^2}f(r), \quad\text{with}\quad f(r) = ((r^2+a^2)z-a\nu)\mu' - 4rz\mu.$$
    Thus, any critical point in $(r_+,\infty)$ must satisfy either
    $$(r^2+a^2)z-a\nu = 0, \quad\text{or}\quad f(r)=0.$$
    In the first case, we have $f(r) = -4rz\mu$ and
    $$V_\nu''(r) = -\frac{2rz}{\mu^2}f(r) = \frac{8r^2}{\mu} > 0.$$
    In the second case, we have $(r^2+a^2)z-a\nu = \frac{4rz\mu}{\mu'}$ and
    \begin{equation*}
        f'(r) = ((r^2+a^2)z-a\nu)\mu'' - 2rz\mu' - 2z\mu = \frac{2z}{\mu'}\bigl(2r\mu\mu''- r(\mu')^2 - \mu\mu'\bigr).
    \end{equation*}
    Note that $\mu'(r) > 0$ on $(r_+,\infty)$. Thus, in this case
    $$V_\nu''(r) = -\frac{4rz}{\mu\mu'}f'(r) = \frac{8r}{\mu(\mu')^2}\bigl(r(\mu')^2 + \mu\mu' - 2r\mu\mu''\bigr).$$
    A direct calculation in terms of $\mu(r) = r^2 - 2mr + a^2$ shows that
    $$r(\mu')^2 + \mu\mu' - 2r\mu\mu'' = 2g(r), \quad\text{with}\quad g(r) = r^3 - 3mr^2 + (4m^2-a^2)r - ma^2.$$
    Now $g(m) = 2m^3 - 2ma^2 > 0$ and $g'(r) > 0$ everywhere, so $g(r) > 0$ on $(r_+,\infty)$. This shows that $V_\nu''(r) > 0$ also in the second case.

    Thus, any critical point of $V_\nu$ in $(r_+,\infty)$ must be a strict local minimum, implying that there can be at most one such critical point. Note that $V_\nu(r) \to \infty$ as $r\to\infty$. If ${(r^2+a^2)z-a\nu \neq 0}$, then also $V_\nu(r) \to \infty$ as $r\to r_+$ from above, so there must be such a unique minimum. If on the other hand $(r_+^2+a^2)z-a\nu = 0$, then we can write
    $$V_\nu(r) = \frac{(r^2-r_+^2)^2}{\mu} = \frac{(r-r_+)(r+r_+)^2}{r-r_-},$$
    where $r_- < r_+$ are the two zeros of $\mu$. Thus, $V_\nu(r_+) = 0$ and $V_\nu'(r_+) = \frac{(2r_+)^2}{r_+-r_-} > 0$, which implies that $V_\nu'(r) > 0$ everywhere, since otherwise there would be a critical point that is not a local minimum.
\end{proof}

In terms of the conserved quantities $\Tilde{\kappa}$ and $\nu$, the semiclassical characteristic set can be divided into three qualitatively different regimes. For $\Tilde{\kappa} < V_\nu(r_{\mathrm{min}})$ all integral curves in $\{r>r_+\}$ either enter from or escape to spatial infinity, and the characteristic set can be split into inward and outward moving integral curves. For $\Tilde{\kappa} > V_\nu(r_{\mathrm{min}})$, the characteristic set in $\{r>r_+\}$ can be split into integral curves with $r$ bounded from above and integral curves with $r$ bounded from below, which never reach the horizon. Finally, for $\Tilde{\kappa} = V_\nu(r_{\mathrm{min}})$ trapping occurs within the characteristic set.

Here, we take $V_\nu(r_{\mathrm{min}}) = V_\nu(r_+) = 0$ when $(r_+^2+a^2)z-a\nu=0$, as suggested by Lemma \ref{lemma_potential}. For $(r_+^2+a^2)-za\nu < 0$, there exists $r>r_+$ with $((r^2+a^2)z-a\nu)^2 = 0$ and, as $V_\nu(r) \geq 0$ for $r>r_+$, this must be the minimum of Lemma \ref{lemma_potential}. Thus, also for $(r_+^2+a^2)-za\nu < 0$, we have $V_\nu(r_{\mathrm{min}}) = 0$. This implies that $(r_+^2+a^2)-za\nu \leq 0$ is only possible in the case $\Tilde{\kappa} > V_\nu(r_{\mathrm{min}})$, since $\Tilde{\kappa} = 0$ would lead to the contradictory statements $za\nu = a^2\sin^2(\theta)$ and $za\nu \geq (r_+^2+a^2)$.

We will now submit the three regimes to a more careful examination. In the case ${\Tilde{\kappa} < V_\nu(r_{\mathrm{min}})}$, we have
$$\mu(V_\nu(r) - \Tilde{\kappa}) > 0, \quad \forall r>r_+,$$
so \eqref{semiclassical_xi_roots} has two distinct real solutions for all $r>r_+$. The characteristic set in this regime has the two connected components
$$\Sigma_\hbar\cap\{\Tilde{\kappa} < V_\nu(r_{\mathrm{min}})\}\cap\{r>r_+\} = \Sigma_{\mathrm{in}} \cup \Sigma_{\mathrm{out}},$$
where
\begin{equation*}
\begin{split}
    \Sigma_{\mathrm{in}} &= \{(\xi-zh) < \mu^{-1}((r^2+a^2)z-a\nu)\} \cap \Sigma_\hbar \cap\{r>r_+\},\\
    \Sigma_{\mathrm{out}} &= \{(\xi-zh) > \mu^{-1}((r^2+a^2)z-a\nu)\} \cap \Sigma_\hbar \cap\{r>r_+\}.
\end{split}
\end{equation*}
By \eqref{semiclassical_hamiltonian_r_xi}, we have $H_{p_\hbar}r < 0$ everywhere on $\Sigma_{\mathrm{in}}$ and $H_{p_\hbar}r > 0$ on $\Sigma_{\mathrm{out}}$. In fact, $H_{p_\hbar}r$ is bounded away from zero on $\{r\geq r_++\delta\}$ for any $\delta>0$ (and $\nu,\Tilde{\kappa}$ fixed), so integral curves must reach the complex scaling region in finite time, in the forward direction on $\Sigma_{\mathrm{out}}$ and in the backward direction on $\Sigma_{\mathrm{in}}$.

Recall that $z((r_+^2+a^2)z-a\nu) > 0$ in the $\Tilde{\kappa} < V_\nu(r_{\mathrm{min}})$ case. Thus, by \eqref{xi_limits_horizon} we see that, for $z=1$, the inward moving branch crosses the horizon, while the outward moving branch tends to fiber infinity, and for $z=-1$ the reverse is true. Note that over the horizon in $T^*X_\beta$, we have $H_{p_\hbar}r = 2((r_+^2+a^2)z-a\nu) \neq 0$. So, for $z=1$, $H_{p_\hbar}r$ remains negative and bounded away from zero all the way across the horizon in $\Sigma_{\mathrm{in}}$ and integral curves in $\Sigma_{\mathrm{in}}$ cross the horizon in finite time and continue towards the boundary of $X_\beta$ at $r=r_0$. Since $H_{p_\hbar}r > 0$ on $\Sigma_{\mathrm{out}}$, $\xi-zh$ becomes arbitrarily large in the backward direction along the Hamiltonian flow in $\Sigma_{\mathrm{out}}$, and thus must exponentially approach the radial source by the discussion in Section \ref{section_semiclassical_radial_sets}. For $z=-1$ the situation is reversed, that is, integral curves in $\Sigma_{\mathrm{out}}$ cross the horizon in finite time in the backward direction and continue towards $r=r_0$, whereas integral curves in $\Sigma_{\mathrm{in}}$ exponentially approach the radial sink in the forward direction. As can be seen from \eqref{semiclassical_xi_roots} and \eqref{xi_limits_horizon}, there is an additional component of the characteristic set contained entirely in $\{r<r_+\}$, which is outward flowing for $z=1$ and inward flowing for $z=-1$ and approaches the radial sink ($z=1$) respectively the radial source ($z=-1$). The flow in this regime is depicted for $z=1$ in Figure \ref{fig_flow1}.

\begin{figure}[h]
    \centering
    \includegraphics{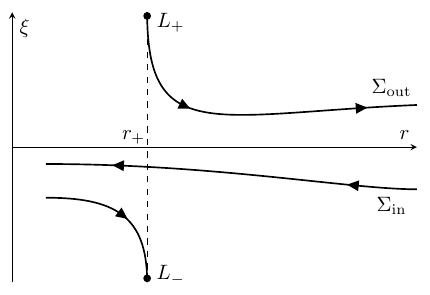}
    \caption{Schematic representation in the $r-\xi$ plane (for fixed values of the conserved quantities $\nu,\Tilde{\kappa}$) of the flow of $H_{p_\hbar}$ on $\Sigma_\hbar$ in the case $z=1$ and $\Tilde{\kappa} < V_\nu(r_{\mathrm{min}})$.}
    \label{fig_flow1}
\end{figure}

We now turn to the case $\Tilde{\kappa} > V_\nu(r_{\mathrm{min}})$. Here, there is no restriction on the sign of $z((r_+^2+a^2)z-a\nu))$. When this quantity is non-zero, $r_\mathrm{min} \in (r_+,\infty)$ and there is an open interval around $r_{\mathrm{min}}$, where $\mu(V_\nu(r) - \Tilde{\kappa}) < 0$. By \eqref{semiclassical_xi_roots} $\Sigma_\hbar$ is empty over this interval, so the characteristic set in $r>r_+$ has the following two connected components
$$\Sigma_\hbar\cap\{\Tilde{\kappa} > V_\nu(r_{\mathrm{min}})\}\cap\{r>r_+\} = \Sigma_\infty \cup \Sigma_\mathrm{hor},$$
where
\begin{equation*}
    \Sigma_\infty = \{r>r_{\mathrm{min}}\}\cap\Sigma_\hbar, \quad \Sigma_\mathrm{hor} = \{r_+<r<r_{\mathrm{min}}\}\cap\Sigma_\hbar.
\end{equation*}
Recall from Lemma \ref{lemma_potential} that $V_\nu(r) \to \infty$ as $r\to\infty$ and $r\to r_+$ with $V_\nu$ strictly decreasing for $r<r_{\mathrm{min}}$ and strictly increasing for $r>r_{\mathrm{min}}$, so there are exactly two values of $r$, where $\mu(V_\nu(r) - \Tilde{\kappa}) = 0$. At these points we have $H_{p_\hbar}r = 0$ by \eqref{semiclassical_hamiltonian_r_xi} and 
$$\xi-zh = \frac{1}{\mu}((r^2+a^2)z-a\nu).$$
The second derivative of $r$ along the Hamiltonian flow is given on the set $\{H_{p_\hbar}r = 0\}$ by
\begin{equation}
\label{second_derivative_r}
\begin{split}
    H_{p_\hbar}^2r &= 2\mu H_{p_\hbar}(\xi-zh) = -4(r-m)\mu(\xi-zh)^2 + 8zr\mu(\xi-zh) \\
    &= -4(r-m)\frac{1}{\mu}((r^2+a^2)z-a\nu)^2 + 8zr\frac{1}{\mu}((r^2+a^2)z-a\nu) = 2\mu V_\nu'(r).
\end{split}
\end{equation}
Thus, for an integral curve $\gamma(t)$ in $\Sigma_\infty$, the function $r(t)=r(\gamma(t))$ has a strict minimum and no other critical points, whereas, for $\gamma(t)$ in $\Sigma_\mathrm{hor}$, $r(t)$ has a strict maximum and no other critical points. Note that at these turning points the $\xi_+$ solution of \eqref{semiclassical_xi_roots} becomes the $\xi_-$ solution and vice versa.
Away from the turning points and the black hole horizon, $\dot r$ is bounded away from zero, so integral curves in $\Sigma_\infty$ enter the complex scaling region in finite time in both the backward and forward direction, and integral curves in $\Sigma_\mathrm{hor}$ enter $\{r<r_++\delta\}$ for any $\delta>0$ in both directions.

To understand the behavior near the horizon, we examine equation \eqref{xi_limits_horizon}. In the case ${z((r_+^2+a^2)z-a\nu) > 0}$ the situation is as follows: for $z=1$, integral curves in $\Sigma_\mathrm{hor}$ approach the radial source in the backward direction and cross the horizon in finite time in the forward direction, whereas, for $z=-1$, they approach the radial sink in the forward direction and cross the horizon in the backward direction. As before there is the additional component of the characteristic set contained entirely in $\{r<r_+\}$.

\begin{figure}[b]
    \centering
    \includegraphics{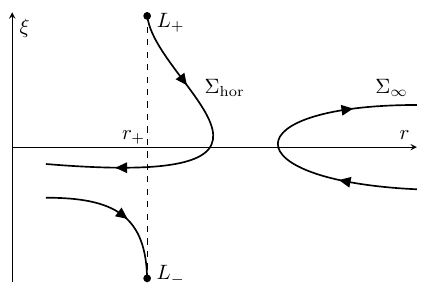}
    \includegraphics{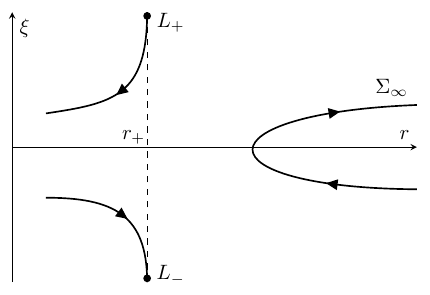}
    \includegraphics{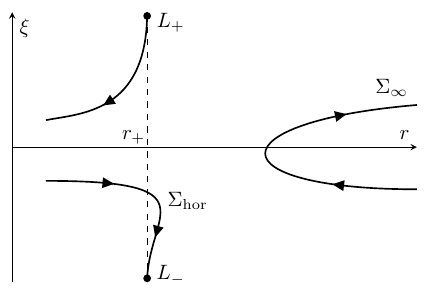}
    \caption{Schematic representation of the flow of $H_{p_\hbar}$ on $\Sigma_\hbar$ for $z=1$ and $\Tilde{\kappa} > V_\nu(r_{\mathrm{min}})$. Top left: $r_+^2+a^2-a\nu > 0$. Top right: $r_+^2+a^2-a\nu = 0$. Bottom: $r_+^2+a^2-a\nu < 0$, in which case $\Sigma_\mathrm{hor}$ is confined to the ergoregion.}
    \label{fig_flow2}
\end{figure}

When $z((r_+^2+a^2)z-a\nu) < 0$, the roles of $z=1$ and $z=-1$ are exactly reversed. Thus, integral curves in $\Sigma_\mathrm{hor}$ flow from the radial sink across the horizon in the backward direction for $z=1$, and from the radial source across the horizon in the forward direction for $z=-1$. Note that, if such an integral curve entered the complex scaling region, it would pose a problem for propagation estimates, since we would be unable to obtain estimates by propagating in the direction imposed by the sign of $z$. However, we claim that these integral curves are confined to the ergoregion, i.e. $\{\mu + a^2\sin^2(\theta) \leq 0\}$. Then, choosing $R_0$ large enough, we can ensure that they do not enter the complex scaling region. Indeed, for $r_+^2+a^2-za\nu < 0$, the minimum of $V_\nu(r)$ is given by $r^2+a^2-za\nu = 0$, so $\Sigma_\mathrm{hor}$ remains entirely in the region $\{r^2+a^2-za\nu < 0\}$. On the characteristic set in $r>r_+$, we must have
\begin{equation*}
    \frac{1}{\mu}\bigl(r^2+a^2-za\nu\bigr)^2 = V_\nu(r) \geq \Tilde{\kappa} \geq \frac{1}{a^2\sin^2(\theta)}\bigl(a^2\sin^2 - za\nu\bigr)^2.
\end{equation*}
When $\mu \geq a^2\sin^2(\theta)$, this implies
\begin{equation*}
    (r^2+a^2-za\nu)^2 \geq (a^2\sin^2 - za\nu)^2 = (r^2+a^2-za\nu)^2 - 2\rkerr^2(r^2+a^2-za\nu) + \rkerr^4,
\end{equation*}
where $\rkerr^2 = r^2+a^2\cos^2(\theta) > 0$. Therefore, we must have $r^2+a^2-za\nu > 0$ on the characteristic set outside the ergoregion.

Finally, when $(r_+^2+a^2)z-a\nu = 0$, there is no characteristic set near the horizon in $\{r>r_+\}$, see Section \ref{section_char_set_near_horizon}. Furthermore, $V_\nu(r_+)=0$ and $V_\nu(r)$ is strictly increasing by Lemma \ref{lemma_potential}. So there is only one turning point with $\Tilde{\kappa}= V_\nu(r)$, which is contained in $\Sigma_\infty$, and the component $\Sigma_\mathrm{hor}$ is empty. There are now two components of the characteristic set contained entirely in $\{r<r_+\}$. By the discussion at the end of Section \ref{section_char_set_near_horizon}, the Hamiltonian flow on these two components approaches the radial source and radial sink respectively. The flow in the various cases is depicted in Figure \ref{fig_flow2} for $z=1$.

\begin{rmk}
\label{rmk_geodesics}
    The Hamiltonian flow of the principal symbol is related to the geodesic flow of the Kerr metric. Indeed, the semiclassical principal symbol (away from complex scaling) is just a rescaling of the inverse metric, viewed as a function on the cotangent bundle (i.e. the norm-squared of $zdt_*+\xi dr+\eta d\theta + \nu d\varphi_*$):
    $$p_\hbar(r,\theta,\varphi_*,\xi,\eta,\nu) = \rkerr^2 g^{-1}(t_*,r,\theta,\varphi_*,z,\xi,\eta,\nu).$$
    Note that this is independent of $t_*$ by the stationarity of the Kerr metric. The parameter $z=\pm 1$ corresponds to $\tau$, the momentum associated to the coordinate $t_*$. The semiclassical characteristic set can be understood as the union of lightcones in cotangent space intersected with the hypersurface $\{t_*=0,\tau=\pm1\}$. Moreover, integral curves of the Hamiltonian flow on the characteristic set are just null-geodesics lifted to the cotangent bundle and projected to the hypersurface $\{t_*=0,\tau=\pm1\}$. From this perspective, the presence of the integral curves flowing in the ``wrong direction'' above, e.g. exiting the horizon in the $z=+1$ case, seems physically odd. However, recall that $\partial_{t_*}$ fails to be time-like inside the ergoregion, which also provides a physical explanation of why such integral curves are confined to this region.
\end{rmk}

We now consider the case $\Tilde{\kappa} = V_\nu(r_\mathrm{min})$, where trapping occurs. At $r=r_\mathrm{min}$, we have $H_{p_\hbar}r=0$ by \eqref{semiclassical_hamiltonian_r_xi} and $H_{p_\hbar}(\xi-zh) = V_\nu'(r_\mathrm{min}) = 0$ by \eqref{second_derivative_r}. Thus, $r$ and $\xi$ remain constant for all times along integral curves at $r_\mathrm{min}$, although there can still be flow in the $\theta,\varphi_*,\eta$ directions. Away from $r=r_\mathrm{min}$, we have $V_\nu(r) > \Tilde{\kappa}$, and hence $H_{p_\hbar}r\neq 0$. Thus, the trapped set of the Hamiltonian flow within the characteristic set is given by
\begin{equation*}
    \Gamma = \bigl\{r=r_\mathrm{min}(\nu),\,\, \xi-zh = \mu^{-1}((r^2+a^2)z-a\nu),\,\, \Tilde{\kappa}=V_\nu(r_\mathrm{min})\bigr\}.
\end{equation*}
We write $\Gamma_s$ and $\Gamma_u$ for the forward and backward trapped sets within the characteristic set. $\Gamma_s$ is given by the $\xi_+$ branch of \eqref{semiclassical_xi_roots} in $r_+<r<r_\mathrm{min}$ and by the $\xi_-$ branch in $r>r_\mathrm{min}$, and the opposite holds for $\Gamma_u$. Notice that the Hamiltonian flow in $\Gamma_s$ indeed approaches the trapped set in the forward direction, while the flow in $\Gamma_u$ approaches the trapped set in the backward direction. In $\{r > r_\mathrm{min}+\delta\}$ for any $\delta>0$, $H_{p_\hbar}r$ is bounded away from zero and the Hamiltonian flow enters the complex scaling region in finite time in the forward direction along $\Gamma_u$ and in the backward direction along $\Gamma_s$. The approach towards the horizon is described by \eqref{xi_limits_horizon}. Recall that we must have $z((r_+^2+a^2)z-a\nu) > 0$ when $\Tilde{\kappa}=V_\nu(r_{\mathrm{min}})$. Thus, for $z=1$, the flow in $\Gamma_u$ crosses the horizon in the forward direction and in $\Gamma_s$ approaches the radial source in the backward direction, while for $z=-1$, the flow in $\Gamma_s$ crosses the horizon in the backward direction and in $\Gamma_u$ approaches the radial sink in the forward direction. The flow for $z=1$ is depicted in Figure \ref{fig_flow3}.

\begin{figure}[h]
    \centering
    \includegraphics{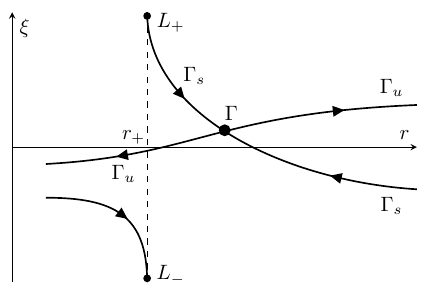}
    \caption{Schematic representation of the flow of $H_{p_\hbar}$ on $\Sigma_\hbar$ in the case $z=1$ and $\Tilde{\kappa} = V_\nu(r_{\mathrm{min}})$. The trapped set $\Gamma$ is located at a singe point in the $r-\xi$ plane (for fixed values of $\nu,\Tilde{\kappa}$). $\Gamma_s$ is the forward trapped set and $\Gamma_u$ the backward trapped set.}
    \label{fig_flow3}
\end{figure}

It is known that the trapped set for Kerr is normally hyperbolic, see the end of Section \ref{section_semiclassical} for a discussion of this notion. This was shown by Wunsch and Zworski in the case of small angular momentum, see \cite[Proposition 2.1]{wunsch_zworski_trapping}, and then extended by Dyatlov to the entire subextremal range, see \cite[Section 3.2]{dyatlov_ringdown}. (In fact, they show the stronger notion of r-normal hyperbolicity for each r.)

\begin{lemma}[{\cite[Proposition 3.2]{dyatlov_ringdown}}]
\label{kerr_normally_hyperbolic}
    The trapped set for the Hamiltonian flow of $p_\hbar$ on $\Sigma_\hbar$ is normally hyperbolic.
\end{lemma}

Note that \cite[Proposition 3.1]{dyatlov_ringdown} shows that the trapped set $\Gamma$ is a compact subset of $\{r>r_+\}$. Thus, we can choose $R_0$ large enough to ensure that the trapped set is disjoint from the complex scaling region. Defining functions of $\Gamma_u$, $\Gamma_s$ inside $\Sigma_\hbar$ are given explicitly by
\begin{equation*}
\begin{split}
    \varphi_u &= \xi - zh - \frac{(r^2+a^2)z-a\nu}{\mu} - \mathrm{sgn}(r-r_\mathrm{min})\frac{1}{\sqrt{\mu}}\sqrt{V_\nu(r)-V_\nu(r_\mathrm{min})} \\
    \varphi_s &= \xi - zh - \frac{(r^2+a^2)z-a\nu}{\mu} + \mathrm{sgn}(r-r_\mathrm{min})\frac{1}{\sqrt{\mu}}\sqrt{V_\nu(r)-V_\nu(r_\mathrm{min})},
\end{split}
\end{equation*}
see also \cite[p. 30]{dyatlov_ringdown}, where different notation is used. As ${V_\nu(r)-V_\nu(r_\mathrm{min}) \sim (r-r_\mathrm{min})^2}$ these are indeed smooth functions. In \cite[eq. 3.34]{dyatlov_ringdown} the functions $w_u,w_s$ of condition \eqref{trapping_condition} are calculated and shown to be positive near the trapped set.

Thanks to the normal hyperbolicity of trapping, we can apply the microlocal estimate of \cite[Theorem 4.7]{hintz_vasy_quasilinear} to obtain an estimate at the trapped set with an error term supported away from $\Gamma_u$. Note that the imaginary part of the subprincipal symbol of $P_\hbar$ is proportional to $h^{-1}\Im(z) = \Im(\sigma)$. Thus, the condition in \eqref{expansion_rate} will amount to a bound on the negative imaginary part of the spectral parameter. This bound can be related to the minimal expansion rate of the linearized flow in the unstable/stable directions.

\subsection{Proof of the semiclassical resolvent estimate}

We are now in a position to prove Proposition \ref{prop_semiclassical_estimate}.

\begin{proof}[Proof of Proposition \ref{prop_semiclassical_estimate}]
    We prove the result for $z=1$. The proof for $z=-1$ is entirely analogous with estimates propagated in the opposite direction along the Hamiltonian flow. Our strategy is the same as in the proof of \ref{prop_fredholm_estimates}, using the semiclassical version of the respective estimates. However, the structure of the semiclassical characteristic set, as discussed above, is more complicated. In particular, there is a trapped set. Since the trapping is normally hyperbolic by Lemma \ref{kerr_normally_hyperbolic}, we can use \cite[Theorem 4.7]{hintz_vasy_quasilinear} to nonetheless propagate estimates into the trapped set. Note that this trapping causes the loss of a factor of $h$ in Proposition \ref{prop_semiclassical_estimate} compared to the propagation and radial estimates. 

    In the region $\{r>R_1\}$, where complex scaling is applied, $P_\hbar$ is elliptic and we can estimate for some $\Tilde{\chi} \in C^\infty_c(X_\beta)$:
    $$\|Au\|_{H_h^s} \leq C \bigl(\|\Tilde{\chi} P_\hbar u\|_{H_h^{s-2}} + h^N\|\Tilde{\chi}u\|_{H_h^{-N}}\bigr)$$
    for any compactly supported $A\in\Psi_\hbar^0(X_\beta)$ with $\WF_\hbar(A) \subset \{r>R_1\}$. Of course, we have a similar semiclassical elliptic estimate off the characteristic of $P_\hbar$ in all of $X_\beta$. Near either of the radial sets, we can apply semiclassical radial estimates, see \cite[Theorem E.52]{dyatlov_zworski_scattering_book}, to obtain
    $$\|Au\|_{H_h^s} \leq C \bigl(h^{-1}\|\Tilde{\chi} P_\hbar u\|_{H_h^{s-1}} + h^N\|\Tilde{\chi}u\|_{H_h^{-N}}\bigr),$$
    for some compactly supported $A\in\Psi_\hbar^0(X_\beta)$ with the respective radial set contained in $\Ell_\hbar(A)$. We will propagate these estimates to the entire characteristic set $\Sigma_\hbar$, being careful to only propagate in the forward direction along integral curves of $H_{p_\hbar}$ that enter the region $\{r>R_1\}$, as required by the sign of $\Im(\sigma_\hbar(P_\hbar))$ there, see \cite[Theorem E.47]{dyatlov_zworski_scattering_book}. Note that components of the characteristic set contained entirely inside the horizon can be dealt with easily, since there is no constraint on the propogation direction there. Integral curves on such components flow from one of the radial sets towards the boundary of $X_\beta$ at $r=r_+$, so the radial estimates can be propogated throughout this part of the characteristic set.

    For the rest of the characteristic set, we separate our analysis into the three regimes of the conserved quantities $\Tilde{\kappa},\nu$ discussed above. In $\{\Tilde{\kappa} < V_\nu(r_{\mathrm{min}})\}$ the characteristic set has the components $\Sigma_{\mathrm{in}}$ and $\Sigma_{\mathrm{out}}$, see Figure \ref{fig_flow1}. On $\Sigma_{\mathrm{in}}$ we can propagate the elliptic estimate forward from $\{r>R_1\}$ towards the boundary at $r=r_0$. On $\Sigma_{\mathrm{out}}$, we can propagate the radial estimate forward from the radial source into the region $\{r>R_1\}$. Now consider $\{\Tilde{\kappa} > V_\nu(r_{\mathrm{min}})\}$, where the characteristic set has the components $\Sigma_\infty$ and $\Sigma_\mathrm{hor}$, see Figure \ref{fig_flow2}. On $\Sigma_\infty$, the elliptic estimate can be propagated forward along the flow from $\{r>R_1\}$ all the way back into $\{r>R_1\}$. On $\Sigma_\mathrm{hor}$ the dynamics depends on the sign of the quantity $r_+^2+a^2-a\nu$. For $r_+^2+a^2-a\nu > 0$, we propagate the radial estimate forward along the flow in $\Sigma_\mathrm{hor}$ from the radial source towards $r=r_0$. In the case $r_+^2+a^2-a\nu < 0$, we must propagate backwards along the Hamiltonian flow, from the radial sink towards $r=r_0$. By the discussion in Section \ref{section_semiclassical_flow} this is permitted, since such integral curves remain in the ergoregion with $\mu<a^2\sin^2(\theta)$ and hence do not enter the complex scaling region, so $\Im(\sigma_\hbar(P_\hbar))=0$ along the flow.

    Finally, when $\Tilde{\kappa} = V_\nu(r_{\mathrm{min}})$ we must deal with the trapped set, see Figure \ref{fig_flow3}. On the forward trapped set $\Gamma_s$, we can propagate the elliptic estimate forward along the flow from $\{r>R_1\}$ towards the trapped set, and the radial estimate from the radial source towards the trapped. Together with the discussion above, this gives the estimate
    $$\|Au\|_{H_h^s} \leq C \bigl(h^{-1}\|\Tilde{\chi} P_\hbar u\|_{H_h^{s-1}} + h^N\|\Tilde{\chi}u\|_{H_h^{-N}}\bigr),$$
    everywhere except on the backward trapped set $\Gamma_u$, that is, for any $A\in\Psi_\hbar^0(X_\beta)$ with $\WF_\hbar(A)\cap \Gamma_u =\emptyset$. 
    Since the trapped set is normally hyperbolic by Proposition \ref{kerr_normally_hyperbolic}, we can use \cite[Theorem 4.7]{hintz_vasy_quasilinear} to estimate
    \begin{equation}
    \label{trapping_estimate_kerr}
        \|Bu\|_{H_h^s} \leq C \bigl(h^{-2}\|\Tilde{\chi} P_\hbar u\|_{H_h^{s-1}} + h^{-1}\|Eu\|_{H_h^s} + h^N\|\Tilde{\chi}u\|_{H_h^{-N}}\bigr),
    \end{equation}
    where $\Ell_\hbar(B)$ contains the trapped set and $\WF_\hbar(E)$ is disjoint from $\Gamma_u$. Of course this estimate requires the additional condition on the imaginary part of the subprincipal symbol of $P_\hbar$, see \eqref{expansion_rate}. A direct calculation shows that
    \begin{equation*}
        \sigma_\hbar\Bigl(\frac{1}{2ih}(P_\hbar-P_\hbar^*)\Bigr) = -2h^{-1}\Im(z)\Bigl(1 + \frac{2mr(r^2+a^2)}{\rkerr^2\mu} - \frac{2mr}{\rkerr^2\mu}a\nu\Bigr).
    \end{equation*}
    Now recall that $r_+^2+a^2-a\nu > 0$ on the trapped set, so we have $a\nu < r^2+a^2$ at trapping. Using this in the expression above, we find
    $$\sigma_\hbar\Bigl(\frac{1}{2ih}(P_\hbar-P_\hbar^*)\Bigr) < -2h^{-1}\Im(z).$$
    Thus, choosing $\Im(z) > -h\gamma$ for some constant $\gamma > 0$ as in the statement of the Proposition, ensures that condition \eqref{expansion_rate} is met.
    Because $\WF_\hbar(E)\cap\Gamma_u = \emptyset$, we can apply the non-trapping estimates to the error term in \eqref{trapping_estimate_kerr} and obtain
    $$\|Bu\|_{H_h^s} \leq C \bigl(h^{-2}\|\Tilde{\chi} P_\hbar u\|_{H_h^{s-1}} + h^N\|\Tilde{\chi}u\|_{H_h^{-N}}\bigr)$$
    microlocally in a neighborhood of the trapped set. This estimate in turn can now be propagated forward along the Hamiltonian flow in $\Gamma_u$ towards $r=r_0$ and $r>R_1$.

    We have now dealt with the entire characteristic set. Combining the various estimates on different parts of $T^*X_\beta$ by a semiclassical partition of unity (and downgrading the factor of $h^{-1}$ to $h^{-2}$ for the non-trapping estimates), gives
    $$\|\chi u\|_{H_h^s} \leq C \bigl(h^{-2}\|\Tilde{\chi} P_\hbar u\|_{H_h^{s-1}} + h^N\|\Tilde{\chi}u\|_{H_h^{-N}}\bigr)$$
    for some $\chi \in C^\infty_c(X_\beta)$ satisfying $\chi = 1$ on $\{r\in (r_0+\delta,R)\}$ for some large $R$. As in the proof of Proposition \ref{prop_fredholm_estimates}, this estimate can be closed by using a semiclassical hyperbolic estimate near $r=r_0$, see \cite[Theorem E.57]{dyatlov_zworski_scattering_book}, and the uniform ellipticity of $P_\hbar$ for large $r$, see Lemma \ref{semiclassical_complex_scaling}.
\end{proof}

\section{Low energy estimates}
\label{section_low_energy}

In this section, we will study the low energy behavior of $P_\beta(\sigma)$ and prove that there is no accumulation of quasinormal modes at $\sigma=0$, Theorem \ref{thm_low_energy}. This will follow from certain uniform estimates near, and down to, $\sigma=0$. We begin by showing that the estimates of Section \ref{section_fredholm_estimates} based on the analysis of the principal symbol hold uniformly for $\sigma$ in compacta.

\begin{lemma}
\label{lemma_uniform_interior_estimate}
    Let $K \subset \C$ be compact and let $s>\frac{1}{2}-\alpha\Im(\sigma)$ for all $\sigma\in K$. Then the following estimate holds uniformly in $\sigma$ for $\sigma \in K$:
    \begin{equation}
    \label{uniform_interior_estimate}
        \|u\|_{\Bar{H}^s(X_\beta)} \leq C\bigl(\|P_\beta(\sigma)u\|_{\Bar{H}^{s-1}(X_\beta)} + \|u\|_{\Bar{H}^{-N}(X_\beta)}\bigr), \quad \forall u \in \mathcal{X}_\beta.
    \end{equation}
\end{lemma}
\begin{proof}
    The principal symbol of $P_\beta(\sigma)$ is independent of $\sigma$. Thus, the microlocal estimates in the proof of Proposition \ref{prop_fredholm_estimates} and the uniform elliptic estimates in $r>R_0$ from the proof of Lemma \ref{lemma_scattering_elliptic} hold for all $\sigma \in \C$, not just $\sigma \in \Lambda_\beta$. It is only the analysis of $P_\beta(\sigma)$ as $r\to\infty$ in Lemma \ref{lemma_scattering_elliptic}, in other words scattering ellipticity, that breaks down for $\sigma \notin \Lambda_\beta$. Thus, we certainly have
    $$\|u\|_{\Bar{H}^s(X_\beta)} \leq C_\sigma\bigl(\|P_\beta(\sigma)u\|_{\Bar{H}^{s-1}(X_\beta)} + \|u\|_{\Bar{H}^{-N}(X_\beta)}\bigr)$$
    pointwise for each $\sigma\in\C$. 

    In order to obtain uniformity in $\sigma$, consider $\sigma_0 \in K$. We have
    $$P_\beta(\sigma) - P_\beta(\sigma_0) = (\sigma-\sigma_0)Q_1 + (\sigma^2-\sigma_0^2)Q_0,$$
    where $Q_1 \in \mathrm{Diff}^1(X_\beta)$, $Q_0 \in \mathrm{Diff}^0(X_\beta)$ do not depend on $\sigma$ and have coefficients extending beyond $r=r_0$ and bounded as $r\to\infty$. Thus, the operator norm satisfies
    $$\|P_\beta(\sigma) - P_\beta(\sigma_0)\|_{\Bar{H}^{s}(X_\beta) \to \Bar{H}^{s-1}(X_\beta)} < \frac{1}{2C}$$
    for all $\sigma$ with $|\sigma - \sigma_0|$ small enough. Applying this to the pointwise estimate, we find
    $$\|u\|_{\Bar{H}^s(X_\beta)} \leq 2C_{\sigma_0}\bigl(\|P_\beta(\sigma)u\|_{\Bar{H}^{s-1}(X_\beta)} + \|u\|_{\Bar{H}^{-N}(X_\beta)}\bigr)$$
    for $\sigma$ near $\sigma_0$. This shows that \eqref{uniform_interior_estimate} holds locally uniformly in $\sigma$, and hence also uniformly on the compact set $K$.
\end{proof}

Note that the error term in \eqref{uniform_interior_estimate}, unlike in the Fredholm estimates of Proposition \ref{prop_fredholm_estimates}, is not compact. The $r\to\infty$ asymptotic behavior of $P_\beta(\sigma)$ in some sense degenerates as $\sigma\to 0$ within $\Lambda_\beta$. In fact, the zero energy operator $P_\beta(0)$ is not well-behaved on the weighted Sobolev spaces $\Bar{H}^{s,l}(X_\beta)$ of Section \ref{section_fredholm_property}, i.e. scattering Sobolev spaces. Rather, $P_\beta(0)$ should be viewed as a b-differential operator and studied on weighted b-Sobolev spaces $\Bar{H}_\mathrm{b}^{s,l}(X_\beta)$. These spaces agree with $\Bar{H}^{s,l}(X_\beta)$ away from $r=\infty$, but regularity is measured with respect to $r\partial_r$ and $\partial_\omega$, where $\omega$ denotes coordinates on $\Sph^2$. In Lemma \ref{lemma_zero_energy_estimate} below, we show that $P_\beta(0)$ is well-behaved on such weighted b-Sobolev spaces. However, at non-zero energies we still wish to work on the scattering Sobolev spaces $\Bar{H}^{s,l}(X_\beta)$. Thus, one should in some sense patch together scattering and b-Sobolev spaces in a continuous manner as $\sigma\to 0$. This is achieved by the scattering-b-transition Sobolev spaces introduced in Section \ref{section_scattering-b-transition}.

Elaborating on this idea, the $r\to\infty$ behavior of $P_\beta(\sigma)$ is quite sensitive to the value of $|\sigma|$ at low energies. This suggests that we consider $P_\beta(\sigma)$ in terms of the rescaled coordinate $\tau = |\sigma|r$. Using coordinates $\tau$ and $\sigma$ on $(X_\beta \times \Lambda_\beta) \cap \{\tau > c\}$, the operator $P_\beta(\sigma)$ is given to leading order both at $\tau = \infty$ and $\sigma = 0$ by
\begin{equation*}
    P_\beta(\sigma) \sim e^{-2i\beta}|\sigma|^2\bigl(\tau^{-2}D_\tau\tau^2D_\tau + \tau^{-2}\Delta_{\Sph^2} -e^{2i(\arg(\sigma)+\beta)}\bigr).
\end{equation*}
Modulo the overall $|\sigma|^2$ decay, this is invertible on scattering Sobolev spaces uniformly in $\sigma$ for $\arg(\sigma)$ bounded away from $-\beta, \pi-\beta$. On the other hand, using coordinates $r$ and $\tau$ in $(X_\beta \times \Lambda_\beta) \cap \{\tau < C\}$, the leading order behavior of $P_\beta(\sigma)$ at both $\tau=0$ and $r=\infty$ becomes
\begin{equation*}
    P_\beta(\sigma) \sim r^{-2}\bigl(r^2P_\beta(0) - \tau^2e^{2i\arg(\sigma)}\bigl).
\end{equation*}
This is a well-behaved weighted b-differential operator, e.g. its b-normal operator is invertible, all the way down to $\tau = 0$. The idea of working with the rescaled coordinate $\tau=|\sigma|r$ down to $\sigma = 0$ and $r=\infty$ is made precise by resolving the point $|\sigma| = \frac{1}{r} = 0$ through a blow-up.

\subsection{Scattering-b-transition calculus}
\label{section_scattering-b-transition}

Here, we introduce the algebra of scattering-b-transition pseudodifferential operators and the corresponding scattering-b-transition Sobolev spaces that will be used in the rest of this section. This pseudodifferential calculus was introduced in \cite{guillarmou_hassell}. We follow \cite[Appendix A.3]{hintz_mode_stability} here, see also \cite{vasy_lagrangian_approach} for a more sophisticated version involving second microlocalization. The constructions take place on a $n$-dimensional manifold with boundary $X$. We denote by $\bdf$ a boundary defining function, that is, a non-negative smooth function on $X$ satisfying $\bdf^{-1}(0) = \partial X$ and $d\bdf \neq 0$ on $\partial X$. We will apply these constructions to subsets of $\R^n$ by adding the sphere at infinity to obtain the manifold with boundary. More precisely, we define the radial compactification of $\R^n$ by
$$\overline{\R^n} = \R^n \sqcup \bigl([0,\infty)_\bdf\times\Sph^{n-1}\bigr) / \sim,$$
where we identify $(r,\omega) \in \R^n$, in polar coordinates, with $(\bdf,\omega) = \bigl(\frac{1}{r},\omega\bigr)$. In this way, $\bdf$ is a boundary defining function on $\overline{\R^n}$ and the image of the embedding $\R^n \hookrightarrow \overline{\R^n}$ is given by $\{\bdf > 0\}$.

The scattering-b-transition calculus is based on the b- and scattering calculi, see \cite{melrose_atiyah-patodi-singer,grieser_b-calculus} for the b-calculus and \cite{melrose_scattering_calculus,vasy_minicourse} for the scattering calculus. We note that these spaces of pseudodifferential operators $\Psi_{\text{b}}^m(X)$ and $\Psi_{\text{sc}}^m(X)$ are microlocalizations of the spaces $\mathrm{Diff}_{\text{b}}^k(X)$ and $\mathrm{Diff}_{\text{sc}}^k(X)$ of b- and scattering differential operators respectively. These, in turn, are generated by the spaces of b-vector fields $\mathcal{V}_{\text{b}}(X)$ respectively scattering vector fields $\mathcal{V}_{\text{sc}}(X)$, where $\mathcal{V}_{\text{b}}(X)$ is the Lie algebra of vector fields on $X$ tangent to the boundary $\partial X$ and $\mathcal{V}_{\text{sc}}(X) = \bdf\mathcal{V}_{\text{b}}(X)$. In local coordinates $(\bdf,\omega)$ near a point in $\partial X$, where $\bdf$ is a boundary defining function and $\omega$ are coordinates on $\partial X$, $\mathcal{V}_{\text{b}}(X)$ is spanned by 
$$\bdf\partial_\bdf \,\,\text{and}\,\, \partial_{\omega^j} \,\,(j = 1,\dots,n-1),$$
and $\mathcal{V}_{\text{sc}}(X)$ is spanned by 
$$\bdf^2\partial_\bdf \,\,\text{and}\,\, \bdf\partial_{\omega^j} \,\,(j = 1,\dots,n-1).$$
Note that on $\overline{X_\beta} \subset \overline{\R^3}$ as above operators in $\mathrm{Diff}_{\text{b}}^k(X_\beta)$, respectively $\mathrm{Diff}_{\text{sc}}^k(X_\beta)$, take the form
$$\sum_{j+|\alpha|\leq k} a_{j,\alpha}(r,\omega)(r\partial_r)^j\partial_\omega^\alpha, \quad\text{respectively}\quad \sum_{j+|\alpha|\leq k} a_{j,\alpha}(r,\omega)\partial_r^j(r^{-1}\partial_\omega)^\alpha,$$
where the coefficients $a_{j,\alpha}$ depend smoothly on $\frac{1}{r}$. We denote by 
$$\mathrm{Diff}_{\text{b}}^{k,l}(M) = \bdf^{-l}\mathrm{Diff}_{\text{b}}^k(M), \quad \mathrm{Diff}_{\text{sc}}^{k,l}(M) = \bdf^{-l}\mathrm{Diff}_{\text{sc}}^k(M)$$ spaces of weighted b-/scattering differential operators. 

Operators in the b-/scattering calculi naturally act on corresponding b-/scattering Sobolev spaces. Let $\mu = \bdf^{-n}\mu_0$ where $\mu_0$ is a smooth b-density on $X$, i.e. locally of the form $\mu_0 = f(\bdf,\omega)\frac{d\bdf}{\bdf}d\omega$ near the boundary with $f\in C^\infty(X)$ bounded away from zero. Then we set $H^0_{\text{b}}(X) = H^0_{\text{sc}}(X) = L^2(X,\mu)$. For $s>0$, we define 
$$H^s_{\text{b}}(X) = \{u\in H^0_{\text{b}}(X),\, Au \in H^0_{\text{b}}(X)\}, \quad \|u\|^2_{H^s_{\text{b}}(X)} = \|u\|^2_{H^0_{\text{b}}(X)} + \|Au\|^2_{H^0_{\text{b}}(X)},$$
where $A$ is a fixed elliptic element of $\Psi^s_{\text{b}}(X)$. The space $H^s_{\text{sc}}(X)$ is defined similarly with $A$ an elliptic element of $\Psi^s_{\text{sc}}(X)$. Note that for $X\subset \overline{\R^n}$ and $k\in\N$ these norms are equivalent to
$$\|u\|^2_{H^k_{\text{b}}(X)} = \sum_{j+|\alpha|\leq k} \|(r\partial_r)^j\partial_\omega^\alpha u\|^2_{L^2(X)}, \quad \|u\|^2_{H^k_{\text{sc}}(X)} = \sum_{j+|\alpha|\leq k} \|\partial_r^j(r^{-1}\partial_\omega)^\alpha u\|^2_{L^2(X)}.$$
In particular, on subsets of $\overline{\R^n}$ scattering Sobolev spaces are just the standard Sobolev spaces on $\R^n$.
For $s<0$ we define $H^s_{\text{b}}(X) = \bigl(H^{-s}_{\text{b}}(X)\bigr)^*$ by duality (with respect to the $L^2(X,\mu)$ inner product), and similarly for $H^s_{\text{b}}(X)$. Finally, for $l\in\R$ we denote weighted b-/scattering Sobolev spaces by
$$H^{s,l}_{\text{b}}(X) = \bdf^l H^s_{\text{b}}(X), \quad H^{s,l}_{\text{sc}}(X) = \bdf^l H^s_{\text{sc}}(X).$$

If the manifold $X$ has two boundary components, i.e. $\partial X = B_\text{sc} \sqcup B_\mathrm{b}$, we can define scattering,b objects with scattering behavior near $B_\text{sc}$ and b-behavior near $B_\mathrm{b}$, see \cite[Section 2.4]{hintz_cone_points}. With $\bdf_\text{sc}$, $\bdf_\mathrm{b}$ denoting defining functions for $B_\text{sc}$ respectively $B_\mathrm{b}$, the corresponding differential operators $\mathrm{Diff}^k_{\text{sc,b}}(X)$ are generated by the space of vector fields $\mathcal{V}_\text{sc,b}(X) = \bdf_\text{sc}\mathcal{V}_\mathrm{b}(X)$. We define weighted operators and Sobolev spaces with a weight at each boundary component:
$$\mathrm{Diff}^{k,q,l}_{\text{sc,b}}(X) = \bdf_\text{sc}^{-q}\bdf_\mathrm{b}^{-l}\mathrm{Diff}^k_{\text{sc,b}}(X), \quad H^{s,q,l}_{\text{sc,b}}(X) = \bdf_\text{sc}^{q}\bdf_\mathrm{b}^{l}H^s_{\text{sc,b}}(X).$$
The particular case we will be interested in is
$X = [0,\infty]_\tau \times \Sph^{n-1},$
where $[0,\infty] \subset \overline{\R}$ is viewed as a subset of the radial compactification of $\R$. A defining function for the b-end at $\tau=0$ is $\bdf_\mathrm{b} = \frac{\tau}{1+\tau}$ and a defining function for the scattering end at $\tau=\infty$ is $\bdf_\text{sc} = \frac{1}{1+\tau}$. Elements of $\mathrm{Diff}^k_{\text{sc,b}}(X)$ then take the form
$$\sum_{j+|\alpha|\leq k} a_{j,\alpha}(\tau,\omega)\Bigl(\frac{\tau}{1+\tau}\partial_\tau\Bigr)^j\Bigl(\frac{1}{1+\tau}\partial_\omega\Bigr)^\alpha,$$
where the coefficients $a_{j,\alpha}$ are smooth down to $\tau=0$ and additionally smooth functions of $\frac{1}{\tau}$ near $\tau=\infty$.
The Sobolev norms for positive integer regularity can be written explicitly as
\begin{equation}
\label{sc,b_integer_norm}
    \|u\|^2_{H^{k,q,l}_{\text{sc,b}}(X)} = \sum_{j+|\alpha|\leq k}\Bigl\|\Bigl(\frac{1}{1+\tau}\Bigr)^{-q}\Bigl(\frac{\tau}{1+\tau}\Bigr)^{-l} \Bigl(\frac{\tau}{1+\tau}\partial_\tau\Bigr)^j\Bigl(\frac{1}{1+\tau}\partial_\omega\Bigr)^\alpha u\Bigr\|^2_{L^2(X,\tau^{n-1}d\tau d\omega)}
\end{equation}

We now turn to the scattering-b-transition calculus. Let $X$ be a manifold with boundary $\partial X$ and interior $X^\circ$. Denote by $\bdf$ a boundary defining function on $X$. Let
$$\Sigma = \{\sigma \in \C,\, |\sigma| \leq c,\, \phi_- \leq \arg(\sigma) \leq \phi_+\}$$
for some $c>0$ and $0 < \phi_+ - \phi_- < 2\pi$.
The space $\Psi_\scb^m(X)$ of $m$-th order scattering-b-transition pseudodifferential operators can be viewed as consisting of parameter-dependent $m$-th order pseudodifferential operators on $X^\circ$ with certain prescribed behavior near the boundary of $X$ and as $\sigma \to 0$ within $\Sigma$. Our exposition differs slightly from \cite[Appendix A.3]{hintz_mode_stability} in that we allow our parameter to take values in the compact sector $\Sigma$ as opposed to just an interval, i.e. $\sigma \in [0,1]$.

Let
$$\Sigma^{\text{res}} = [0,c]_{|\sigma|}\times[\phi_-,\phi_+]_{\arg(\sigma)},$$
which we think of as $[\Sigma;\{0\}]$, the blow-up of $\Sigma$ at the point $\sigma=0$. We continue to denote by $|\sigma|$ and $\arg(\sigma)$ the respective extensions to the blow-up $\Sigma^{\text{res}}$. Note that the projection
$$(|\sigma|,\arg(\sigma)) \in [0,c]\times[\phi_-,\phi_+] \to |\sigma|e^{i\arg(\sigma)} \in \Sigma$$
maps $(0,c]\times[\phi_-,\phi_+]$ diffeomorphically onto $\Sigma\setminus\{0\}$. Smooth functions on $\Sigma$ can be pulled-back along this projection to smooth functions on $\Sigma^{\text{res}}$, which are constant on $\{0\}\times[\phi_-,\phi_+]$.

Denote by
$$X_{\text{sc-b}} = [X\times\Sigma^{\text{res}};\partial X\times(\{0\}\times[\phi_-,\phi_+])]$$
the blow-up of $X\times\Sigma^{\text{res}}$ along the submanifold $\partial X\times(\{0\}\times[\phi_-,\phi_+])$. Then $X_{\text{sc-b}}$ is a manifold with corners, see \cite[Chapter 1]{melrose_mwc} for this notion and \cite[Chapter 5]{melrose_mwc} for the blow-up construction. There is a projection $\pi_\text{sc-b}:X_{\text{sc-b}} \to X\times\Sigma^{\text{res}}$, the blow-down map, which is a diffeomorphism away from the blown-up submanifold. $X_\scb$ has three boundary hypersurfaces that will be relevant to our analysis:
\begin{itemize}
    \item The transition face: $\tf = \pi_\text{sc-b}^{-1}(\{\partial X\times\{0\}\times[\phi_-,\phi_+]\})$,
    \item The scattering face ($\scf$): the closure of $\pi_\scb^{-1}\bigl(\partial X \times (0,c]\times[\phi_-,\phi_+] \bigr)$ in $X_\scb$,
    \item The zero face ($\zf$): the closure of $\pi_\scb^{-1}\bigl((X\setminus \partial X) \times \{0\}\times[\phi_-,\phi_+]\bigr)$ in $X_\scb$.
\end{itemize}
Near the interior of tf, $\tau = \frac{|\sigma|}{\bdf}$ defines a smooth coordinate. Note that $\tau \to 0$ at zf and $\tau \to \infty$ at scf. In fact, 
$$\tf \cong [0,\infty]_\tau\times\partial X \times [\phi_-,\phi_+]_{\arg(\sigma)},$$
where $[0,\infty]\subset\overline{\R}$ is understood as a subset of the radial compactification of $\R$. Furthermore,
$$\scf \cong \partial X \times [0,c]_{|\sigma|}\times[\phi_-,\phi_+]_{\arg(\sigma)}, \quad \zf \cong X\times[\phi_-,\phi_+]_{\arg(\sigma)}.$$
Note that $\tf$ can be invariantly defined as the inward pointing spherical normal bundle to $\partial X\times\{0\}\times[\phi_-,\phi_+]$ within $X\times\Sigma^{\text{res}}$. 
Away from $\scf$, $(\bdf,\frac{|\sigma|}{\bdf},\arg(\sigma))$ give smooth coordinates on $X_\scb$, whereas, away from $\zf$, $(\frac{\bdf}{|\sigma|},|\sigma|,\arg(\sigma))$ are smooth coordinates.
Smooth defining functions on $X_\scb$ for the three boundary faces above are given by (the extension of)
$$\bdf_\scf = \frac{\bdf}{\bdf + |\sigma|}, \quad \bdf_\tf = \bdf + |\sigma|, \quad \bdf_\zf = \frac{|\sigma|}{\bdf + |\sigma|}.$$

With $\mathcal{V}_\mathrm{b}(X_\scb)$ denoting the space of smooth vector fields on $X_\scb$ that are tangent to $\scf$, $\tf$ and $\zf$ we define the Lie algebra of scattering-b-transition vector fields by
$$\mathcal{V}_\scb(X) = \{V \in \bdf_\scf\mathcal{V}_\mathrm{b}(X_\scb),\, V|\sigma| = V\arg(\sigma) = 0\}.$$
In terms of local coordinates, $\mathcal{V}_\scb(X)$ is spanned over $C^\infty(X_\scb)$ by
\begin{equation}
\label{sc-b_vector_fields}
    \frac{\bdf}{\bdf+|\sigma|}\bdf\partial_\bdf \,\,\,\text{and}\,\,\, \frac{\bdf}{\bdf+|\sigma|}\partial_{\omega^j} \,\,(j = 1,\dots,n-1),
\end{equation}
where $\omega$ denotes coordinates on $\partial X$. Note that $[\mathcal{V}_\scb(X),\mathcal{V}_\scb(X)] \subset \bdf_\scf\mathcal{V}_\scb(X)$. $\mathcal{V}_\scb(X)$ can be viewed as the space of smooth sections of a rank $n$ vector bundle $T_\scb X \to X_\scb$, the scattering-b-transition tangent bundle, with \eqref{sc-b_vector_fields} as a local frame. The corresponding dual bundle is denoted $T_\scb^*X$. Note that for $c_0>0$ the restriction of $V \in \mathcal{V}_\scb(X)$ to $\{|\sigma|=c_0\} \cong X\times [\phi_-,\phi_+]_{\arg(\sigma)}$ defines a smooth family (in $\arg(\sigma)$) of scattering vector fields, whereas the restriction of $V$ to $\zf$ defines a smooth family of b-vector fields (hence the name scattering-b-transition algebra). Restricting $V$ to $\tf$ gives an element of $\mathcal{V}_\text{sc,b}(\tf)$ with scattering behavior near $\tf\cap\scf$ and b-behavior near $\tf\cap\zf$.

$\mathcal{V}_\scb(X)$ generates the graded algebra of scattering-b-transition differential operators $\mathrm{Diff}_\scb(X)$, that is, $\mathrm{Diff}_\scb^k(X)$ consists of finite sums of up to $k$-fold products of scattering-b-transition vector fields. There is a multiplicative principal symbol map $\sigma^\scb_k$ taking values in degree $k$ homogeneous polynomials on the fibers of $T^*_\scb X$ and giving a short exact sequence
$$0 \to \mathrm{Diff}_\scb^{k-1}(X) \hookrightarrow \mathrm{Diff}_\scb^k(X) \xrightarrow{\sigma^\scb_k} P^k(T_\scb^*X) / P^{k-1}(T_\scb^*X) \to 0.$$
In local coordinates, $\sigma^\scb_k(P)$ is given by mapping
\begin{equation}
\label{sc-b_symbol_map}
    \frac{\bdf}{\bdf+|\sigma|}\bdf\partial_\bdf \to \xi_\scb, \quad\text{and}\quad \frac{\bdf}{\bdf+|\sigma|}\partial_{\omega^j} \to \eta_\scb^j
\end{equation}
on the degree $k$ part of $P$, where $\xi_\scb$ and $\eta_\scb^j$ denote the corresponding fiber coordinates on $T^*_\scb X$. In addition, there is a principal symbol $\sigma^\scb_\scf$ at scf, which describes the asymptotics of $P$ at the scattering face. It fits into a short exact sequence:
$$0 \to \bdf_\scf\mathrm{Diff}_\scb^k(X) \hookrightarrow \mathrm{Diff}_\scb^k(X) \xrightarrow{\sigma^\scb_\scf} P^k(T_\scb^*X) / \bdf_\scf P^k(T_\scb^*X) \to 0.$$
$\sigma^\scb_\scf(P)$ is given in local coordinates by applying \eqref{sc-b_symbol_map} to the full operator $P$ and then restricting to scf. This produces a polynomial (not necessarily homogeneous) on the fibers of $T_\scb^*X$ over scf. The asymptotic behavior at zf and tf is described by the normal operators $N_\zf(P)$ and $N_\tf(P)$. These are obtained by restricting $P \in \mathrm{Diff}_\scb^k(X)$ to the respective boundary face and fit into short exact sequences
\begin{align}
    0 \to \bdf_\zf\mathrm{Diff}_\scb^k(X) \hookrightarrow \mathrm{Diff}_\scb^k(X) &\xrightarrow{N_\zf} \mathrm{Diff}_\mathrm{b}^k(X) \to 0 \label{zf_exact_seq}\\
    0 \to \bdf_\tf\mathrm{Diff}_\scb^k(X) \hookrightarrow \mathrm{Diff}_\scb^k(X) &\xrightarrow{N_\tf} \mathrm{Diff}_\mathrm{sc,b}^k(\tf) \to 0. \label{tf_exact_seq}
\end{align}
Note that $N_\zf(P)$ is a b-differential operator on $X$ and $N_\tf(P)$ is a scattering,b-differential operator on tf.
We will denote by 
$$\mathrm{Diff}_\scb^{k,q,l,w}(X) = \bdf_\scf^{-q}\bdf_\tf^{-l}\bdf_\zf^{-w}\mathrm{Diff}_\scb^k(X)$$
spaces of weighted scattering-b-transition differential operators. 

The algebra of scattering-b-transition pseudodifferential operators $\Psi_\scb(X)$ is the microlocalization of $\mathrm{Diff}_\scb(M)$. It consists of smooth families (in $\Sigma^{\text{res}}$) of linear operators on $\Dot{C}^\infty(X)$ whose Schwartz kernels are conormal distributions on the scattering-b-transition double space, see \cite[Section A.3]{hintz_mode_stability}. A typical element $A \in \Psi_\scb^m(M)$ can be obtained in local coordinates by the quantization procedure
\begin{equation*}
\begin{split}
    Au(\bdf,\omega) = \frac{1}{(2\pi)^n}\int &\exp\Bigl(i\Bigl(\frac{\bdf-\bdf'}{\bdf\frac{\bdf}{\bdf+|\sigma|}}\xi_\scb + \frac{\omega-\omega'}{\frac{\bdf}{\bdf+|\sigma|}}\cdot\eta_\scb\Bigr)\Bigr) \chi\Bigl(\Bigl|\log\Bigl(\frac{\bdf}{\bdf'}\Bigr)\Bigr|\Bigr)\chi(|\omega-\omega'|) \\
    &\cdot a(\bdf,\omega,|\sigma|,\arg(\sigma),\xi_\scb,\eta_\scb)u(\bdf',\omega')\,d\xi_\scb d\eta_\scb\,\frac{d\bdf'}{\bdf'\frac{\bdf'}{\bdf'+|\sigma|}}\frac{d\omega'}{(\frac{\bdf'}{\bdf'+|\sigma|})^{n-1}},
\end{split}
\end{equation*}
where $a \in C^\infty(T^*_\scb X)$ is a symbol of order $m$ in $(\xi_\scb,\eta_\scb)$. The cutoff $\chi \in C^\infty_c(\R)$ with $\chi=1$ near $0$ localizes to a neighborhood of the diagonal. Note that when composing pseudodifferential operators one may need to enlarge the support of the cutoff. In this way, one obtains a well-defined operator algebra, i.e.
$$\Psi_\scb^m(X) \circ \Psi_\scb^{m'}(X) \subset \Psi_\scb^{m+m'}(X),$$
and a multiplicative principal symbol map $\sigma^\scb$ fitting into the short exact sequence
$$0 \to \bdf_\scf\Psi_\scb^{m-1}(X) \hookrightarrow \Psi_\scb^m(X) \xrightarrow{\sigma^\scb} S^m(T_\scb^*X) / \bdf_\scf S^{m-1}(T_\scb^*X) \to 0.$$
On the subspace of classical pseudodifferential operators, which in particular includes the differential operators, the principal symbol map splits into two parts as above: the principal symbol at fiber infinity and the principal symbol at scf. The existence of a well-defined principal symbol enables the use of symbolic estimates. In particular, elliptic estimates hold in the scattering-b-transition calculus.
We can define spaces of weighted scattering-b-transition operators by
$$\Psi_\scb^{m,q,l,w}(X) = \bdf_\scf^{-q}\bdf_\tf^{-l}\bdf_\zf^{-w}\Psi_\scb^m(X).$$

Operators in $\Psi_\scb(M)$ naturally act on corresponding weighted scattering-b-transition Sobolev spaces $H_{\scb,\sigma}^{s,q,l,w}(X)$. These are parameter-dependent Hilbert spaces on $X$. For each fixed $|\sigma|>0$, an element in $H_{\scb,\sigma}^{s,q,l,w}(X)$ is a smooth family (in $\arg(\sigma)$) of elements in the scattering Sobolev space $H_\mathrm{sc}^{s,q}(X)$. However, we equip $H_{\scb,\sigma}^{s,q,l,w}(X)$ with the $|\sigma|$-dependent norm (for $s>0$):
$$\|u\|^2_{H_{\scb,\sigma}^{s,q,l,w}(X)} = \|\bdf_\scf^{-q}\bdf_\tf^{-l}\bdf_\zf^{-w}u\|^2_{L^2(X,\mu)} + \|Au\|^2_{L^2(X,\mu)},$$
where $A$ is a fixed elliptic element of $\Psi_\scb^{s,q,l,w}(X)$ and $\mu = \bdf^{-n}\mu_0$ with $\mu_0$ a smooth b-density on $X$. Note that for $k\in\N$ this is equivalent to
\begin{equation}
\label{integer_norms}
    \|u\|^2_{H_{\scb,\sigma}^{k,q,l,w}(X)} = \sum_{j+|\alpha|\leq k} \Bigl\|\bdf_\scf^{-q}\bdf_\tf^{-l}\bdf_\zf^{-w}\Bigl(\frac{\bdf}{\bdf+|\sigma|}\bdf\partial_\bdf\Bigr)^j\Bigl(\frac{\bdf}{\bdf+|\sigma|}\partial_{\omega}\Bigr)^\alpha u\Bigr\|^2_{L^2(X)}.
\end{equation}
For $s<0$ we define the scattering-b-transition Sobolev space by duality with respect to the $L^2(X,\mu)$ inner product, that is, $H_{\scb,\sigma}^{s,q,l,w}(X) = \bigl(H_{\scb,\sigma}^{-s,q,l,w}(X)\bigr)^*$.

In a neighborhood of the respective boundary faces, these norms can be related to b-norms on zf and scattering,b-norms on tf. In fact, for $\chi \in C^\infty(X_\scb)$ identically $1$ in a neighborhood of zf and with support away from scf, there is a uniform equivalence of norms:
\begin{equation}
\label{zf_norm_equivalence}
    \|\chi u\|_{H_{\scb,\sigma}^{s,q,l,w}(X)} \sim |\sigma|^{-w}\|\chi u\|_{H_{b}^{s,l-w}(X)},
\end{equation}
i.e. there exists a constant $C>0$ such that
$$\frac{1}{C}\|\chi u\|_{H_{\scb,\sigma}^{s,q,l,w}(X)} \leq |\sigma|^{-w}\|\chi u\|_{H_{b}^{s,l-w}(X)} \leq C\|\chi u\|_{H_{\scb,\sigma}^{s,q,l,w}(X)}$$
uniformly in $\sigma$. For positive integer regularity this can be seen from \eqref{integer_norms}. Indeed, $\frac{\bdf}{|\sigma|}$ is bounded away from scf, so on $\supp(\chi)$ we have
$$\bdf_\scf \sim C, \quad \bdf_\tf \sim \bdf, \quad \bdf_\zf \sim |\sigma|\bdf^{-1} \quad\text{and}\quad \frac{\bdf}{\bdf+|\sigma|}\bdf\partial_\bdf \sim \bdf\partial_\bdf, \quad \frac{\bdf}{\bdf+|\sigma|}\partial_\omega \sim \partial_\omega.$$
For the general case see \cite[Appendix A.3]{hintz_mode_stability}. Let now $\psi \in C^\infty(X_\scb)$ have support in a small neighborhood of tf with $\psi=1$ near tf. Then we have the uniform equivalence of norms
\begin{equation}
\label{tf_norm_equivalence}
    \|\psi u\|_{H_{\scb,\sigma}^{s,q,l,w}(X)} \sim |\sigma|^{-l-\frac{n}{2}}\|\psi u\|_{H_{\mathrm{sc,b}}^{s,q,w-l}(\tf)},
\end{equation}
where we use the coordinate $\tau = \frac{|\sigma|}{\bdf}$ on tf and the density $\tau^{n-1}d\tau d\omega$. Indeed, near tf we have
$$\bdf_\scf \sim \frac{1}{1+\tau}, \quad \bdf_\tf \sim |\sigma|\frac{1+\tau}{\tau}, \quad \bdf_\zf \sim \frac{\tau}{1+\tau}$$
and
$$\frac{\bdf}{\bdf+|\sigma|}\bdf\partial_\bdf \sim \frac{\tau}{1+\tau}\partial_\tau, \quad \frac{\bdf}{\bdf+|\sigma|}\partial_\omega \sim \frac{1}{1+\tau}\partial_\omega.$$
Thus, for positive integer regularity \eqref{tf_norm_equivalence} follows by comparing \eqref{integer_norms} to \eqref{sc,b_integer_norm}. Note that the factor of $|\sigma|^{-\frac{n}{2}}$ comes from relating the densities.
See \cite[Appendix A.3]{hintz_mode_stability} for the general case.

\subsection{Kerr spectral family as a scattering-b transition operator}

In the remainder of this section we will view $P_\beta(\sigma)$ as an element of the scattering-b-transition algebra and study its properties within this calculus. We apply the constructions of Section \ref{section_scattering-b-transition} to our spaces $X_\beta$ by viewing $X_\beta$ as a subset of $\R^3$ using the coordinate map $F_\beta^{-1}$ and and adding the sphere at infinity. In this way, we obtain a manifold with boundary $\overline{X_\beta}$, which is an open subset of the radial compactification of $\overline{\R^3}$. Note that $\bdf=\frac{1}{r}$ is a boundary defining function and
$$\overline{X_\beta} = \bigl[0,\tfrac{1}{r_0}\bigr)_\bdf\times\Sph^2.$$
We will continue to use the density $\rkerr^2\sin(\theta)drd\theta d\varphi_*$ on $X_\beta$. In terms of the boundary defining function, this has the form $\bdf^{-3}$ times a smooth b-density and thus fits into the framework of Section \ref{section_scattering-b-transition}.
As our parameter space we take the compact sector
\begin{equation}
\label{parameter_space}
    \Sigma_{\beta}^{c,\delta} = \bigl\{\sigma \in \C,\, |\sigma| \leq c,\, -\beta + \delta \leq \arg(\sigma) \leq \pi - \beta - \delta\bigr\} \subset \Lambda_\beta.
\end{equation}
for some $c>0$ and $\delta>0$ small. Thus, $\arg(\sigma)$ will be bounded away from the values $-\beta$ and $\pi-\beta$. This will allow us to obtain uniform estimates for the normal operators of $P_\beta(\sigma)$.

Note that the topological boundary of $X_\beta$ at $r=r_0$ is not included in $\overline{X_\beta}$. We will work on scattering-b-transition Sobolev spaces that extend across $r=r_0$. Concretely, we set
$$\Bar{H}_{\scb,\sigma}^{s,q,l,w}(X_\beta) = \{u\in H_{\scb,\sigma}^{s,q,l,w}(X_\beta),\, u = v|_{X_\beta} \,\text{ for some }\, v\in H_{\scb,\sigma}^{s,q,l,w}(\overline{\R^3})\}$$
with norm
$$\|u\|_{\Bar{H}_{\scb,\sigma}^{s,q,l,w}(X_\beta)} = \inf\{\|v\|_{H_{\scb,\sigma}^{s,q,l,w}(\overline{\R^3})},\, v\in H_{\scb,\sigma}^{s,q,l,w}(\overline{\R^3}),\, u = v|_{X_\beta}\}.$$
In the following lemma, we show that $P_\beta(\sigma)$ indeed defines a scattering-b-transition operator, elliptic near $\bdf=0$, and compute its normal operators at zf and tf.

\begin{lemma}
\label{lemma_sc-b_interpretation}
    For every $\beta \in (-\pi,\pi)$, the complex scaled Kerr spectral family satisfies
    $$P_\beta(\sigma) \in \mathrm{Diff}_{\emph{\scb},\sigma}^{2,0,-2,0}(X_\beta).$$
    For $R$ large enough, $P_\beta(\sigma)$ is uniformly elliptic in $\{r>R\}$, i.e. in $\{\bdf_\scf\bdf_\tf < \frac{1}{R}\}$, as an element of the scattering-b-transition algebra. 
    Its principal symbol at \emph{scf} is
    \begin{equation}
    \label{scf_symbol}
        \sigma^\emph{\scb}_\scf\bigl(\bdf_\tf^{-2}P_\beta(\sigma)\bigr) = e^{-2i\beta}|\xi_\emph{\scb}|^2 - e^{2i\arg(\sigma)}
    \end{equation}
    Its normal operators at \emph{tf} and \emph{zf} are given by
    \begin{equation}
    \label{normal_operators}
    \begin{split}
        N_{\tf}\bigl(\bdf_\tf^{-2}P_\beta(\sigma)\bigr) &= e^{-2i\beta}\frac{1}{(1+\tau)^2}\Bigl((\tau D_\tau)^2 - i\tau D_\tau + \Delta_{\Sph^2}\Bigl) - \frac{\tau^2}{(1+\tau)^2}e^{2i\arg(\sigma)}, \\
        N_{\zf}\bigl(P_\beta(\sigma)\bigr) &= P_\beta(0),
    \end{split}
    \end{equation}
    where we use the coordinate $\tau = |\sigma|r$ on \emph{tf}.
\end{lemma}
\begin{proof}
    Writing the operator $P_\beta(\sigma)$ in the complex scaling region in terms of the scattering-b-transition vector fields
    \begin{align*}
        D^\scb_r &= \frac{r}{1+|\sigma|r}D_r = -\frac{\bdf}{\bdf+|\sigma|}\bdf D_\bdf, \\
        D^\scb_\theta &= \frac{1}{1+|\sigma|r}D_\theta = \frac{\bdf}{\bdf+|\sigma|}D_\theta, \\
        D^\scb_{\varphi_*} &= \frac{1}{1+|\sigma|r}D_{\varphi_*} = \frac{\bdf}{\bdf+|\sigma|}D_{\varphi_*},
    \end{align*}
    and the boundary defining functions
    \begin{align*}
        \bdf_\scf = \frac{\bdf}{\bdf+|\sigma|} = \frac{1}{1+|\sigma|r}, \quad \bdf_\tf = \bdf + |\sigma| = \frac{1+|\sigma|r}{r} \quad \bdf_\zf = \frac{|\sigma|}{\bdf+|\sigma|} = \frac{|\sigma|r}{1+|\sigma|r},
    \end{align*}
    we find
    \begin{equation}
    \label{sc-b_explicit}
    \begin{split}
        P_\beta(\sigma) = \bdf_\tf^2 &\Biggl( \frac{r^2}{\rkerr_\beta^2}\Bigl(\frac{1}{f_\beta'}D^{\scb}_r\frac{\mu_\beta}{f_\beta'r^2}D^{\scb}_r + \Delta_{\Sph^2}^\scb + \bdf_\scf\bdf_\tf\frac{2a}{f_\beta'} D_{\varphi_*}^\scb D^\scb_r \\
        &-2i\bdf_\scf\frac{\mu_\beta}{(f_\beta')^2r^2}D^\scb_r + i\bdf_\scf^2\frac{\mu_\beta}{(f_\beta')^2r^2}D^\scb_r + \bdf_\scf^2\bdf_\tf^2\bdf_\zf e^{i\arg(\sigma)} \frac{4maf_\beta r}{\mu_\beta}D_{\varphi_*}^\scb\Bigl) \\
        &- \bdf_\zf^2e^{2i\arg(\sigma)} - \bdf_\scf\bdf_\tf\bdf_\zf^2e^{2i\arg(\sigma)}\frac{2mf_\beta(f_\beta^2+a^2)r}{\rkerr_\beta^2\mu_\beta} \Biggr),
    \end{split}
    \end{equation}
    where
    $$\Delta_{\Sph^2}^\scb = \frac{1}{\sin(\theta)}D^\scb_\theta\sin(\theta)D^\scb_\theta + \frac{1}{\sin^2(\theta)}(D_{\varphi_*}^\scb)^2 = \bdf_\scf^2\Delta_{\Sph^2}.$$
    In \eqref{sc-b_explicit}, we have factored out the order of vanishing at the various boundary faces. The remaining functions extend to elements of $C^\infty(X^\scb)$ bounded away from zero at the boundary faces. From \eqref{sc-b_explicit} it is evident that
    $$P_\beta(\sigma) \in \bdf_\tf^2\mathrm{Diff}_\scb^{2,0,0,0}(X_\beta) = \mathrm{Diff}_\scb^{2,0,-2,0}(X_\beta).$$

    We can read off the principal symbol at fiber infinity in the scattering-b-transition algebra from the first line of \eqref{sc-b_explicit}. It satisfies 
    $$\sigma^\scb_2\bigl(\bdf_\tf^{-2} P_\beta(\sigma)\bigr) = e^{-2i\beta}|\xi_\scb|^2 + \mathcal{O}(\bdf_\scf\bdf_\tf)|\xi_\scb|^2,$$
    where $\xi_\scb$ denotes fiber coordinates on $T_\scb^*X$. Note that $\frac{r^2}{\rkerr_\beta^2} = e^{-2i\beta} + \mathcal{O}(\bdf_\scf\bdf_\tf)$ and $\frac{\mu_\beta}{(f_\beta')^2r^2} = 1 + \mathcal{O}(\bdf_\scf\bdf_\tf)$.
    Thus, choosing $\bdf_\scf\bdf_\tf = r^{-1}$ small enough, we have
    $$\bigl|\sigma^\scb_2\bigl(\bdf_\tf^{-2} P_\beta(\sigma)\bigr)\bigr| \geq C|\xi_\scb|^2$$
    uniformly in $X_\scb\cap\{\bdf_\scf\bdf_\tf < \frac{1}{R}\}$.
    Moreover, to leading order at scf, we have
    $$\bdf_\tf^{-2} P_\beta(\sigma) = (D^\scb_r)^2 + \Delta^\scb_{\Sph^2} - \bdf_\zf^2 e^{2i\arg(\sigma)} + \mathcal{O}(\bdf_\scf).$$
    Since $\bdf_\zf = 1$ at scf, \eqref{scf_symbol} follows.

    In terms of the coordinate $\tau = |\sigma|r$ in the interior of tf, we have
    $$D^\scb_r = \frac{\tau}{1+\tau}D_\tau, \quad D^\scb_\theta = \frac{1}{1+\tau}D_\theta, \quad D^\scb_{\varphi_*} = \frac{1}{1+\tau}D_{\varphi_*}$$
    and
    $$\bdf_\scf = \frac{1}{1+\tau}, \quad \bdf_\zf = \frac{\tau}{1+\tau}.$$
    Thus, restricting $\bdf_\tf^{-2} P_\beta(\sigma)$ to tf, we obtain
    \begin{equation*}
    \begin{split}
        N_\tf(\bdf_\tf^{-2} P_\beta(\sigma)) = e^{-2i\beta}&\Bigl(\bigl(\frac{\tau}{1+\tau} D_\tau\bigr)^2 - 2i\frac{1}{1+\tau} \frac{\tau}{1+\tau}D_\tau + i\frac{1}{(1+\tau)^2}\frac{\tau}{1+\tau}D_\tau \\
        &+ \frac{1}{(1+\tau)^2}\Delta_{\Sph^2}\Bigl) - \frac{\tau^2}{(1+\tau)^2}e^{2i\arg(\sigma)},
    \end{split}
    \end{equation*}
    which gives the first equation in \eqref{normal_operators} after rearranging. Finally, restricting \eqref{sc-b_explicit} to zf gives $N_{\zf}\bigl(P_\beta(\sigma)\bigr) = P_\beta(0)$ in the complex scaling region. Note that away from complex scaling, i.e. in $\{r<R_0\}$, the operator $P_\beta(\sigma)$ takes a different form, see \eqref{operator_full}. Within this region, $|\sigma|$ is a boundary defining function for zf (we are away from scf and tf), so the second equation in \eqref{normal_operators} follows by restricting to $|\sigma|=0$.
\end{proof}

Using the ellipticity, in the scattering-b-transition calculus, of $P_\beta(\sigma)$ for $r$ large enough together with Lemma \ref{lemma_uniform_interior_estimate}, we can prove estimates of the form \eqref{uniform_interior_estimate} on the scattering-b-transition Sobolev spaces. Note that ellipticity at scf allows the error term to be measured in norms with arbitrarily high decay order at scf.

\begin{lemma}
\label{lemma_sc-b_elliptic_estimate}
    Let $\beta \in (-\pi,\pi)$ and $c,\delta > 0$. Consider $P_\beta(\sigma) \in \mathrm{Diff}_{\emph{\scb},\sigma}^{2,0,-2,0}(X_\beta)$ for $\sigma \in \Sigma_{\beta}^{c,\delta}$ as in \eqref{parameter_space}.
    Let $s,q,l,w \in \R$ with $s > \frac{1}{2}-\alpha\Im(\sigma)$ for all $\sigma\in \Sigma_{\beta}^{c,\delta}$. Then for any $N\in\R$ there exists $C>0$ so that the following estimate holds uniformly in $\sigma$ for all $u \in \Bar{H}_{\emph{\scb},\sigma}^{s,q,l,w}(X_\beta)$ with $P_\beta(\sigma)u \in \Bar{H}_{\emph{\scb},\sigma}^{s-1,q,l+2,w}(X_\beta)$:
    \begin{equation}
    \label{sc-b_interior_estimate}
        \|u\|_{\Bar{H}_{\emph{\scb},\sigma}^{s,q,l,w}(X_\beta)} \leq C\bigl(\|P_\beta(\sigma)u\|_{\Bar{H}_{\emph{\scb},\sigma}^{s-1,q,l+2,w}(X_\beta)} + \|u\|_{\Bar{H}_{\emph{\scb},\sigma}^{-N,-N,l,w}(X_\beta)}\bigr).
    \end{equation}
\end{lemma}
\begin{proof}
    Let $\chi, \Tilde{\chi} \in C^\infty(X_\beta)$ be cutoff functions supported in $\{r<2R\}$ for some large $R$ and with $\chi=1$ on $\{r\leq R\}$ and $\Tilde{\chi} = 1$ on $\supp(\chi)$. Then the uniform estimate of Lemma \ref{lemma_uniform_interior_estimate} applied to $\chi u$, together with elliptic estimates (in the standard algebra $\Psi^2(X_\beta)$) for $P_\beta(\sigma)$ on $\supp(\nabla \chi)$, gives
    $$\|\chi u\|_{\Bar{H}^s(X_\beta)} \leq C\bigl(\|\Tilde{\chi} P_\beta(\sigma)u\|_{\Bar{H}^{s-1}(X_\beta)} + \|\Tilde{\chi} u\|_{\Bar{H}^{-N}(X_\beta)}\bigr)$$
    uniformly for $\sigma \in \Sigma_{\beta}^{c,\delta}$.
    Away from the boundary at $r=\infty$, we have a uniform equivalence of norms
    $$\|\chi u\|_{\Bar{H}_{\scb,\sigma}^{s,q,l,w}} \sim |\sigma|^w\|\chi u\|_{\Bar{H}^s(X_\beta)},$$
    that is, there exists some $C>0$ such that
    $$\frac{1}{C}\|\chi u\|_{\Bar{H}_{\scb,\sigma}^{s,q,l,w}} \leq |\sigma|^w\|\chi u\|_{\Bar{H}^s(X_\beta)} \leq C\|\chi u\|_{\Bar{H}_{\scb,\sigma}^{s,q,l,w}}, \quad\forall \sigma\in\Sigma_{\beta}^{c,\delta}.$$
    Thus, we obtain
    $$\|\chi u\|_{\Bar{H}_{\scb,\sigma}^{s,q,l,w}(X_\beta)} \leq C\bigl(\|\Tilde{\chi} P_\beta(\sigma)u\|_{\Bar{H}_{\scb,\sigma}^{s-1,q,l+2,w}(X_\beta)} + \|\Tilde{\chi} u\|_{\Bar{H}_{\scb,\sigma}^{-N,q,l,w}(X_\beta)}\bigr).$$
    Note that this could also be obtained by developing the symbolic estimates, i.e. propagation and radial point estimates, for the scattering-b-transition calculus, which we chose to avoid.

    By Lemma \ref{lemma_sc-b_interpretation}, $P_\beta(\sigma) \in \mathrm{Diff}_\scb^{2,0,-2,0}$ is uniformly elliptic on $\supp(1-\chi)$. Combining the estimate on $\chi u$ with elliptic estimates (at fiber infinity) in the scattering-b-transition calculus gives
    $$\|u\|_{\Bar{H}_{\scb,\sigma}^{s,q,l,w}(X_\beta)} \leq C\bigl(P_\beta(\sigma)u\|_{\Bar{H}_{\scb,\sigma}^{s-1,q,l+2,w}(X_\beta)} + \|u\|_{\Bar{H}_{\scb,\sigma}^{-N,q,l,w}(X_\beta)}\bigr).$$
    By \eqref{scf_symbol} $P_\beta(\sigma)$ is also elliptic at scf, uniformly for $\sigma \in \Sigma_{\beta}^{c,\delta}$. Thus, elliptic estimates at scf allow us to improve the error term above and obtain \eqref{sc-b_interior_estimate}. Note that one could also use an argument as in the proof of Lemma \ref{lemma_scattering_elliptic}, based on the uniform invertibility of the leading order part of $P_\beta(\sigma)$ at scf, if one wanted to avoid symbolic estimates at scf.
\end{proof}

\subsection{Estimates on the normal operators}

In this subsection we focus on the normal operators computed in Lemma \ref{lemma_sc-b_interpretation}. We show that these operators have trivial kernels on appropriate Sobolev spaces and prove normal operator estimates that will be crucial for the desired uniform low energy estimate in Proposition \ref{prop_low_energy_estimate} below.

We begin at the transition face. Recall that
$$\tf = [0,\infty]_\tau\times\Sph^2, \quad\text{with}\quad \tau = |\sigma|r.$$
Here, $[0,\infty] \subset \overline{\R}$ is viewed as a subset of the radial compactification of $\R$. We use the measure $\tau^2d\tau d\omega$ on tf, with $d\omega$ denoting the standard measure on $\Sph^2$. 
The normal operator map $N_\tf$ naturally gives rise to an operator with b-behavior at $\tau=0$ and scattering behavior at $\tau=\infty$, see \eqref{tf_exact_seq}. We thus work on the Sobolev spaces $H_{\mathrm{sc,b}}^{s,q,l}(\tf)$, see Section \ref{section_scattering-b-transition}. Note that a similar model operator was studied in \cite[Lemma 3.20]{hintz_mode_stability}. In contrast to this reference, our operator is elliptic at $\tf\cap\scf$, so we do not need to use a variable decay order there.

\begin{lemma}
\label{lemma_tf_operator}
    Let $\beta \in (-\pi,\pi)$ and $\delta > 0$. Let further $s,q \in \R$ and $l \in (\frac{1}{2},\frac{3}{2})$. Then the following estimate holds uniformly for $\arg(\sigma) \in [-\beta+\delta,\pi-\beta-\delta]$:
    $$\|u\|_{H_{\mathrm{sc,b}}^{s,q,l}(\tf)} \leq C\|N_{\tf}\bigl(\bdf_\tf^{-2}P_\beta(\sigma)\bigr)u\|_{H_{\mathrm{sc,b}}^{s-2,q,l}(\tf)}.$$
\end{lemma}
\begin{proof}
    Note that $N_{\tf}\bigl(\bdf_\tf^{-2}P_\beta(\sigma)\bigr)$ is an elliptic element of $\mathrm{Diff}_{\mathrm{sc,b}}^{2,0,0}(\tf)$. The form in \eqref{normal_operators} makes the b-structure at $\tau=0$ evident. We can rewrite
    $$N_{\tf}\bigl(\bdf_\tf^{-2}P_\beta(\sigma)\bigr) = \frac{\tau^2}{(1+\tau)^2}\Bigl(e^{2i\beta}\bigl(\tau^{-2}D_\tau\tau^2D_\tau + \tau^{-2}\Delta_{\Sph^2}\bigr) - e^{2i\arg(\sigma)}\Bigr)$$
    to emphasize the scattering structure at $\tau=\infty$. Asymptotics at the scattering end are controlled by the principal symbol at $\tau=\infty$, which is just $e^{2i\beta}|\xi|^2 - e^{2i\arg(\sigma)}$. Thus, our operator is scattering elliptic at $\tau=\infty$ uniformly for $\arg(\sigma) \in (-\beta+\delta,\pi-\beta-\delta)$. Asymptotics at the b-end are controlled by the b-normal operator at $\tau=0$, which is just $(\tau D_\tau)^2 - i\tau D_\tau + \Delta_{\Sph^2}$. This is invertible on the range of weights in the Lemma. Thus, elliptic theory gives the semi-Fredholm estimate
    $$\|u\|_{H_{\mathrm{sc,b}}^{s,q,l}(\tf)} \leq C\bigl(\|N_{\tf}\bigl(\bdf_\tf^{-2}P_\beta(\sigma)\bigr)u\|_{H_{\mathrm{sc,b}}^{s-2,q,l}(\tf)} + \|u\|_{H_{\mathrm{sc,b}}^{-N,-N,-N}(\tf)}\bigr)$$
    uniformly for $\arg(\sigma)$ as in the Lemma.

    It remains to prove injectivity. Note that the above estimate holds for all $s,q\in\R$. Thus, an element $u \in H_{\mathrm{sc,b}}^{s,q,l}(\tf)$ with $N_{\tf}\bigl(\bdf_\tf^{-2}P_\beta(\sigma))u = 0$ lies in $H_{\mathrm{sc,b}}^{\infty,\infty,l}(\tf)$. Sobolev embedding then implies that $u \in C^\infty\bigl((0,\infty)_\tau\times\Sph^2\bigr)$ with $u$ Schwartz at the scattering end and conormal of weight $\tau^{l-\frac{3}{2}}$ at the b-end. That is, $u$ and any number of derivatives with respect to $\partial_\tau, \tau^{-1}\partial_\omega$ decay superpolynomially at $\tau=\infty$, whereas $u$ and any number of derivatives with respect to $\tau\partial_\tau, \partial_\omega$ decay as $\tau^{l-\frac{3}{2}}$ at $\tau=0$ (note that for $l\in(\frac{1}{2},\frac{3}{2})$ this indeed gives decay.). Thus, $u$ is a smooth function on $(0,\infty)_\tau\times\Sph^2$ lying in the kernel of $e^{2i\beta}\Delta - e^{2i\arg(\sigma)}$ and decaying as $\tau \to 0,\infty$. This is only possible for $u=0$. (This can be seen by decomposing $u$ into spherical harmonics and performing a partial integration.)
\end{proof}

\begin{rmk}
    We note that the estimate in Lemma \ref{lemma_tf_operator}, though not the invertibility of the operator, in fact only requires the condition $l>\frac{1}{2}$ on the weight. Similarly, the estimate in Lemma \ref{lemma_zero_energy_estimate} below holds for weights $l > -\frac{3}{2}$. Applying both of these normal operator estimates to an element $u \in \Bar{H}_{\scb,\sigma}^{s,q,l,w}(X_\beta)$, as we will do in Proposition \ref{prop_low_energy_estimate}, then imposes the range of weights $w-l \in (\frac{1}{2},\frac{3}{2})$ according to the equivalence of norms in \eqref{zf_norm_equivalence} and \eqref{tf_norm_equivalence}.
\end{rmk}

We now turn to the normal operator at zf, which is just the zero energy operator of our complex scaled spectral family. We can view this as a b-differential operator on $X_\beta$, see \eqref{zf_exact_seq}, and we prove estimates on weighted b-Sobolev spaces. Note that $\bdf = \frac{1}{r}$ is a defining function on zf for $\zf\cap\tf$. Thus, the $\bdf_\tf^{2}$ decay of $P_\beta(\sigma)$ at tf corresponds to the presence of a weight $\bdf^{2}$ for the zf-normal operator, i.e. $P_\beta(0) \in \mathrm{Diff}_\mathrm{b}^{2,-2}(X_\beta)$. The injectivity of $P_\beta(0)$ is more subtle than for the tf-normal operator. The case $\beta=0$ is covered in \cite[Lemma 3.19]{hintz_mode_stability}, where a corresponding estimate is proved for the zero energy operator on Kerr (without complex scaling). We reduce the $\beta \neq 0$ case to this result by showing in Lemma \ref{lemma_zero_operator_kernel} below that complex scaling does not affect the triviality of the kernel.

\begin{lemma}
\label{lemma_zero_energy_estimate}
    Let $\beta \in (-\pi,\pi)$, $s>\frac{1}{2}$ and $l \in (-\frac{3}{2},-\frac{1}{2})$. Then the zero energy operator $P_\beta(0)$ has trivial kernel on the weighted b-Sobolev space $\Bar{H}_\mathrm{b}^{s,l}(X_\beta)$ and the following estimate holds:
    \begin{equation}
    \label{zero_energy_estimate}
        \|u\|_{\Bar{H}_\mathrm{b}^{s,l}(X_\beta)} \leq C\|P_\beta(0)u\|_{\Bar{H}_\mathrm{b}^{s-1,l+2}(X_\beta)}.
    \end{equation}
\end{lemma}
\begin{proof}
    Away from infinity, the microlocal elliptic, propagation and radial point estimates, combined with hyperbolic estimates near $r=r_0$, give
    \begin{equation}
    \label{zero_op_interior_estimate}
        \|\chi u\|_{\Bar{H}_\mathrm{b}^{s,l}(X_\beta)} \leq C\bigl(\|\Tilde{\chi} P_\beta(\sigma)u\|_{\Bar{H}_\mathrm{b}^{s-1,l+2}(X_\beta)} + \|\Tilde{\chi} u\|_{\Bar{H}_\mathrm{b}^{-N,l}(X_\beta)}\bigr)
    \end{equation}
    for any smooth $\chi,\Tilde{\chi}$ with support in some ball $B_R$ and such that $\Tilde{\chi}=1$ on $\supp(\chi)$. Note that, away from infinity, the weighted b-Sobolev norms $\|\cdot\|_{\Bar{H}_\mathrm{b}^{s,l}(X_\beta)}$ are equivalent to the usual Sobolev norm $\|\cdot\|_{\Bar{H}^{s}(X_\beta)}$.
    
    The zero energy operator is given in local coordinates by
    $$P_\beta(0) = \frac{1}{\rkerr_\beta^2}\Bigl(\frac{1}{f_\beta'}D_r\frac{\mu_\beta}{f_\beta'}D_r + \Delta_{\Sph^2} + \frac{2a}{f_\beta'}D_rD_{\varphi_*}\Bigr),$$
    where $\Delta_{\Sph^2}$ is the (positive) Laplacian on the sphere. This is a weighted b-differential operator on $X_\beta$:
    $$P_\beta(0) \in \mathrm{Diff}_\mathrm{b}^{2,-2}(X_\beta).$$
    Its principal symbol in the b-algebra is given by
    $$\sigma^\mathrm{b}_2(r^2P_\beta(0)) = \frac{r^2}{\rkerr_\beta^2}\Bigl(\frac{\mu_\beta}{(f_\beta')^2r^2}\xi_\mathrm{b}^2 + \eta_\mathrm{b}^2 + \frac{\nu_\mathrm{b}^2}{\sin^2(\theta)} + \frac{2a}{f_\beta'r}\xi_\mathrm{b}\nu_\mathrm{b}\Bigr).$$
    Estimating the terms as in the proof of Lemma \ref{lemma_scattering_elliptic}, shows that $r^2P_\beta(0)$ is uniformly elliptic as a b-operator in $\{|x| > R\}$ for $R$ large enough. Elliptic b-estimates, see for instance \cite[Lemma 2.5]{hintz_gluing_black_holes}, combined with the interior estimate \eqref{zero_op_interior_estimate}, then give
    \begin{equation*}
        \|u\|_{\Bar{H}_\mathrm{b}^{s,l}(X_\beta)} \leq C\bigl(\|P_\beta(\sigma)u\|_{\Bar{H}_\mathrm{b}^{s-1,l+2}(X_\beta)} + \|u\|_{\Bar{H}_\mathrm{b}^{-N,l}(X_\beta)}\bigr).
    \end{equation*}

    As in the proof of Lemma \ref{lemma_scattering_elliptic}, we can use the invertibility of the leading order part of $P_\beta(0)$ at $r=\infty$ to improve the error term in the above estimate. In local coordinates on $X_\beta$, we have
    $$P_\beta(0) = e^{-2i\beta}\Delta + Q, \quad Q \in \mathrm{Diff}_\mathrm{b}^{2,-3}(X_\beta),$$
    where $\Delta$ is just the Laplacian on $\R^3$. Using the invertibility of the Laplacian on weigthed b-Sobolev spaces with weight $l \in (-\frac{3}{2},-\frac{1}{2})$, see for instance \cite[Theorem 3.1]{hintz_geometric_singular_analysis_notes}, we obtain the following semi-Fredholm estimate:
    \begin{equation}
    \label{zero_energy_semi_fredholm}
        \|u\|_{\Bar{H}_\mathrm{b}^{s,l}(X_\beta)} \leq C\bigl(\|P_\beta(\sigma)u\|_{\Bar{H}_\mathrm{b}^{s-1,l+2}(X_\beta)} + \|u\|_{\Bar{H}_\mathrm{b}^{-N,l-1}(X_\beta)}\bigr).
    \end{equation}
    Note the compactness of the inclusion $\Bar{H}_\mathrm{b}^{s,l}(X_\beta) \hookrightarrow \Bar{H}_\mathrm{b}^{-N,l-1}(X_\beta)$.

    The injectivity of $P_\beta(0)$ follows from the corresponding injectivity of the original Kerr zero energy operator $P_0(0)$ on $\Bar{H}_\mathrm{b}^{s,l}(X_\beta)$ with $l\in(-\frac{3}{2},-\frac{1}{2})$, see \cite[Lemma 3.19]{hintz_mode_stability} and the invariance of the dimension of the kernel under complex scaling, which we prove in Lemma \ref{lemma_zero_operator_kernel} below. Finally, the estimate \eqref{zero_energy_estimate} follows from injectivity of $P_\beta(0)$ together with the semi-Fredholm estimate \eqref{zero_energy_semi_fredholm} by standard functional analytic arguments.
\end{proof}

\begin{lemma}
\label{lemma_zero_operator_kernel}
    For $s>\frac{1}{2}$ and $l \in (-\frac{3}{2},-\frac{1}{2})$ the dimension of the kernel of $P_\beta(0)$ on $\Bar{H}_\mathrm{b}^{s,l}(X_\beta)$ is independent of $\beta$. That is, for all $\beta \in (-\pi,\pi)$:
    $$\dim(\ker(P_\beta(0))) = \dim(\ker(P_0(0))).$$
\end{lemma}
\begin{proof}
    The proof is similar to that of Lemma \ref{prop_qnm_scaling_independent}. We prove that every $\beta_1 \in (-\pi,\pi)$ has a neighborhood $(\beta_1-\delta,\beta_1+\delta)$, so that $\dim(\ker(P_{\beta_1}(0))) = \dim(\ker(P_{\beta_2}(0)))$ for all $\beta_2 \in (\beta_1-\delta,\beta_1+\delta)$. The result follows by covering $[0,\beta]$ or $[\beta,0]$ by such neighborhoods.

    Denoting $\Omega = \{x\in\R^3 \,\,|\,\, |x|>R_0\}$ and the complex scaled contours $\Gamma_{\beta_1} = F_{\beta_1}(\Omega)$, $\Gamma_{\beta_2} = F_{\beta_2}(\Omega)$, we define the contours $\Gamma^{\beta_1,\beta_2}_{R}$ and $\Gamma^{\beta_1,\beta_2}_{R,s}$ as in the proof of Lemma \ref{prop_qnm_scaling_independent}. That is, $\Gamma^{\beta_1,\beta_2}_{R}$ coincides with $\Gamma_{\beta_2}$ in $\{R \leq |x|\leq 2R\}$ and with $\Gamma_{\beta_1}$ outside of $\{\frac{R}{2} \leq |x| \leq 4R\}$ and the family of contours $s \in [0,1] \to \Gamma^{\beta_1,\beta_2}_{R,s}$ interpolates between $\Gamma_{\beta_1}$ and $\Gamma^{\beta_1,\beta_2}_{R}$. Taking $\beta_1,\beta_2 \in [0,\pi)$ or $\beta_1,\beta_2 \in (-\pi,0]$, by Lemma \ref{analytic_coefficients} the Kerr zero energy operator $P(0)$ extends to a differential operator with analytic coefficients on an open set $U \subset \C^3$ including $\Gamma^{\beta_1,\beta_2}_{R,s}$ for all $s \in [0,1]$. Furthermore, $P(0)\bigr|_{\Gamma^{\beta_1,\beta_2}_{R,s}}$ is elliptic for all $s$.

    Let now $u_1 \in \Bar{H}_\mathrm{b}^{s,l}(X_{\beta_1})$ satisfy $P_{\beta_1}u_1 = 0$. By elliptic regularity, $u_1$ is smooth on $\Gamma_{\beta_1}$. Applying Lemma \ref{deformation_lemma}, we obtain an analytic continuation of $u_1$ to a neighborhood of $\bigcup_{s\in[0,1]}\Gamma^{\beta_1,\beta_2}_{R,s}$. In particular, restricting to the contour $\Gamma^{\beta_1,\beta_2}_R = \Gamma^{\beta_1,\beta_2}_{R,1}$ gives an element $u^{\beta_1,\beta_2}_{R} \in C^\infty(\Gamma^{\beta_1,\beta_2}_{R})$ satisfying $P^{\beta_1,\beta_2}_R(0)u^{\beta_1,\beta_2}_{R} = 0$, where $P^{\beta_1,\beta_2}_R(0) = P(0)\bigr|_{\Gamma^{\beta_1,\beta_2}_R}$. By patching together these solutions $u^{\beta_1,\beta_2}_{R}$ as in \eqref{u_2_construction}, we obtain an element $u_2 \in \Bar{C}^\infty(X_{\beta_2})$ satisfying $P_{\beta_2}(\sigma)u_2 = 0$.

    It remains to show that $u_2 \in \Bar{H}_\mathrm{b}^{s,l}(X_{\beta_2})$. As in the proof of Lemma \ref{prop_qnm_scaling_independent}, this will follow by estimating the $\Bar{H}_\mathrm{b}^{s,l}$-norm of $u_2$ in terms of the $\Bar{H}_\mathrm{b}^{s,l}$-norm of $u_1$. To this end, note that the Laplacian on $\R^3$ is invertible as an operator
    $$\Delta: H_\mathrm{b}^{s,l}(\R^3) \to H_\mathrm{b}^{s-2,l+2}(\R^3)$$
    between weighted b-Sobolev spaces with $l \in (-\frac{3}{2},-\frac{1}{2})$, see for instance \cite[Theorem 3.1]{hintz_geometric_singular_analysis_notes}. Thus, for all $v \in H_\mathrm{b}^{s,l}(\R^3)$, we have
    \begin{equation}
    \label{b_laplacian_estimate}
        \|v\|_{H_\mathrm{b}^{s,l}(\R^3)} \leq C\|\Delta v\|_{H_\mathrm{b}^{s-2,l+2}(\R^3)}.
    \end{equation}
    
    Let now $\chi \in C_c^\infty(\R)$ satisfy $\supp(\chi) \subset (\tfrac{1}{4},8)$ and $\chi = 1$ on $[\tfrac{1}{2},4]$. Define
    $$\chi_R \in C_c^\infty(\R^3), \quad \chi_R(x) = \chi\Bigl(\frac{|x|}{R}\Bigr).$$
    Since $|x| < 8R$ on the support of $\chi_R$, we have
    $$\sup_{x\in\R^3}|r^j\partial_r^j\chi_R(x)| = \sup_{x\in\R^3}\Bigl|\frac{r^j}{R^j}\chi^{(j)}\Bigl(\frac{|x|}{R}\Bigr)\Bigr| \leq 8^j\sup_{t\in\R} |\chi^{(j)}(t)|.$$
    Thus, for any fixed $j\in\N$, we can bound the supremum norm of up to $j$ b-derivatives of $\chi_R$ by a constant independent of $R$.

    Applying the estimate \eqref{b_laplacian_estimate} to $\chi_R u$, we find for all $u \in C^\infty(\R^3)$ and with constants $C$ independent of $R$:
    \begin{equation}
    \label{b_laplacian_split_estimate}
    \begin{split}
        \|u\|_{H_\mathrm{b}^{s,l}(\{\frac{R}{2}<|x|<4R\})} &\leq \|\chi_R u\|_{H_\mathrm{b}^{s,l}(\R^3)} \leq C\|\Delta\chi_R u\|_{H_\mathrm{b}^{s-2,l+2}(\R^3)} \\
        &\leq C\bigl(\|\chi_R \Delta u\|_{H_\mathrm{b}^{s-2,l+2}(\R^3)} + \|[\Delta,\chi_R] u\|_{H_\mathrm{b}^{s-2,l+2}(\R^3)}\bigr) \\
        &\leq C\bigl(\|e^{-2i\beta_1}\Delta u\|_{H_\mathrm{b}^{s-2,l+2}(\{\frac{R}{4}<|x|<8R\})} + \|u\|_{H_\mathrm{b}^{s,l}(\{\frac{R}{4}<|x|<\frac{R}{2}\}\cup\{4R<|x|<8R\})}\bigr),
    \end{split}
    \end{equation}
    where we used that
    $$[\Delta,\chi_R] = -r^{-2}\bigl(2(r\partial_r\chi_R)r\partial_r + (r\partial_r)^2\chi_R + r\partial_r\chi_R\bigr),$$
    and by the observation above multiplication by $\chi_R$ and its b-derivatives define operators on $H_\mathrm{b}^{s,l}(\R^3)$ that are bounded uniformly in $R$. By the notation we mean
    $$\|u\|_{H_\mathrm{b}^{k,l}(\{\frac{R}{2}<|x|<4R\})} = \sum_{j + |\alpha|\leq k}\|r^l(r\partial_r)^j\partial_\omega^\alpha u\|_{L^2({\frac{R}{2}<|x|<4R\})}}$$
    and similarly for the other b-Sobolev norms on open subsets of $\R^3$.

    In local coordinates, we can view $P^{\beta_1,\beta_2}_R(0)$ as an operator on $\Omega \subset \R^3$. For any $\varepsilon>0$ we can choose $|\beta_1-\beta_2|$ small enough, so that
    \begin{equation}
    \label{deformation-laplacian}
        \|(P^{\beta_1,\beta_2}_R(0) - e^{-2i\beta_1}\Delta)u\|_{H_\mathrm{b}^{s-2,l+2}(\{\frac{R}{4}<|x|<8R\})} \leq \varepsilon \|u\|_{H_\mathrm{b}^{s,l}(\{\frac{R}{4}<|x|<8R\})}
    \end{equation}
    for all $R$ large enough. Indeed, denoting the phase function of the deformation, see \eqref{phase_function_deformation}, by $\phi_R(r) = \phi^{\beta_1,\beta_2}_{R,1}(r)$ for simplicity, and with $f_R(r) = e^{i\phi_R(r)}r$, we have
    \begin{equation*}
    \begin{split}
        P^{\beta_1,\beta_2}_R(0) - e^{-2i\beta_1}\Delta &= \rkerr(f_R)^{-2}\Bigl(\frac{\rkerr(f_R)^2}{e^{2i\beta_1}r^2} - 1\Bigr)\bigl((r\partial_r)^2 + r\partial_r + \Delta_{\Sph^2}\bigr) \\
        &+ \rkerr(f_R)^{-2}\Bigl((r\partial_r)^2 - \Bigl(\frac{f_R}{f_R'r}r\partial_r\Bigr)^2 + r\partial_r - \frac{f_R}{f_R'r}r\partial_r\Bigr) \\
        &- \rkerr(f_R)^{-2}\Bigl(\frac{2a}{f_R'r}r\partial_r\partial_{\varphi_*} - \frac{2m}{f_R'r}r\partial_r\frac{f_R}{r}r\partial_r + \frac{a}{f_R'r}r\partial_r\frac{1}{f_R'r}r\partial_r\Bigr).
    \end{split}
    \end{equation*}
    The last term can be controlled on $H_\mathrm{b}^{s,l}(\{\frac{R}{4}<|x|<8R\})$ by choosing $R$ large enough. For the first two terms, notice that
    \begin{align*}
        \frac{\rkerr(f_R)^2}{e^{2i\beta_1}r^2} - 1 = e^{2i\phi_R(r) - 2i\beta_1} + \frac{a^2\cos^2(\theta)}{e^{2i\beta_1}r^2}, \qquad 1 - \frac{f_R(r)}{f_R'(r)r} &= \frac{ir\phi_R'(r)}{1+ir\phi_R'(r)}.
    \end{align*}
    Now $|\phi_R(r) - \beta_1| \leq |\beta_2 - \beta_1|$ for all $R$, and by \eqref{phase_function_deformation} we have
    $$r\partial_r\phi_R(r) = (\beta_2 - \beta_1)\Bigl(\frac{r}{R}\Tilde{\chi}'\Bigl(\frac{r}{R}\Bigr)\psi(\log(r)) + \Tilde{\chi}\Bigl(\frac{r}{R}\Bigr)\psi'(\log(r))\Bigr).$$
    Thus, the supremum norm on $(\tfrac{R}{4},8R)$ of the functions $(\phi_R-\beta_1)$ and $r\partial_r\phi_R$ can be controlled uniformly in $R$ by choosing $|\beta_2-\beta_1|$ small enough. The same is true of higher order b-derivatives of $\phi_R$. The claim in \eqref{deformation-laplacian} follows.

    Inserting \eqref{deformation-laplacian} in the estimate \eqref{b_laplacian_split_estimate}, using
    $$\|u\|_{H_\mathrm{b}^{s,l}(\{\frac{R}{4}<|x|<8R\})} = \|u\|_{H_\mathrm{b}^{s,l}(\{\frac{R}{2}<|x|<4R\})} + \|u\|_{H_\mathrm{b}^{s,l}(\{\frac{R}{4}<|x|<\frac{R}{2}\}\cup\{4R<|x|<8R\})}$$
    and absorbing one of the terms into the left hand side, we find
    $$\|u\|_{H_\mathrm{b}^{s,l}(\{\frac{R}{2}<|x|<4R\})} \leq C\bigl(\|P^{\beta_1,\beta_2}_R(0) u\|_{H_\mathrm{b}^{s-2,l+2}(\{\frac{R}{4}<|x|<8R\})} + \|u\|_{H_\mathrm{b}^{s,l}(\{\frac{R}{4}<|x|<\frac{R}{2}\}\cup\{4R<|x|<8R\})}\bigr).$$
    Applying this estimate to $u^{\beta_1,\beta_2}_{R}$, the solution on the contour $\Gamma^{\beta_1,\beta_2}_{R}$ obtained by analytic continuation, we obtain
    $$\|u^{\beta_1,\beta_2}_{R}\|_{H_\mathrm{b}^{s,l}(\{\frac{R}{2}<|x|<4R\})} \leq C\|u^{\beta_1,\beta_2}_{R}\|_{H_\mathrm{b}^{s,l}(\{\frac{R}{4}<|x|<\frac{R}{2}\}\cup\{4R<|x|<8R\})}.$$
    Since $u^{\beta_1,\beta_2}_{R} = u_1$ in $\{\frac{R}{4}<|x|<\frac{R}{2}\}\cup\{4R<|x|<8R\})$ and $u^{\beta_1,\beta_2}_{R} = u_2$ in $\{R<|x|<2R\}$, we finally find
    $$\|u_2\|_{H_\mathrm{b}^{s,l}(\{R<|x|<2R\})} \leq C\|u_1\|_{H_\mathrm{b}^{s,l}(\{\frac{R}{4}<|x|<\frac{R}{2}\}\cup\{4R<|x|<8R\})}$$
    with $C$ independent of $R$. Summing over $R = 2^jR_0$ shows that $u_2 \in \Bar{H}_\mathrm{b}^{s,l}(X_{\beta_2})$.
\end{proof}

\subsection{Uniform estimates near zero energy}

We are now ready to prove the main estimate of this section \eqref{low_energy_estimate}. The proof is based on the normal operator estimates of the previous subsection. These allow us to improve the error term in \eqref{sc-b_interior_estimate} by slightly decreasing the weights at both zf and tf. This leads to an error that decays as $\sigma \to 0$, giving invertibility of $P_\beta(\sigma)$ for $|\sigma|$ small enough. The proof is close in spirit to that of \cite[Proposition 3.21]{hintz_mode_stability}.

\begin{prop}
\label{prop_low_energy_estimate}
    Let $\beta \in (-\pi,\pi)$ and $\sigma \in \Sigma_{\beta}^{c,\delta}$ as in \eqref{parameter_space} for some $c,\delta > 0$. Let further ${s,q,l,w \in \R}$ with $s > \frac{1}{2} - \alpha\Im(\sigma)$ for all $\sigma \in \Sigma_{\beta}^{c,\delta}$ and $w-l \in (\frac{1}{2},\frac{3}{2})$. Let $\epsilon > 0$ be such that $w-l-\epsilon \in (\frac{1}{2},\frac{3}{2})$ and take $s_0\in (\frac{1}{2},s)$. Then for all $u\in \Bar{H}_{\emph{\scb},\sigma}^{s,q,l,w}(X_\beta)$ with $P_\beta(\sigma)u \in \Bar{H}_{\emph{\scb},\sigma}^{s-1,q,l+2,w}(X_\beta)$ the following estimate holds uniformly for $\sigma \in \Sigma_{\beta}^{c,\delta}$:
    \begin{equation}
    \label{low_energy_estimate}
        \|u\|_{\Bar{H}_{\emph{\scb},\sigma}^{s,q,l,w}(X_\beta)} \leq C\bigl(\|P_\beta(\sigma)u\|_{\Bar{H}_{\emph{\scb},\sigma}^{s-1,q,l+2,w}(X_\beta)} + \|u\|_{\Bar{H}_{\emph{\scb},\sigma}^{s_0,-N,l-\epsilon,w-\epsilon}(X_\beta)}\bigr).
    \end{equation}
\end{prop}
\begin{proof}
    We start from the estimate in Lemma \ref{lemma_sc-b_elliptic_estimate} and improve the error term at zf and tf using the estimates for the normal operators proved in the previous subsection. In order to apply the estimate for the zero energy operator, we use regularity $s_0\in (\frac{1}{2},s)$ for the error term in \eqref{sc-b_interior_estimate}. We begin at the zero face. Thus, let $\chi \in C^\infty(X_\scb)$ be identically $1$ in a neighborhood of zf and have support away from scf (e.g. $\chi = \Tilde{\chi}(\frac{|\sigma|}{\bdf})$ with $\Tilde{\chi} \in C_c^\infty([0,1))$ identically $1$ near $0$). We have
    $$\|u\|_{\Bar{H}_{\scb,\sigma}^{s_0,-N,l,w}(X_\beta)} \leq \|\chi u\|_{\Bar{H}_{\scb,\sigma}^{s_0,-N,l,w}(X_\beta)} + C\|(1-\chi)u\|_{\Bar{H}_{\scb,\sigma}^{s_0,-N,l,-N}(X_\beta)},$$
    since the second term has support away from zf. For the first term we use the equivalence of norms \eqref{zf_norm_equivalence} to estimate
    $$\|\chi u\|_{\Bar{H}_{\scb,\sigma}^{s_0,-N,l,w}(X_\beta)} \leq C|\sigma|^{-w}\|\chi u\|_{\Bar{H}_{b}^{s_0,l-w}(X_\beta)}$$
    uniformly in $\sigma$. Note that $l-w \in (-\frac{3}{2},-\frac{1}{2})$ by assumption. Thus, we can apply Lemma \ref{lemma_zero_energy_estimate} to obtain
    \begin{equation*}
    \begin{split}
        \|\chi u\|_{\Bar{H}_{b}^{s_0,l-w}(X_\beta)} &\leq C\|P_\beta(0)\chi u\|_{\Bar{H}_{b}^{s_0-1,l-w+2}(X_\beta)} \\
        &\leq C\bigl(\|\chi P_\beta(0)u\|_{\Bar{H}_{b}^{s_0-1,l-w+2}(X_\beta)} + \|[P_\beta(0),\chi] u\|_{\Bar{H}_{b}^{s_0-1,l-w+2,}(X_\beta)},\bigr).
    \end{split}
    \end{equation*}
    Employing \eqref{zf_norm_equivalence} again to express this estimate in terms of the scattering-b-transition norms, we find
    \begin{equation*}
        \|\chi u\|_{\Bar{H}_{\scb,\sigma}^{s_0,-N,l,w}(X_\beta)} \leq C\bigl(\|\chi P_\beta(0)u\|_{\Bar{H}_{\scb,\sigma}^{s_0-1,-N,l+2,w}(X_\beta)} + \|u\|_{\Bar{H}_{\scb,\sigma}^{s_0,-N,l,-N}(X_\beta)}\bigr),
    \end{equation*}
    where we used that the commutator vanishes in a neighborhood of zf and hence
    $$[P_\beta(0),\chi] \in \mathrm{Diff}_\scb^{1,0,-2,-N}(X_\beta)$$
    for all $N$.
    As the zf-normal operator, $P_\beta(0)$ characterizes $P_\beta(\sigma)$ modulo $\bdf_\zf$. In fact, we have
    $$\chi P_\beta(\sigma) - \chi P_\beta(0) \in \mathrm{Diff}_\scb^{1,0,-2,-1}(X_\beta),$$
    and thus
    $$\|\chi P_\beta(0)u\|_{\Bar{H}_{\scb,\sigma}^{s_0-1,-N,l+2,w}(X_\beta)} \leq \|\chi P_\beta(\sigma)u\|_{\Bar{H}_{\scb,\sigma}^{s_0-1,-N,l+2,w}(X_\beta)} + C\|u\|_{\Bar{H}_{\scb,\sigma}^{s_0,-N,l,w-1}(X_\beta)}.$$
    Altogether, we have improved the error term in \eqref{sc-b_interior_estimate} by one order of decay at zf:
    \begin{equation}
    \label{zf_improvement}
        \|u\|_{\Bar{H}_{\scb,\sigma}^{s,q,l,w}(X_\beta)} \leq C\bigl(\|P_\beta(\sigma)u\|_{\Bar{H}_{\scb,\sigma}^{s-1,q,l+2,w}(X_\beta)} + \|u\|_{\Bar{H}_{\scb,\sigma}^{s_0,-N,l,w-1}(X_\beta)}\bigr).
    \end{equation}

    We now proceed in a similar fashion at tf. Let $\psi \in C^\infty(X_\scb)$ have support in a small neighborhood of tf with $\psi=1$ near tf (e.g. $\psi = \Tilde{\psi}(\bdf + |\sigma|)$ with $\Tilde{\psi} \in C^\infty_c[0,\varepsilon)$ identically $1$ near $0$). Then with $\epsilon \in (0,1)$ as in the statement of the Proposition, we have
    \begin{equation*}
    \begin{split}
        \|u\|_{\Bar{H}_{\scb,\sigma}^{s_0,-N,l,w-1}(X_\beta)} &\leq \|u\|_{\Bar{H}_{\scb,\sigma}^{s_0,-N,l,w-\epsilon}(X_\beta)} \\
        &\leq \|\psi u\|_{\Bar{H}_{\scb,\sigma}^{s_0,-N,l,w-\epsilon}(X_\beta)} + C\|(1-\psi)u\|_{\Bar{H}_{\scb,\sigma}^{s_0,-N,-N,w-\epsilon}(X_\beta)}.
    \end{split}
    \end{equation*}
    The equivalence of norms in \eqref{tf_norm_equivalence} gives
    $$\|\psi u\|_{\Bar{H}_{\scb,\sigma}^{s_0,-N,l,w-\epsilon}(X_\beta)} \leq C|\sigma|^{-l-\frac{3}{2}}\|\psi u\|_{H_{\mathrm{sc,b}}^{s_0,-N,w-l-\epsilon}(\tf)}$$
    uniformly in $\sigma$. Since $w-l-\epsilon \in (\frac{1}{2},\frac{3}{2})$ by assumption, we can apply the estimate on the tf-normal operator from Lemma \ref{lemma_tf_operator} to obtain
    \begin{equation*}
    \begin{split}
        \|\psi u\|_{H_{\mathrm{sc,b}}^{s_0,-N,w-l-\epsilon}(\tf)} &\leq C\|N_{\tf}\bigl(\bdf_\tf^{-2}P_\beta(\sigma)\bigr)\psi u\|_{H_{\mathrm{sc,b}}^{s_0-2,-N,w-l-\epsilon}(\tf)} \\
        &\leq C\bigl(\|\psi N_{\tf}\bigl(\bdf_\tf^{-2}P_\beta(\sigma)\bigr)u\|_{H_{\mathrm{sc,b}}^{s_0-2,-N,w-l-\epsilon}(\tf)} \\
        &\quad+ \|\bigl[N_{\tf}\bigl(\bdf_\tf^{-2}P_\beta(\sigma)\bigr),\psi\bigr] u\|_{H_{\mathrm{sc,b}}^{s_0-2,-N,w-l-\epsilon}(\tf)}\bigr).
    \end{split}
    \end{equation*}
    Rewriting this in terms of the scattering-b-transition norms by \eqref{tf_norm_equivalence} and using
    \begin{align*}
        \psi\bdf_\tf^{-2}P_\beta(\sigma) - \psi N_{\tf}\bigl(\bdf_\tf^{-2}P_\beta(\sigma)\bigr) &\in \mathrm{Diff}_\scb^{2,0,-1,0}(X_\beta), \\
        \bigl[N_{\tf}\bigl(\bdf_\tf^{-2}P_\beta(\sigma)\bigr),\psi\bigr] &\in \mathrm{Diff}_\scb^{1,0,-N,0}(X_\beta) \quad \forall N,
    \end{align*}
    we find
    \begin{equation*}
        \|\psi u\|_{\Bar{H}_{\scb,\sigma}^{s_0,-N,l,w-\epsilon}(X_\beta)} \leq C\bigl(\|\psi P_\beta(\sigma) u\|_{\Bar{H}_{\scb,\sigma}^{s_0-1,-N,l+2,w-\epsilon}(X_\beta)} + \|\psi u\|_{\Bar{H}_{\scb,\sigma}^{s_0,-N,l-1,w-\epsilon}(X_\beta)}\bigr).
    \end{equation*}
    Altogether, we have improved the error term in \eqref{zf_improvement} by one order of decay at tf, at the cost of increasing the weight at zf from $w-1$ to $w-\epsilon$. Finally, increasing the weight at tf from $l-1$ to $l-\epsilon$ for convenience gives the result.
\end{proof}

From the uniform low energy estimate of Proposition \ref{prop_low_energy_estimate} it now follows easily that quasinormal modes cannot accumulate at zero energy.

\begin{proof}[Proof of Theorem \ref{thm_low_energy}]
    Notice that $\bdf_\tf^\epsilon\bdf_\zf^\epsilon = |\sigma|^\epsilon$. So the estimate in Proposition \ref{prop_low_energy_estimate} can be rewritten as follows
    $$\|u\|_{\Bar{H}_{\scb,\sigma}^{s,q,l,w}(X_\beta)} \leq C\bigl(\|P_\beta(\sigma)u\|_{\Bar{H}_{\scb,\sigma}^{s-1,q,l+2,w}(X_\beta)} + |\sigma|^\epsilon\|u\|_{\Bar{H}_{\scb,\sigma}^{s_0,-N,l,w}(X_\beta)}\bigr).$$
    Choosing $|\sigma|$ small enough, we can absorb the final term into the left hand side. Thus, for all $\beta \in (-\pi,\pi)$ and all $\delta > 0$ small, there exists some $c=c_{\beta,\delta}$ such that
    \begin{equation}
    \label{low_energy_no_kernel}
        \|u\|_{\Bar{H}_{\scb,\sigma}^{s,q,l,w}(X_\beta)} \leq C\|P_\beta(\sigma)u\|_{\Bar{H}_{\scb,\sigma}^{s-1,q,l+2,w}(X_\beta)}
    \end{equation}
    for $\sigma \in \Sigma_{\beta}^{c,\delta} = \{\sigma \in \C,\, |\sigma|\leq c,\, -\beta + \delta \leq \arg(\sigma) \leq \pi-\beta-\delta\}$.

    We apply this to elements of $\Bar{H}^s(X_\beta)$ lying in the kernel of $P_\beta(\sigma)$ by noting that $\Bar{H}^s(X_\beta)$ is included in a scattering-b-transition Sobolev space with appropriate weights. Indeed, for all $k\in\N$, we have
    $$\Bar{H}^k(X_\beta) \subset \Bar{H}_{\scb,\sigma}^{k,0,-k,0}(X_\beta).$$
    This follows from
    \begin{equation*}
    \begin{split}
        \sum_{j+|\alpha|\leq k} \bigl\|(\bdf+|\sigma|)^k \bigl(\frac{\bdf}{\bdf+|\sigma|}\partial_\omega\bigr)^\alpha \bigl(\frac{\bdf}{\bdf+|\sigma|}\bdf\partial_\bdf\bigr)^j u \bigl\|^2_{L^2(X_\beta)} \leq C \sum_{j+|\alpha|\leq k} \bigl\| \bigl(\bdf\partial_\omega\bigr)^\alpha \bigl(\bdf^2\partial_\bdf\bigr)^j u \bigl\|^2_{L^2(X_\beta)}.
    \end{split}
    \end{equation*}
    In particular, we have $\Bar{H}^1(X_\beta) \subset \Bar{H}_{\scb,\sigma}^{1,0,-1,0}(X_\beta)$. This space satisfies the requirements of Proposition \ref{prop_low_energy_estimate} as long as $c<\frac{1}{2\alpha}$. Thus, \eqref{low_energy_no_kernel} shows that the kernel of $P_\beta(\sigma)$ in $\Bar{H}^1(X_\beta)$ is trivial for all $\sigma \in \Sigma_{\beta}^{c,\delta}$. Covering the set in the statement of Theorem \ref{thm_low_energy} by a finite number of sets of the form $\Sigma_{\beta}^{c,\delta}$ finishes the proof.
\end{proof}

\newpage
\printbibliography

\end{document}